\documentclass[letterpaper, 10pt, twoside, onecolumn, final, openright,
chapterprefix=true 
]{scrbook}
\usepackage[english]{babel}
\usepackage[utf8x]{inputenc}
\usepackage{listings}
\usepackage{hyperref}
\hypersetup{%
	colorlinks = true,
	linkcolor  = black,
	citecolor  = blue
}
\usepackage{caption}
\usepackage{lscape}
\usepackage{subfigure}
\usepackage{graphicx}
\usepackage{relsize}
\usepackage{amsmath}
\usepackage{mathrsfs}
\usepackage{graphicx}
\usepackage{setspace}
\usepackage{mathtools}
\usepackage{float}
\usepackage{amsthm}
\usepackage{amssymb}
\usepackage{exscale}
\usepackage{thmtools}
\usepackage[dvipsnames,usenames]{xcolor}
\usepackage{microtype}
\usepackage{enumerate}
\usepackage{stmaryrd}
\usepackage{dsfont}
\usepackage{tikz}
\usetikzlibrary{calc}
\usepackage{bm}
\usepackage{wrapfig}
\usepackage{booktabs}
\usepackage[final]{listofsymbols}
\usepackage[symbols,nogroupskip,sort=none]{glossaries-extra}
\usepackage{float}
\usepackage{nameref}
\usepackage[capitalize,nameinlink]{cleveref}
\usepackage{nomencl}
\makenomenclature
\usepackage{titlesec}
\usepackage{bbm}
\usepackage{accents}
\usepackage{lipsum}
\usepackage{ stmaryrd }
\usepackage{ marvosym }
\usepackage{tikz}
\usepackage{natbib}
\definecolor{cardinalred}{RGB}{140,21,21}

\crefformat{definition}{\textcolor{Blue}{Definition}~#2{\color{Blue}\textcolor{gray}{\S\phantom{.}#1}}#3}
\crefformat{thm}{\textcolor{cardinalred}{Theorem}~#2{\color{cardinalred}\textcolor{gray}{\S\phantom{.}#1}}#3}
\crefformat{co}{\textcolor{cardinalred}{Corollary}~#2{\color{cardinalred}\textcolor{gray}{\S\phantom{.}#1}}#3}
\crefformat{pr}{\textcolor{cardinalred}{Proposition}~#2{\color{cardinalred}\textcolor{gray}{\S\phantom{.}#1}}#3}
\crefformat{lem}{\textcolor{cardinalred}{Lemma}~#2{\color{cardinalred}\textcolor{gray}{\S\phantom{.}#1}}#3}
\crefformat{example}{\textcolor{OliveGreen}{Example}~#2{\color{OliveGreen}\textcolor{gray}{\S\phantom{.}#1}}#3}
\crefformat{notation}{\textcolor{oldgold}{Notation}~#2{\color{oldgold}\textcolor{gray}{\S\phantom{.}#1}}#3}
\crefformat{rem}{\textcolor{Brown}{Remark}~#2{\color{Brown}\textcolor{gray}{\S\phantom{.}#1}}#3}
\crefformat{appendix}{\textcolor{Black}{Appendix}~#2{\color{Blue}\textcolor{gray}{\S\phantom{.}#1}}#3}

\newcommand*{\ob}[1]{%
	\makebox[0pt][l]{%
		\kern.3em
		\raisebox{1.8ex}{\relsize{-3}{$\circ$}}
	}%
	#1%
}

\newcommand\buildcircleland[1]{%
	\begin{tikzpicture}[anchor=base,baseline=-1pt, inner sep=0, outer sep=0,align=center]
		\node[draw,circle,align=center] (X) {$\sim$};
	\end{tikzpicture}%
}

\makeatletter
\newsavebox\myboxA
\newsavebox\myboxB
\newlength\mylenA

\newcommand*\xoverline[2][0.75]{%
	\sbox{\myboxA}{$\m@th#2$}%
	\setbox\myboxB\null
	\ht\myboxB=\ht\myboxA%
	\dp\myboxB=\dp\myboxA%
	\wd\myboxB=#1\wd\myboxA
	\sbox\myboxB{$\m@th\overline{\copy\myboxB}$}
	\setlength\mylenA{\the\wd\myboxA}
	\addtolength\mylenA{-\the\wd\myboxB}%
	\ifdim\wd\myboxB<\wd\myboxA%
	\rlap{\hskip 0.5\mylenA\usebox\myboxB}{\usebox\myboxA}%
	\else
	\hskip -0.5\mylenA\rlap{\usebox\myboxA}{\hskip 0.5\mylenA\usebox\myboxB}%
	\fi}
\makeatother

\makeatletter
\providecommand*{\cupdot}{%
	\mathbin{%
		\mathpalette\@cupdot{}%
	}%
}
\newcommand*{\@cupdot}[2]{%
	\ooalign{%
		$\m@th#1\cup$\cr
		\hidewidth$\m@th#1\cdot$\hidewidth
	}%
}
\makeatother

\makeatletter
\providerobustcmd*{\bigcupdot}{%
	\mathop{%
		\mathpalette\bigop@dot\bigcup
	}%
}
\newrobustcmd*{\bigop@dot}[2]{%
	\setbox0=\hbox{$\m@th#1#2$}%
	\vbox{%
		\lineskiplimit=\maxdimen
		\lineskip=-0.7\dimexpr\ht0+\dp0\relax
		\ialign{%
			\hfil##\hfil\cr
			$\m@th\cdot$\cr
			\box0\cr
		}%
	}%
}
\makeatother

\definecolor{bananayellow}{rgb}{1.0, 0.88, 0.21}
\definecolor{citrine}{rgb}{0.89, 0.82, 0.04}
\definecolor{goldenyellow}{rgb}{1.0, 0.87, 0.0}
\definecolor{jonquil}{rgb}{0.98, 0.85, 0.37}
\definecolor{mikadoyellow}{rgb}{1.0, 0.77, 0.05}
\definecolor{mustard}{rgb}{1.0, 0.86, 0.35}
\definecolor{oldgold}{rgb}{0.81, 0.71, 0.23}

\declaretheoremstyle[
qed=\raisebox{1pt}{\textcolor{cardinalred}{\openbox}},
headfont = \color{cardinalred},
headformat = \NAME{} \textcolor{gray}{\NUMBER}:\NOTE,
]{theoremstyles}
\declaretheorem[
style=theoremstyles,
name=\textsc{Theorem \textcolor{gray}{\S}},
numberwithin=section,
]{thm}

\declaretheoremstyle[
qed=\raisebox{1pt}{\textcolor{cardinalred}{\openbox}},
headfont = \color{cardinalred},
headformat = \NAME{} \textcolor{gray}{\NUMBER}:\NOTE,
]{Corollarystyles}
\declaretheorem[
style=Corollarystyles,
name=\textsc{Corollary \textcolor{gray}{\S}},
numberlike=thm,
]{co}

\declaretheoremstyle[
qed=\raisebox{1pt}{\textcolor{cardinalred}{\openbox}},
headfont = \color{cardinalred},
headformat = \NAME{} \textcolor{gray}{\NUMBER}:\NOTE,
]{Propostyles}
\declaretheorem[
style=Propostyles,
name=\textsc{Proposition \textcolor{gray}{\S}},
numberlike=thm,
]{pr}

\declaretheoremstyle[
qed=\raisebox{1pt}{\textcolor{cardinalred}{\openbox}},
headfont = \color{cardinalred},
headformat = \NAME{} \textcolor{gray}{\NUMBER}:\NOTE,
]{Lemmastyles}
\declaretheorem[
style=Lemmastyles,
name=\textsc{Lemma \textcolor{gray}{\S}},
numberlike=thm,
]{lem}

\declaretheoremstyle[
qed=\raisebox{1pt}{\textcolor{Blue}{\openbox}},
headfont = \protect\color{Blue},
headformat = \NAME{} \textcolor{gray}{\NUMBER}:\NOTE,
]{definitions}
\declaretheorem[
style=definitions,
name=\textsc{Definition \textcolor{gray}{\S}},
numberlike=thm,
]{definition}

\declaretheoremstyle[
qed=\raisebox{1pt}{\textcolor{OliveGreen}{\openbox}},
headfont = \color{OliveGreen},
headformat = \NAME{} \textcolor{gray}{\NUMBER}:\NOTE,
]{examples}

\declaretheoremstyle[
qed=\raisebox{1pt}{\textcolor{oldgold}{\openbox}},
headfont = \color{oldgold},
headformat = \NAME{} \textcolor{gray}{\NUMBER}:\NOTE,
]{notations}
\declaretheorem[
style=notations,
name=\textsc{Notation \textcolor{gray}{\S}},
numberlike=thm,
]{notation}

\declaretheoremstyle[
qed=\raisebox{1pt}{\textcolor{Brown}{\openbox}},
headfont = \color{Brown},
headformat = \NAME{} \textcolor{gray}{\NUMBER}:\NOTE,
]{remarks}
\declaretheorem[
style=remarks,
name=\textsc{Remark \textcolor{gray}{\S}},
numberlike=thm,
]{rem}

\declaretheoremstyle[
numbered=no,
notebraces={}{},
qed=\raisebox{1pt}{\textcolor{gray}{\openbox}},
headfont = \itshape\color{gray},
]{proofs}
\declaretheorem[
style=proofs,
name=Proof of,
]{proof}

\newtheorem{ass}{Assumption}[section]

\newcommand{\Hilbert}[1]{$(\mathbb{H},\langle\,\cdot\,|\,\cdot\,\rangle_{\mathbb{H}})$}

\definecolor{gray75}{gray}{0.75}
\newcommand{\hsp}{\hspace{10pt}}
\titleformat{\chapter}[hang]{\bfseries\scshape\huge\color{cardinalred}}{\text{Chapter \textcolor{gray}{\S\thechapter}}\hsp\textcolor{gray75}{|}\hsp}{0pt}{\bfseries\scshape\huge\color{cardinalred}}
\titleformat{\chapter}[hang]{\bfseries\scshape\huge\color{cardinalred}}{\text{Chapter \textcolor{gray}{\S\thechapter}}}{6pt}{\scshape\huge\color{cardinalred}}
\titleformat{\section}[block]{\Large\scshape\thesection\centering\scshape}{}{1em}{}
\titleformat{\subsection}[block]{\large\scshape\centering\scshape}{}{1em}{}

\DeclareFontFamily{OMX}{lmex}{}
\DeclareFontShape{OMX}{lmex}{m}{n}{<->lmex10}{}

\usepackage[a4paper,top=3cm,bottom=2cm,left=2cm,right=2cm,marginparwidth=1.75cm, head = 18.0pt]{geometry}

\usepackage{abstract}

\newcommand{\pRz}{\mathbb R_{>0}}
\newcommand{\LebpRz}{\lambda_{>0}}
\newcommand{\BorelpRz}{\mathscr{B}_{>0}}
\newcommand{\BorelIR}{\mathscr{B}}
\newcommand{\norm}[1]{\left\|#1\right\|}

\newcommand{\M}{\operatorname M}

\newcommand{\MelU}{\M_{U}}
\newcommand{\MelUrez}{\M_{U}^\dagger}
\newcommand{\MelUhat}{\widehat{\M}_{U}}
\newcommand{\MelUhatrez}{\widehat{\M}_{U}^\dagger}

\newcommand{\sigmaYhat}{\widehat{\sigma}_{Y}^2}

\newcommand{\MelX}{\M_{X}}
\newcommand{\MelXcheck}{\widecheck{\M}_{X}}
\newcommand{\MelXhat}{\widehat{\M}_{X}}
\newcommand{\MelY}{\M_{Y}}

\newcommand{\MelYhat}{\widehat{\M}_{Y}}

\newcommand{\MelZhat}{\widehat{\M}_{Z}}

\newcommand{\SXhat}{ \widehat{S}^{X}}
\newcommand{\SX}{ S^{X}}

\newcommand{\const}{\mathfrak c}
\newcommand{\Const}{\mathfrak C}

\newcommand{\IB}{\mathbb B}
\newcommand{\IC}{\mathbb C}

\newcommand{\IE}{\mathbb E}

\newcommand{\IL}{\mathbb L}

\newcommand{\IN}{\mathbb N}

\newcommand{\IP}{\mathbb P}
\newcommand{\IQ}{\mathbb Q}
\newcommand{\IR}{\mathbb R}
\newcommand{\IS}{\mathbb S}

\newcommand{\IV}{\mathbb V}
\newcommand{\IW}{\mathbb W}

\newcommand{\rmh}{\mathrm{h}}

\newcommand{\rmw}{\mathrm{w}}
\newcommand{\rmwbar}{\bar{\mathrm{w}}}

\newcommand{\rmwrez}{\rmw^\dagger}
\newcommand{\rmv}{\mathrm{v}}
\newcommand{\rmvU}{\mathrm{v}_U}
\newcommand{\rmvhat}{\widehat{\mathrm{v}}}
\newcommand{\rmd}{\mathrm{d}}
\newcommand{\rmt}{\mathrm{t}}
\newcommand{\rmx}{\mathrm{x}}

\newcommand{\frako}{\mathfrak{o}}

\newcommand{\frakM}{\mathfrak{M}}

\newcommand{\khat}{\hat{k}}

\newcommand{\Rz}{\mathbb R}

\newcommand{\Nz}{\mathbb N}

\newcommand{\Pz}{\mathbb P}
\newcommand{\1}{\mathds 1}

\newcommand{\E}{\mathbb E}
\newcommand{\Var}{\mathbb V\mathrm{ar}}

\makeatletter
\newcommand{\ocon}{\mathbin{\mathpalette\make@circled *}}
\newcommand{\make@circled}[2]{%
	\ooalign{$\m@th#1\smallbigcirc{#1}$\cr\hidewidth$\m@th#1#2$\hidewidth\cr}%
}
\newcommand{\smallbigcirc}[1]{%
	\vcenter{\hbox{\scalebox{0.77778}{$\m@th#1\bigcirc$}}}%
}
\makeatother

\newcommand{\nset}[1]{{\left\llbracket #1\right\rrbracket }}

\DeclareMathOperator*{\argmin}{arg\,min}

\newcommand{\pen}{\operatorname{pen}}

\DeclareFontFamily{U}{mathx}{}
\DeclareFontShape{U}{mathx}{m}{n}{<-> mathx10}{}
\DeclareSymbolFont{mathx}{U}{mathx}{m}{n}
\DeclareMathAccent{\widehat}{0}{mathx}{"70}
\DeclareMathAccent{\widecheck}{0}{mathx}{"71}
\newcommand{\cK}[1][]{\mathcal{K}_{#1}}
\newcommand{\ko}[1][\circ]{k_{#1}}
\newcommand{\kn}[1][n]{\ko[{n}]}

\newcommand{\evM}{\mathfrak{M}}
\newcommand{\evO}{\mho}
\newcommand{\evA}{\mathfrak{A}}

\newcommand{\MelXtilde}{\widetilde{\M}_{X}}
\newcommand{\skalar}[1]{\big\langle#1\big\rangle}

\numberwithin{equation}{section}

\begin{document}
	\setcounter{chapter}{1}
	\counterwithout{section}{chapter}
	\thispagestyle{plain}
	\begin{center}
		\huge \textbf{Multiplicative deconvolution under unknown error distribution}
		
		\vspace{0.2cm}
		\large\textsc{Sergio Brenner Miguel}\footnote{M$\Lambda$THEM$\Lambda$TIKON, Im Neuenheimer Feld 205, D-69120 Heidelberg;  \href{mailto:brennermiguel@math.uni-heidelberg.de}{brennermiguel@math.uni-heidelberg.de}}
		\hspace{2cm}
		\large\textsc{Jan Johannes}\footnote{M$\Lambda$THEM$\Lambda$TIKON, Im Neuenheimer Feld 205, D-69120 Heidelberg;  \href{mailto:johannes@math.uni-heidelberg.de}{johannes@math.uni-heidelberg.de}}\\
			\vspace{0.2cm}
		\large\textsc{Maximilian Siebel}\footnote{M$\Lambda$THEM$\Lambda$TIKON, Im Neuenheimer Feld 205, D-69120 Heidelberg; \href{mailto:siebel@math.uni-heidelberg.de}{siebel@math.uni-heidelberg.de}}\\
		\end{center}
		\vspace{0.9cm}
		\normalsize{We consider a multiplicative deconvolution problem, in which the density $f$ or the survival function $\SX$ of a strictly positive random variable $X$ is estimated nonparametrically based on an i.i.d. sample from a noisy observation $Y = X\cdot U$ of $X$. The multiplicative measurement error $U$ is supposed to be independent of $X$. The objective of this work is to construct a fully data-driven estimation procedure when the error density $f^U$ is unknown. We assume that in addition to the i.i.d. sample from $Y$, we have at our disposal an additional i.i.d. sample drawn independently from the error distribution. The proposed estimation procedure combines the estimation of the Mellin transformation of the density $f$ and a regularisation of the inverse of the Mellin transform by a spectral cut-off. The  derived risk bounds and oracle-type inequalities cover both - the estimation of the density $f$ as well as the survival function $\SX$. The main issue addressed in this work is the data-driven choice of the cut-off parameter using a model selection approach. We discuss conditions under which the fully data-driven estimator can attain the oracle-risk up to a constant without any previous knowledge of the error distribution. We compute convergences rates under classical smoothness assumptions. We illustrate the estimation strategy by a simulation study with different choices of distributions.}

	\section{Introduction}\label{sec: 1}
In this work let $X$ and $U$ be strictly positive and independent random variables both admitting unknown densities $f=f^X$ and $f^U$ accordingly. We propose a data-driven estimation procedure for the density $f$ as well as the survival function $\SX$ of $X$ based on an independent and identically distributed (i.i.d.) sample of size $n$ from a multiplicative measurement model $Y = X\cdot U$ and an additional sample of size $m$ drawn independently from the unknown error density $f^U$.
In this situation $Y$ admits also a density, given by 
	\begin{equation}\label{eq:multiplicative_convolution}
			f^{Y}(y):=[f\ocon f^U](y):=\int_{\pRz}f(x)f^U(y/x)x^{-1}\mathrm{d}x.
		\end{equation}
Estimating $f$ from i.i.d. observations following the law of $f^Y$ is a statistical inverse problem called multiplicative deconvolution. Multiplicative censoring is introduced and studied in \cite{Vardi1989} and \cite{VardiZhang1992}. It corresponds to the particular multiplicative deconvolution problem with multiplicative error $U$ uniformly distributed on $[0, 1]$. \cite{EsKlaassenOudshoorn2000} motivate and explain multiplicative censoring in survival analysis. \cite{VardiZhang1992} and \cite{AsgharianWolfson2005} consider the estimation of the cumulative distribution function of $X$. Treating the model as an inverse problem \cite{AndersenHansen2001} study series expansion methods. The density estimation in a multiplicative censoring model is considered in \cite{BrunelComteGenon-Catalot2016} using a kernel estimator and a convolution power kernel estimator. \cite{ComteDion2016} analyse a projection density estimator with respect to the Laguerre basis assuming a uniform error distribution on an interval $[1-\alpha, 1+\alpha]$ for $\alpha \in (0, 1)$. A beta-distributed error $U$ is studied in \cite{BelomestnyComteGenon-Catalot2016}. The multiplicative measurement error model covers all those three variations of multiplicative censoring. It was considered by \cite{BelomestnyGoldenshluger2020} studying the point-wise density estimation and by \cite{Brenner-MiguelComteJohannes2023} casting point-wise estimation as the estimation of the value of a known linear functional. The global estimation of the density under multiplicative measurement errors is considered in \cite{Brenner-MiguelComteJohannes2021} using the Mellin transform and a spectral cut-off regularisation of its inverse to define an estimator for the unknown density $f$. In those three papers the key to the analysis of multiplicative deconvolution is the multiplication theorem, which for a density $f^Y = f \ocon f^U$ and their Mellin transformations $\M[f^Y]$, $\M[f]$ and $\M[f^U]$ (defined below) states $\M[f^Y ] = \M[f]\M[f^U]$. It is used in \cite{BelomestnyGoldenshluger2020} and \cite{Brenner-MiguelComteJohannes2021} to construct respectively a kernel density estimator and a spectral cut-off estimator of the density $f$, while the later serves for a plug-in estimator in \cite{Brenner-MiguelComteJohannes2023}.
\cite{Brenner-Miguel2022} studies the global density estimation under multiplicative measurement errors for multivariate random variables while the global estimation of the survival function can be found in \cite{Brenner-MiguelPhandoidaen2022}. For local and global multiplicative deconvolution \cite{BelomestnyGoldenshluger2020} and \cite{Brenner-MiguelComteJohannes2021} both comment on the naive approach to apply standard additive deconvolution methods to the $\log$-transformed data. Essentially, the additive deconvolution theorem for the $\log$-transformed data equals a special case of the multiplicative convolution theorem.\\ 
We now turn to multiplicative deconvolution with unknown error density, which is inspired by similar ideas for
additive deconvolution problems (see for instance \cite{Neumann1997} and \cite{Johannes2009}). Following the estimation strategy in \cite{Brenner-MiguelComteJohannes2021} and \cite{Brenner-MiguelPhandoidaen2022}, and borrowing ideas from the inverse problems community (see for instance \cite{EnglHankeNeubauer1996}), we define spectral cut-off estimators $\hat{f}_k$ and $\SXhat_k$ of $f$ and $\SX$, respectively, by replacing the unknown Mellin transformations of $f^Y$ and $f^U$ by empirical counterparts based on the two samples from $U$ and $Y$ and additional thresholding. The accuracy of the estimators $\hat{f}_k$ and $\SXhat_k$ are measured in terms of the global risk with respect to a weighted $\IL^2$-norm on the positive real line $\pRz$. We observe that both global risks can be embedded into a more general risk analysis, which we then study in detail. The proposed estimation strategy depends on a further tuning parameter $k$, which has to be chosen by the user. In case of a known error density \cite{Brenner-MiguelComteJohannes2021} and \cite{Brenner-MiguelPhandoidaen2022} propose
a data-driven choice of the tuning parameter $k$ by model selection
exploiting the theory of \cite{BarronBirgeMassart1999}, where we refer
to \cite{Massart2007} for an extensive overview. Our aim is to establish a fully data-driven estimation procedure for the density $f$ and the survival function $\SX$ when the error density is unknown and derive oracle-type upper risk bounds as well as convergences rates. A similar approach has been considered for additive deconvolution problems for instance in \cite{ComteLacour2011} and \cite{Johannes2013}. Regarding the two samples sizes $n$ and $m$, by comparing the oracle-type risk bounds in the cases of known and unknown error densities, we characterise the influence of the estimation of the error density on the quality of the estimation. Interestingly, in case of additive convolution on the circle and the real line \cite{Johannes2013} and \cite{ComteLacour2011} derive
respectively oracle-type inequalities with similar structure.\\
The paper is organised as follows. In \cref{sec: 2} we start with recalling the definition of the Mellin transform as well stating certain properties. Secondly, supposing the error density to be known we introduce the spectral cut-off estimators of $f$ and $\SX$ as respectively proposed by \cite{Brenner-MiguelComteJohannes2021} and \cite{Brenner-MiguelPhandoidaen2022}. We study their global risk with respect to weighted $\IL^2$-norms and discuss how to generalise this risk analysis. Afterwards, we investigate in oracle-type inequalities and minimax-optimal convergences rates under regularity assumptions on the Mellin transformations of $f$ and $f^U$, respectively. In \cref{sec: 3,sec: 4} we dismiss the knowledge of the error density. In \cref{sec: 3} we derive a general risk bound and oracle-type inequality. We devote \cref{sec: 4} in constructing the fully data-driven estimation procedure and eventually show a upper risk bound. The general theory applies in particular to the estimation of $f$ and $\SX$. In \cref{sec: 5} we illustrate the proposals of \cref{sec: 3} and \cref{sec: 4} by a simulation study with different choices of distributions.

	\section{Model assumptions and background}\label{sec: 2}
In the following paragraphs we introduce first the multiplicative
measurement model with known error distribution and we recall the
definition of the Mellin transform and its empirical counter part
as well as their properties. Secondly, still assuming the  error
distribution is known in advance we briefly present a data-driven density estimation
strategy due to \cite{Brenner-MiguelComteJohannes2021} which we extend
in the sequel to cover simultaneously the estimation of the density
 $f$ as well as the survival function $\SX$ of $X$. 
\subsection{The multiplicative measurement model}\label{subsec: 2.1}
Let $(\pRz,\BorelpRz,\LebpRz)$ denote the Lebesgue-measure space of
all positive real numbers $\pRz$ equipped with the restriction
$\LebpRz$ of the Lebesgue measure on the Borel $\sigma$-field
$\BorelpRz$. Assume that $X$ and $U$ are independent random variables, both taking values in $\pRz$ and
admitting both a (Lebesgue) density $f:=f^X$ and $f^U$, respectively.  The multiplicative measurement model
describes observations 
following the law of $Y := X\cdot U$. In this situation, $Y$ admits
a  density  $f^Y=f\ocon f^U$ given as multiplicative convolution of
$f$ and $f^U$, which can be computed explicitly as in \eqref{eq:multiplicative_convolution}.
For a detailed discussion of multiplicative convolution and its
properties as operator between function spaces we refer to
\cite{Brenner-Miguel2023}. However, assuming
the error distribution is known, we have in the sequel access to an
independent and identically
distributed (i.i.d.) sample $\{Y_i\}_{i\in\nset{n}}$ drawn from $f^Y$, where we have used
the shorthand notation $\nset{a}:=[1,a]\cap\IN$ for any
$a\in\Rz_{\geq1}$.
\subsection{The (empirical) Mellin transform}
In the subsequent, we keep the following
objects and notations in mind: Given a density function $\rmv$ defined
on $\pRz$, i.e. a Borel-measurable nonnegative function
$\rmv:\pRz\to\IR_{\geq0}$, let $\rmv\LebpRz$ denote the
$\sigma$-finite measure on $(\pRz,\BorelpRz)$ which is $\LebpRz$
absolutely continuous and admits the Radon-Nikodym derivative $\rmv$
with respect to $\LebpRz$. For $p\in[1,\infty]$ let
$\IL^p_+(\rmv):=\IL^p(\pRz,\BorelpRz,\rmv\LebpRz)$ denote the usual
complex Banach-space of all (equivalence classes of)
$\IL^p$-integrable complex-valued function with respect to the measure
$\rmv\lambda_{>0}$. Similarly, we set
$\IL^p(\rmv):=\IL^p(\IR,\BorelIR,\rmv\lambda)$ for a density function
$\rmv$ defined on $\Rz$. If $\rmv= \1$, i.e. $\rmv$ is mapping
constantly to one, we write shortly $\IL^p_+:=\IL^p_+(\mathds{1})$ and
$\IL^p:=\IL^p(\mathds{1})$.  At this point we shall remark, that we have
used and will further use the terminology \textit{density}, whenever we are meaning a
probability density function (such as $f$) and on the other hand side
\textit{density function}, whenever we are meaning a Radon-Nikodym
derivative (such as $\rmv$). For $c\in\IR$ we introduce the density
function $\rmx^c:\pRz\to\pRz$ given by $x\mapsto\rmx^c(x):=x^c$. The Mellin
transform $\M_c$ is the unique linear and bounded operator between the
Hilbert spaces $\IL^2_+(\rmx^{2c-1})$ and $\IL^2$, which for each
$h\in\IL^2_+(\rmx^{2c-1})\cap\IL^1_+(\rmx^{c-1})$ and $t\in\IR$ satisfies
\begin{equation}\label{def:mellin}
  \M_c[h](t):=\int_{\pRz}x^{c-1+\iota 2\pi t}h(x)\rmd\LebpRz(x),
\end{equation}
where $\iota\in\IC$ denotes the imaginary unit. Similar to the
Fourier transform the Mellin
transform $\M_c$  is unitary, i.e.
$\skalar{h,g}_{\IL^2_+(\rmx^{2c-1})} =
\skalar{\M_c[h],\M_c[g]}_{\IL^2}$ for any
$h,g\in\IL^2(\rmx^{2c-1})$. In particular, it satisfies a \textit{Plancherel type identity}
\begin{equation}\label{eq:Plancherel}
  \norm{h}_{\IL^2_+(\rmx^{2c-1})}^2 = \norm{\M_c[h]}_{\IL^2}^2,\quad\forall h\in\IL^2_+(\rmx^{2c-1}).
\end{equation}
Its adjoined (and inverse) $\M_c^{\star}:\IL^2\to\IL^2_+(\rmx^{2c-1})$ fulfils for each
$g\in\IL^2\cap\IL^1$ and $x\in\pRz$
\begin{equation}\label{def:mellin:inv}
  \M_c^{\star}[g](x):=\int_{\Rz}x^{-c-\iota 2\pi t}g(t)\rmd\lambda(t).
\end{equation}
For a detailed discussion of the Mellin
transform and its properties we refer again to \cite{Brenner-Miguel2023}. In analogy to the additive
convolution theorem of the Fourier transform (see
\cite{Meister2009} for definitions and properties), there is a
multiplicative convolution theorem for the Mellin
transform.  Namely, for any $h_1\in\IL^2_+(\rmx^{2c-1})$ and $h_2\in\IL^1_+(\rmx^{c-1})\cap
\IL^2_+(\rmx^{2c-1})$ the multiplicative convolution theorem states 
\begin{equation}\label{eq:multconvothm}
  \M_c[h_1\ocon h_2] = \M_c[h_1] \M_c[h_2].
\end{equation}
Here and subsequently, we assume that $f\in\IL^2_+(\rmx^{2c-1})$ and
$f^U\in\IL^1_+(\rmx^{c-1})\cap \IL^2_+(\rmx^{2c-1})$, hence $f^Y\in\IL^2_+(\rmx^{2c-1})$, for some
$c\in\IR$, which is from now on fixed. Note that $\IL^2$-integrability
is a common assumption in additive deconvolution, which might be here seen
as the particular case $c=1/2$. However, allowing for different values
$c\in\IR$ makes the dependence on
$c\in\IR$ of the assumptions stipulated below  visible.  In order to simplify the presentation in the
following sections, we introduce some frequently used shorthand
notations.
\begin{notation}[Empirical Mellin transformation]\label{not: MellinT} The Mellin transformations of the densities $f^Y$, $f$ and $f^U$, respectively, we denote briefly by 
		\begin{equation*}
			\MelY:= \M_c[f^Y]\text{, }	\MelX:= \M_c[f]\text{ and } \MelU:=\M_c[f^U].
		\end{equation*}
		The reciprocal of a function $\rmw:\Rz\to\IC$ is well-defined on the set $\{\rmw\neq 0\}:=\{t\in\IR:\,\rmw(t)\neq 0\}$ and for each $t\in\IR$ we write for short
		\begin{equation*}
			\rmwrez(t):=\frac{1}{\rmw(t)}\1_{\{\rmw\neq 0\}}(t)
		\end{equation*}
		where $\1_{\{\rmw\neq 0\}}$ denotes an indicator function. Precisely, for any set $A\subseteq\IR$ we write $\1_A$ for the indicator function, mapping $x\in\IR$ to $\1_A(x):=1$, if $x\in A$, and  $\1_A(x):=0$, otherwise. Since the Mellin transformation $\MelY$ of $f^Y$ is unknown we follow \cite{Brenner-MiguelComteJohannes2021} and introduce an empirical counterpart based on the observations $\{Y_i\}_{i\in\nset{n}}$. The \textit{empirical Mellin transformation} $\MelYhat$  is given by
		\begin{equation*}
			\MelYhat(t):=\frac{1}{n}\sum_{i\in\nset n}Y_{i}^{c-1+\iota2\pi t}
		\end{equation*}
		for any $t\in\IR$. Observe that $\MelYhat$ is a point-wise unbiased estimator of $\MelY$, i.e. for all $t\in\IR$, we have $	\IE[\MelYhat(t)] = \MelY(t)$, where $\IE$ denotes the expectation under the distribution of the observations $\{Y_i\}_{i\in\nset{n}}$. Let us further introduce  the point-wise scaled variance of $\MelYhat$ defined for each $t\in\IR$ by
		\begin{align*}
			\IV_Y^2(t):=n\IE[|\MelYhat(t)-\MelY(t)|^2] = \IE[|Y_1^{c-1+\iota2\pi t}-\MelY(t)|^2] 
			\leq 1+\IE[Y_1^{2(c-1)}]=:\sigma_{Y}^2,
		\end{align*}
		whenever $\IE[Y_1^{2(c-1)}]$ is finite.
		The estimation procedure we are discussing next is based on estimating the unknown Mellin transformation $\MelX$ in a first place. Having the multiplicative convolution theorem \eqref{eq:multconvothm} in mind, a estimator of $\MelX$ is given by 
		\begin{equation}
			\MelXtilde(t):=\MelYhat(t) \MelUrez(t)
		\end{equation}
		for each $t\in\IR$. To shorten notation we write $\MelXtilde^k:=\MelXtilde\1_{[-k,k]}$ for each $k\in\pRz$. Finally, in what follow  we denote by $\Const\in\IR_{>0}$ universal numerical constants and by $\Const(\cdot)\in\IR_{>0}$ constants depending only on the argument. In both cases, the values of the constants may change from line to line.
\end{notation}
\subsection{Nonparametric density estimation - known error distribution}\label{subsec:2.3}
In case of a known error
distribution of $U$  we follow  \cite{Brenner-MiguelComteJohannes2021} in defining a spectral cut-off
density estimator $\tilde{f}_k$ of $f$ for each $x\in\pRz$ and tuning
parameter $k\in\pRz$ by
\begin{equation}\label{eq: classical_estimator}
  \tilde{f}_{k}(x):=\int_{[-k,k]}x^{-c-\iota2\pi t}\MelYhat(t)\MelUrez(t)\mathrm{d}\lambda(t) = \M_c^{-1}\left[\MelXtilde^k\right](x).
\end{equation}
assuming that
$\1_{[-k,k]}\MelUrez\in\IL^2$ for each
$k\in\pRz$. The last condition ensures evidently that $\tilde{f}_k$ is well-defined.
If we  require that $\MelU\ne0$ a.s. then we have $f=\M_c^{-1}\left[\MelX
\right]=\M_c^{-1}\left[\MelY\MelUrez \right]$ due to the Mellin
inversion formula and  the  multiplicative convolution theorem.
Intuitively speaking, from the last identity we obtain the  estimator
in \eqref{eq: classical_estimator} by truncating and replacing the unknown Mellin transformation
$\MelY$ of $f^Y$ by its empirical counter part $\MelYhat$.
In \cite{Brenner-MiguelComteJohannes2021} the global $\IL^2_+(\rmx^{2c-1})$-risk of
$\tilde{f}_{k}$ is analysed, which 
for each $k\in\pRz$ is written as
\begin{equation}\label{thm: RiskBound_Known}
  \IE\left[\norm{\tilde{f}_k-f}^2_{\IL^2_+(\rmx^{2c-1})}\right] = \IE\left[\norm{\MelXtilde^k-\MelX}_{\IL^2}^2\right]
  = \norm{\mathds{1}_{[-k,k]^C}\MelX}_{\IL^2}^2 + \frac{1}{n}\norm{\mathds{1}_{[-k,k]}\mathbb{V}_Y\MelUrez}_{\IL^2}^2,
\end{equation}
where the first equality follows directly from Plancherel's identity
(see \eqref{eq:Plancherel}) and the second equality is due to the
usual  squared bias and variance decomposition.
\subsection{Nonparametric survival function estimation -  known error distribution}\label{subsec:2.4}
In this paragraph, following \cite{Brenner-MiguelPhandoidaen2022} we recall  an estimator of the survival function
$\SX$ of $X$ based on the observation $\{Y_i\}_{i\in\nset{n}}$, when the error density $f^U$ is known. 
The survival function of $X$ satisfies
\begin{equation*}
  \SX:
  \begin{cases}
    \pRz&\to[0,1]\\
    x&\mapsto\Pz[X\geq x].
  \end{cases}
\end{equation*}
As $X$ admits the density $f$, we evidently for all $x\in\IR$ have
\begin{equation*}
  \SX(x)= \int_{\pRz}\1_{[x,\infty)}(y)f(y)\mathrm{d}\lambda_{>0}(y).
\end{equation*}
 According to \cite{Brenner-MiguelPhandoidaen2022} we have
$f\in\IL_+^1(\rmx^{c-1})\cap\IL_+^2(\rmx^{2c-1})$ if and only if
$\SX\in\IL_+^1(\rmx^{c-2})\cap \IL_+^2(\rmx^{2c-3})$. Further,
elementary computations show for each $t\in\IR$ and $c\in\Rz_{>1}$ that
\begin{equation*}
  \M_{c-1}[\SX](t) = (c-1+\iota2\pi t)^{-1}\MelX(t),
\end{equation*} 
  Exploiting the last identity  \cite{Brenner-MiguelPhandoidaen2022} propose a spectral cut-off estimator
of $\SX$ given for all
$k,x\in\pRz$ by
\begin{align*}
  \tilde{S}^X_{k}(x)&:=\int_{[-k,k]}x^{-c+1-\iota2\pi t}\frac{\MelYhat(t)\MelUrez(t)}{c-1+\iota2\pi t}\mathrm{d}\lambda(t)
  = \M_{c-1}^{-1}\left[(c-1+\iota2\pi\cdot)^{-1}\MelXtilde^k(\cdot)\right](x).
\end{align*}
and study its $\IL^2_+(\rmx^{2c-3})$-risk, which reads as
\begin{align}\nonumber
  \IE\left[\norm{\tilde{S}^X_k-S}^2_{\IL^2_+(\rmx^{2c-3})}\right] &= \IE\left[\norm{(c-1+\iota2\pi\cdot)^{-1}(\MelXtilde^k-\MelX)}_{\IL^2}^2\right]\\
  \nonumber
  &= \IE\left[\norm{\MelXtilde^k-\MelX}_{\IL^2(\mathrm{t}_c)}^2\right] \\\label{eq:riskSX2}
  &= \norm{\mathds{1}_{[-k,k]^C}\MelX}_{\IL^2(\mathrm{t}_c)}^2 + \frac{1}{n}\norm{\mathds{1}_{[-k,k]}\mathbb{V}_Y\MelUrez}_{\IL^2(\mathrm{t}_c)}^2,		
\end{align}
where the Plancherel's identity \eqref{eq:Plancherel} yields the first
equality, the second makes use of the density function
$\rmt_c(t):=\left((c-1)^2+4\pi^2t^2\right)^{-1}$ for $t\in\IR$
and the third states the  squared bias and variance decomposition.
\subsection{Oracle type inequality and minimax optimal rates}\label{subsec:2.5}
In the previous paragraphs, we have seen that the global
$\IL^2_+(\rmx^{2c-1})$-risk for $\tilde{f}_k$ in \eqref{thm: RiskBound_Known} and the global
$\IL^2_+(\rmx^{2c-3})$-risk for $\tilde{S}^X_k$ in \eqref{eq:riskSX2} exactly equals
\begin{equation}\label{eq:riskdeco}
  \IE\left[\norm{\MelXtilde^k-\MelX}_{\IL^2(\mathrm{v})}^2\right] 
  = \norm{\mathds{1}_{[-k,k]^C}\MelX}_{\IL^2(\mathrm{v})}^2 + \frac{1}{n}\norm{\mathds{1}_{[-k,k]}\mathbb{V}_Y\MelUrez}_{\IL^2(\mathrm{v})}^2,
\end{equation}
with $\rmv = \1$ and $\rmv = \rmt_c$ accordingly. Therefore we study
in the sequel the risk in \eqref{eq:riskdeco} with an arbitrary
density function $\rmv$ more in detail and consider oracle-type
inequalities as well as minimax-optimal convergences rates. Let us
summarise the assumptions we have so far imposed.
\renewcommand{\theass}{A.\Roman{ass}}
\begin{ass}\label{ass:A1}\leavevmode\\	
  Let $X$ and $U$ be independent and $\pRz$-valued random variables
  with density $f$ and $f^U$, respectively. Consider
  i.i.d. observations $\{Y_i\}_{i\in\nset{n}}$ following the law of
  $Y = X\cdot U$, which admits the density $f^Y = f\ocon
  f^U$. In addition, let
  \begin{enumerate}
  \item [i)] $f^U\in\IL_+^1(\rmx^{c-1})\cap\IL_+^2(\rmx^{2c-1})$, $f\in\IL_+^2(\rmx^{2c-1})$ and $f^Y\in\IL_+^1(\rmx^{2(c-1)})$, set $\sigma_{Y}^2:=1+\IE[Y^{2(c-1)}] = 1+\norm{f^Y}_{\IL^1_+(\rmx^{2(c-1)})}$ .
  \item [ii)] $\rmv:\IR\to\IR_{\geq 0}$ be a (measurable) density function and set $\rmvU:=|\MelUrez|^2\rmv$.
  \item [iii)] $\MelX\in\IL^2(\rmv)$, $\MelU\ne0$ a.s. and $\1_{[-k,k]}\in\IL^2(\rmvU)$ (hence $\1_{[-k,k]}\in\IL^2(\rmv)$) for each $k\in\pRz$.\qed
  \end{enumerate}
\end{ass}
Having the  elementary risk decomposition as stated in
\eqref{eq:riskdeco} on hand, we aim to select a tuning parameter
$k_\frako\in\pRz$ minimising the risk, namely
\begin{align}\nonumber
  \IE\left[\norm{\MelXtilde^{k_\frako}-\MelX}_{\IL^2(\mathrm{v})}^2\right]& =\inf_{k\in\pRz}\left\{ \norm{\mathds{1}_{[-k,k]^C}\MelX}_{\IL^2(\mathrm{v})}^2 + \frac{1}{n}\norm{\mathds{1}_{[-k,k]}\mathbb{V}_Y\MelUrez}_{\IL^2(\mathrm{v})}^2\right\}\\\label{eq:oracle2}
 & \leq\sigma_Y^2\cdot  \inf_{k\in\pRz}\left\{ \norm{\mathds{1}_{[-k,k]^C}\MelX}_{\IL^2(\mathrm{v})}^2 + \frac{1}{n}\norm{\mathds{1}_{[-k,k]}\MelUrez}_{\IL^2(\rmv)}^2\right\},
\end{align}
which exists since the first term in \eqref{eq:riskdeco} decreases
while the second one increases for an increasing $k$. However, such a
$k_\frako$ is unattainable as it depends directly on the unknown
Mellin transformation $\MelX$. Therefore $k_\frako$ is called an
oracle and hence the upper bound in \eqref{eq:oracle2} is also called
oracle-type inequality. In the paragraph below, we revisit a
data-driven selection method for $k$ to avoid this lack of
information. For this choice we also obtain a oracle-type inequality
similar to \eqref{eq:oracle2} up to a different constant. We next
briefly recall minimax-optimal convergences rates for the global
$\IL^2(\rmv)$-risk as for instance derived in \cite{Brenner-MiguelComteJohannes2021}
and \cite{Brenner-MiguelPhandoidaen2022}. To do so, we impose regularity
assumptions on the Mellin transformation $\MelX$ and $\MelU$. Firstly,
let us recall the definition of the Mellin-Sobolev space with
regularity parameter $s\in\pRz$, given by
\begin{equation*}
  \IW^s:=\left\{h\in\IL^2(\rmx^{2c-1}):\,\norm{\rmt_2^{\frac{s}{2}}\M_c[h]}_{\IL^2}= \norm{\M_c [h]}_{\IL^2(\mathrm{t}_2^s)}<\infty\right\}.
\end{equation*}
Further, for some radius $L\in\pRz$, we consider the associated Mellin-Sobolev ellipsoid, given by
\begin{equation*}
  \IW^s(L):=\left\{h\in\IW^s:\,\norm{\M_c [h]}_{\IL^2(\mathrm{t}_2^s)}\leq L\right\}.
\end{equation*}
In the following brief discussion we assume an unknown density
$f\in\IW^s(L)$ of $X$, where the regularity $s$ is specified
below. Regarding the Mellin transformation $\MelU$, we assume in the
sequel its \textit{ordinary smoothness} \eqref{eq:os} or \textit{super
  smoothness} \eqref{eq:es}, i.e. for some decay parameter
$\gamma\in\pRz$ and all $t\in\IR$,
\begin{equation}\label{eq:os}
  \exists \const_l,\const_u\in\pRz:\, \const_l\cdot|t|^{-\gamma}\leq |\MelU(t)|\leq\const_u\cdot|t|^{-\gamma}\tag{$\mathfrak{o.s.}$}
\end{equation}
and
\begin{equation}\label{eq:es}
  \exists \const_l,\const_u\in\pRz:\, \exp(-\const_l\cdot|t|^{2\gamma})\leq |\MelU(t)|\leq \exp(-\const_u\cdot|t|^{2\gamma}),\tag{$\mathfrak{s.s.}$}
\end{equation}
respectively. In order to discuss the convergences rates for the
$\IL^2(\rmv)$-risk under these regularity assumptions, we restrict
ourselves to the choice $\rmv:=\rmt_c^a$ for $a\in\IR$, observing that
$a = 0$ corresponds to the global risk for estimating the density $f$
and $a=1$ corresponds to estimating the survival function $\SX$ of $X$
as discussed before. In \cref{tab:table1} we state an uniform upper
bound of the bias term
$\norm{\mathds{1}_{[-k,k]^C}\MelX}_{\IL^2(\mathrm{v})}^2$ over the
Mellin-Sobolev ellipsoid $\IW^s(L)$ and an upper bound for the variance term
$ \frac{1}{n}\norm{\mathds{1}_{[-k,k]}\MelUrez}_{\IL^2(\rmv)}^2$
provided the Mellin transformation $\MelU$ to be either ordinary
smooth \eqref{eq:os} or super smooth \eqref{eq:es}. Moreover in
\cref{tab:table1} is depict for both specifications the order of the
optimal choice $k_\frako$ and the upper risk bound \eqref{eq:oracle2}
as $n\to\infty$, which follow immediately from elementary
computations.
	\begin{table}[H]
	\centering
	\begin{tabular}{||@{\hspace*{3pt}}c@{\hspace*{3pt}}c@{\hspace*{3pt}}c@{\hspace*{3pt}}c@{\hspace*{3pt}}c@{\hspace*{3pt}}c@{\hspace*{3pt}}||}
	      \toprule
	      &$\MelU$ & Bias & Variance & $k_\frako$ & Minimax risk \\ [0.5ex] 
	      \midrule\midrule
	      $a\in(-1/2-\gamma,s)$& $\mathfrak{o.s.}$ & $k^{-2(s-a)}$ & $n^{-1}k^{2(a+\gamma)+1}$ & $n^{\frac{1}{2\gamma+2s+1}}$ &  $n^{-\frac{2(s-a)}{2\gamma+2s+1}}$\\ 
	      \midrule
	      $s>a$&$\mathfrak{s.s.}$  & $k^{-2(s-a)}$ & $n^{-1}k^{(1-2(\gamma-a))_+}\exp(k^{2\gamma})$ & $(\log n)^{\frac{1}{2\gamma}}$ & $(\log n)^{-\frac{s-a}{\gamma}}$\\ [1ex] 
	      \bottomrule
	\end{tabular}
	\caption{}\label{tab:table1}
	\end{table}
 In \cite{Brenner-MiguelComteJohannes2021} and
\cite{Brenner-MiguelPhandoidaen2022} it has been shown that the rates in \cref{tab:table1}
are actually minimax-optimal.
\subsection{Data driven estimation}\label{subsec:2.6}
\cite{Brenner-MiguelComteJohannes2021} and \cite{Brenner-MiguelPhandoidaen2022} propose
a data-driven choice of the tuning parameter $k$ by model selection
exploiting the theory of \cite{BarronBirgeMassart1999}, where we refer
to \cite{Massart2007} for an extensive overview. More precisely they
consider
\begin{equation}\label{eq:ktilde}
  \tilde{k}:\in\argmin_{k\in\mathcal{K}}\left\{-\norm{\mathds{1}_{[-k,k]} \MelYhat\MelUrez}^2_{\IL^2(\rmv)}+\pen_k\right\} = \argmin_{k\in\mathcal{K}}\left\{-\norm{\MelXtilde^k}^2_{\IL^2(\rmv)}+\pen_k\right\} ,
\end{equation}
where $\{\pen_k\}_{k\in\pRz}$ is a family of penalties and $\mathcal{K}\subset\IN$ is an appropriate finite subset specified later. The aim is to analyse the $\IL^2(\rmv)$-risk, namely
\begin{equation*}
  \IE\left[\norm{\MelXtilde^{\tilde{k}}-\MelX}_{\IL^2(\mathrm{v})}^2\right].
\end{equation*}
In the sequel we introduce a family of penalties and set of models
which differ from the original works in preparation of the procedure
presented in \cref{sec: 3,sec: 4} below.  For the upper risk bound we impose slightly stronger assumptions
than \cref{ass:A1} that we state next.
\begin{ass}\label{ass:A2}\leavevmode\\	
  In addition to \cref{ass:A1} let
  $f^Y\in\IL_+^1(\rmx^{8(c-1)})\cap\IL_+^\infty(\rmx^{2c-1})$ such
  that there exists $\eta_Y\in\IR_{\geq1}$ satisfying 
  $\eta_Y\geq\max\left\{\norm{f^Y}_{\IL_+^\infty(\rmx^{2c-1})},\norm{f^Y}_{\IL_+^1(\rmx^{8(c-1)})}\right\}$.
  We
  set $a_Y:=6\norm{f^Y}_{\IL_+^\infty(\rmx^{2c-1})}/\sigma_Y^2$ and 
  $k_Y:=1\vee3a_Y^2$.\qed
\end{ass}
Text-book computations as they can be found for instance in
\cite{Comte_2017_book} are leading to the following standard key
argument, which for any $\ko\in\mathcal{K}$ states
\begin{multline}\label{eq:keyargument01}
  \norm{\MelXtilde^{\tilde{k}}-\MelX}_{\IL^2(\mathrm{v})}^2\leq  3\norm{\mathds{1}_{[-\ko,\ko]^C}\MelX}_{\IL^2(\mathrm{v})}^2
  + 4\pen_{\ko} \\+8 \max_{k\in\mathcal{K}}\left\{\left(\norm{\mathds{1}_{[-k,k]} (\MelYhat-\MelY)\MelUrez}^2_{\IL^2(\rmv)}-\frac{\pen_k}{4}\right)_+\right\},
\end{multline}
using the shorthand notation $a_+:=\max\{0,a\}$ for any
$a\in\IR$. Recalling $\rmvU:=|\MelUrez|^2\rmv$, the last summand in
\eqref{eq:keyargument01} reads as
\begin{equation*}
  \max_{k\in\mathcal{K}}\left\{\left(\norm{\mathds{1}_{[-k,k]} (\MelYhat-\MelY)\MelUrez}^2_{\IL^2(\rmv)}-\frac{\pen_k}{4}\right)_+\right\} =  \max_{k\in\mathcal{K}}\left\{\left(\norm{\mathds{1}_{[-k,k]} (\MelYhat-\MelY)}^2_{\IL^2(\rmvU)}-\frac{\pen_k}{4}\right)_+\right\}.
\end{equation*}
Taking the expectation in the last display, the next proposition
allows to control its value. In its proof we make use of
\cref{re:conc:1} in \cref{app: sec2}, which is based on a Talagrand
inequality, stated in \cref{lem: Talagrand}. Here and subsequently,
for an arbitrary density function $\rmw:\IR\to\IR_{\geq 0}$ satisfying
$\1_{[-k,k]}\in\IL^\infty(\rmw)$ for each $k\in\pRz$, we denote by
\begin{equation}\label{re:conc:def}
  \Delta_k^{\rmw} := \norm{\1_{[-k,k]}}_{\IL^\infty(\rmw)}\text{
    and }\delta_k^{\rmw}
  :=
  \frac{\log(\Delta_k^{\rmw}\lor(k+2))}{\log(k+2)}\in\IR_{\geq1},\;\forall
  k\in\IR_{\geq1}.
\end{equation}
\begin{pr}[Concentration inequality]\label{re:conc} 
  Under \cref{ass:A2} define $ \Delta_k^{\rmvU}$ and
  $\delta_k^{\rmvU}$ as in \eqref{re:conc:def} with
  $\rmvU:=|\MelUrez|^2\rmv$. For $n\in\IN$ consider
  \begin{equation*}
    k_n:=\max\{k\in\nset{n^2}:k\Delta_{k}^{\rmvU} \leq
    n^2\Delta_1^{\rmvU}\}.
  \end{equation*}
  We then have
  \begin{equation*}
    \IE\left[\max_{k\in\nset{k_n}}\left\{\left(\norm{\mathds{1}_{[-k,k]}
            (\MelYhat-\MelY)}^2_{\IL^2(\rmvU)}-12\sigma_{Y}^2\Delta_k^{\rmvU}\delta_k^{\rmvU} k
          n^{-1}\right)_+\right\}\right] \leq \Const\cdot\eta_Y(1\vee k_Y\Delta_{k_Y}^{\rmvU})\cdot n^{-1}. 
  \end{equation*}
\end{pr}
\begin{proof}[\cref{re:conc}]From the
  decomposition \eqref{proof:conc:decomp} in \cref{rem:conc} (with
  $\rmw=\rmvU$) we obtain
  \begin{align*}
    \IE&\left[\max_{k\in\nset{k_n}}\left\{\left(\norm{\mathds{1}_{[-k,k]}
            (\MelYhat-\MelY)}^2_{\IL^2(\rmvU)}-12\sigma_{Y}^2\Delta_k^{\rmvU}\delta_k^{\rmvU} k
          n^{-1}\right)_+\right\}\right]\\
          &\leq
    2 \IE\left[\norm{\mathds{1}_{[-k_n,k_n]}
        (\MelYhat^{\mathrm{u}}-\MelY^{\mathrm{u}})}^2_{\IL^2(\rmvU)}\right]
    2 \IE\left[\max_{k\in\nset{k_n}}\left\{\left(\norm{\mathds{1}_{[-k,k]}
            (\MelYhat^{\mathrm{b}}-\MelY^{\mathrm{b}})}^2_{\IL^2(\rmvU)}-6\sigma_{Y}^2\Delta_k^{\rmvU}\delta_k^{\rmvU} k
          n^{-1}\right)_+\right\}\right],
  \end{align*}
  where we bound the two right hand side terms separately with the
  help of \cref{re:conc:1}. Therewith we obtain
  \begin{align}\nonumber
    &\IE\left[\max_{k\in\nset{k_n}}\left\{\left(\norm{\mathds{1}_{[-k,k]}
            (\MelYhat-\MelY)}^2_{\IL^2(\rmvU)}-12\sigma_{Y}^2\Delta_k^{\rmvU}\delta_k^{\rmvU} k
          n^{-1}\right)_+\right\}\right]\\\nonumber&\leq n^{-1} \Const \bigg[\E\big[Y_1^{8(c-1)}\big] n^{-2}k_n\Delta_{k_n}^{\rmvU}+n^{-4}k_n^2\Delta_{k_n}^{\rmvU}  \\\label{eq:proof:CI1}&\phantom{=}+(1\vee\norm{f^Y}_{\IL^\infty(\rmx^{2c-1})}^2)(1+\norm{\mathds{1}_{[-k_Y,k_Y]}}_{\IL^2(\rmvU)}^2)\sum_{k\in\nset{k_n}}\frac{1}{a_Y}\exp\big(\frac{-k}{a_Y}\big)\bigg].
  \end{align}
  Exploiting that $\sum_{k\in\Nz}\exp(-k/a_Y)\leq a_Y$ and the
  definition of $k_n\in\nset{n^2}$ we have
  \begin{multline*}
    \IE\left[\max_{k\in\nset{k_n}}\left\{\left(\norm{\mathds{1}_{[-k,k]}
            (\MelYhat-\MelY)}^2_{\IL^2(\rmvU)}-12\sigma_{Y}^2\Delta_k^{\rmvU}\delta_k^{\rmvU} k
          n^{-1}\right)_+\right\}\right]\\\leq n^{-1} \Const
    \bigg[(1\vee\E\big[Y_1^{8(c-1)}\big])\Delta_1^{\rmvU} +(1\vee\norm{f^Y}_{\IL^\infty(\rmx^{2c-1})}^2)(1+\norm{\mathds{1}_{[-k_Y,k_Y]}}_{\IL^2(\rmvU)}^2)\bigg],
  \end{multline*}
  which together with
  $\E\big[Y_1^{8(c-1)}\big]=\norm{f^Y}_{\IL^1(\rmx^{8(c-1)})}$,
  $\norm{\mathds{1}_{[-k_Y,k_Y]}}_{\IL^2(\rmvU)}^2\leq
  2k_Y\Delta_{k_Y}^{\rmvU}$,
  $\Delta_1^{\rmvU}\leq k_Y\Delta_{k_Y}^{\rmvU}$ and the definition of
  $\eta_Y$ shows the claim and completes the proof.
\end{proof}
Having \cref{re:conc} at hand we set $\mathcal{K}:=\nset{k_n}\subset \IN$
and define for each $k\in\nset{k_n}$
\begin{equation*}
  \pen_k^{\rmvU} := 48\Delta_k^{\rmvU}\delta_k^{\rmvU} k n^{-1}.
\end{equation*}
Evidently, choosing $\pen_k := \sigma_{Y}^2 \pen_k^{\rmvU}$
\cref{re:conc} allows to bound the expectation of the last summand in
\eqref{eq:keyargument01}. Unfortunately,
$\sigma_{Y}^2=1+ \IE[Y^{2(c-1)}]$ is unknown to us. However, we have
at our disposal an unbiased estimator given by
\begin{equation*}
  \sigmaYhat :=1+ n^{-1}\sum_{i\in\nset{n}}Y_i^{2(c-1)}.
\end{equation*}
Hence, replacing subsequently the unknown $\sigma_{Y}^2$ by its empirical counterpart $\sigmaYhat$ we consider the data driven choice
\begin{equation}\label{eq:khat}
  \hat{k}:\in\argmin_{k\in\nset{k_n}}\left\{-\norm{\MelXtilde^k}^2_{\IL^2(\rmv)}+2\sigmaYhat\pen_k^{\rmvU}\right\}.
\end{equation}
Elementary computations in \cite{Brenner-MiguelComteJohannes2021} show a slightly
changed version of the key argument \eqref{eq:keyargument01}, which
reads for each $\ko\in\nset{k_n}$ as
\begin{multline}\label{eq:keyargument1}
  \norm{\MelXtilde^{\khat}-\MelX}_{\IL^2(\mathrm{v})}^2\leq  3\norm{\mathds{1}_{[-\ko,\ko]^C}\MelX}_{\IL^2(\mathrm{v})}^2
  + 2\sigma_{Y}^2\pen_{\ko}^{\rmvU} + 4\sigmaYhat\pen_{\ko}^{\rmvU}
  + 2(\sigma_{Y}^2-2\sigmaYhat)_+ \pen_{k_n}^{\rmvU} \\+8 \max_{k\in\nset{k_n}}\left\{\left(\norm{\mathds{1}_{[-k,k]} (\MelYhat-\MelY)}^2_{\IL^2(\rmvU)}-\frac{\sigma_{Y}^2\pen_k^{\rmvU}}{4}\right)_+\right\}.
\end{multline}
By applying the expectation on both sides on the inequality of the last display, and taking 
\begin{equation}\label{eq:ko}
  \ko:\in\argmin_{k\in\nset{k_n}}\left\{ \norm{\mathds{1}_{[-k,k]^C}\MelX}_{\IL^2(\mathrm{v})}^2 + \Delta_k^{\rmvU}\delta_k^{\rmvU} k n^{-1}\right\},
\end{equation}
we immediately obtain
\begin{multline}\label{eq:oracleadaptive2}
  \IE\left[\norm{\MelXtilde^{\khat}-\MelX}_{\IL^2(\mathrm{v})}^2\right] 
  \leq 6\sigma_Y^2 \cdot\min_{k\in\nset{k_n}}\left\{ \norm{\mathds{1}_{[-k,k]^C}\MelX}_{\IL^2(\mathrm{v})}^2 + \Delta_k^{\rmvU}\delta_k^{\rmvU} k n^{-1}\right\} + \Const\cdot\eta_Y(1\vee k_Y\Delta_{k_Y}^{\rmvU})\cdot n^{-1},
\end{multline}
due to $\IE[\sigmaYhat] = \sigma_{Y}^2$, \cref{re:conc} as well as
\begin{equation*}
  \IE\left[\left(\frac{\sigma_{Y}^2}{2}-\sigmaYhat\right)_+\pen_{k_n}^{\rmvU}\right] \leq \Const \eta_Y   \Delta_1^{\rmvU}n^{-1}.
\end{equation*}
At this point we want to stress out again that specifying $\rmv=\1$
and $\rmv = \rmt_c$, respectively, we obtain directly from
\eqref{eq:oracleadaptive2} upper bounds for
\begin{equation*}
  \IE\left[\norm{\tilde{f}_{\khat}-f}_{\IL_+^2(\rmx^{2c-1})}^2\right]\text{ and }		\IE\left[\norm{\tilde{S}^X_{\khat}-\SX}_{\IL_+^2(\rmx^{3c-2})}^2\right]
\end{equation*}
due to Plancherel's identity \eqref{eq:Plancherel}. Similar to
\cref{subsec:2.5} we assume in the following brief discussion a density
$f\in\IW^s(L)$ of $X$, where the regularity $s\in\pRz$ is specified
below. Regarding the Mellin transformation $\MelU$, we subsequently
assume again its \textit{ordinary smoothness} \eqref{eq:os} or
\textit{super smoothness} \eqref{eq:es}. In order to discuss the
convergences rates for the $\IL^2(\rmv)$-risk under these regularity
assumptions, we restrict ourselves to the choice $\rmv:=\rmt_c^a$ for
$a\in\IR$ again, observing that $a = 0$ corresponds to the global risk
for estimating the density $f$ and $a=1$ corresponds to estimating the
survival function $\SX$ of $X$ as discussed before. In
\cref{tab:table3} we state again the uniform upper bound of the bias
term $\norm{\mathds{1}_{[-k,k]^C}\MelX}_{\IL^2(\mathrm{v})}^2$ over
the Mellin-Sobolev ellipsoid $\IW^s(L)$ and in an addition an upper
bound for the variance term
$\Delta_k^{\rmvU}\delta_k^{\rmvU} k n^{-1}$ provided the Mellin
transformation $\MelU$ to be either ordinary smooth \eqref{eq:os} or
super smooth \eqref{eq:es}. Moreover in \cref{tab:table2} below is depict
for both specifications the order of the optimal choice $\ko$ as in
\eqref{eq:ko} and the upper risk bound \eqref{eq:oracleadaptive2} as
$n\to\infty$, which follow immediately from elementary computations.
  \begin{table}[H]
  	\centering
    \begin{tabular}{||@{\hspace*{3pt}}c@{\hspace*{3pt}}c@{\hspace*{3pt}}c@{\hspace*{3pt}}c@{\hspace*{3pt}}c@{\hspace*{3pt}}c@{\hspace*{3pt}}||}
      \toprule
      &$\MelU$ & Bias & Variance &
                                   $\ko$ & Data-driven risk \\ [0.5ex] 
           \midrule\midrule
      $a\in(-\frac{1}{2}-\gamma,s)$& $\mathfrak{o.s.}$ & $k^{-2(s-a)}$ & $n^{-1}k^{2(a+\gamma)+1}$ & $n^{\frac{1}{2\gamma+2s+1}}$ &  $n^{-\frac{2(s-a)}{2\gamma+2s+1}}$\\ 
                 \midrule
      $s>a$&$\mathfrak{s.s.}$  & $k^{-2(s-a)}$ & $n^{-1}k^{1+2\gamma}\exp(k^{2\gamma})$ & $(\log n)^{\frac{1}{2\gamma}}$ & $(\log n)^{-\frac{s-a}{\gamma}}$\\ [1ex] 
      \bottomrule
    \end{tabular}
    \caption{}\label{tab:table2}
  \end{table}
Comparing \cref{tab:table1,tab:table2}  we observe that the
order of the risk bound coincide and hence the data driven estimation
is minimax-optimal.
	
\section{Estimation strategy for unknown error density}\label{sec: 3}

After recapitulating an estimation strategy for the multiplicative deconvolution problem
assuming the error density is known in advance, we dismiss this assumption in this section. Inspired by similar ideas for
additive deconvolution problems (see for instance \cite{Neumann1997} and \cite{Johannes2009}), we study estimation in the multiplicative
deconvolution problem with unknown error density. In addition to i.i.d. observations $\{Y_i\}_{i\in\nset{n}}$ following the law of the
multiplicative measurement model $Y = X\cdot U$, we have access to additional measurements $\{U_j\}_{j\in\nset{m}}$, $m\in\IN$, which
are i.i.d. drawn following the law of $U$ independently of the first
sample $\{Y_i\}_{i\in\nset{n}}$. We estimate the Mellin
transformation $\MelU$ by its empirical
counterpart given for each $t\in\IR$ by
\begin{equation*}
  \MelUhat(t):= m^{-1}\sum_{j\in\nset{m}}U_j^{c-1+\iota2\pi t}.
\end{equation*}
Similarly to $\MelYhat$, we observe that
$\IE[\MelUhat(t)] = \MelU(t)$ for all $t\in\IR$, i.e. $\MelUhat$ is an
unbiased estimator of $\MelU$. Here and subsequently, the expectation $\IE$ is considered with respect to the joint distribution of $\{Y_i\}_{i\in\nset{n}}$ and $\{U_j\}_{j\in\nset{m}}$. As in \cref{sec: 2} we intend to divide by
$\MelUhat$ whenever it is well defined, or in equal multiply with
$\MelUhatrez= \frac{1}{\MelUhat}\1_{\{\MelUhat\neq 0\}}$ (see
\cref{not: MellinT}). Note that the indicator set
$\{\MelUhat\neq 0\}:=\{t\in\IR:\,\MelUhat(t)\neq 0\}$ is not
deterministic anymore since it depends on the random variables
$\{U_j\}_{j\in\nset{m}}$. However, we note that $|\MelUhatrez|$ is generally
unbounded on the event $\{\MelUhat\neq 0\}$ which would lead to an
unstable estimation. Hence, we truncate $\MelUhat$ sufficiently far
away from zero. Recalling that $n$ and $m$ denote the samples sizes of
$\{Y_i\}_{i\in\nset{n}}$ and $\{U_j\}_{j\in\nset{m}}$, we thus define
in accordance with \cite{Neumann1997} the random indicator set
\begin{equation*}
  \frakM:=\left\{(m\land n)|\MelUhat|^2\geq 1\right\}:=\left\{t\in\IR:\,(m\land n)|\MelUhat(t)|^2\geq 1\right\},
\end{equation*} 
which only depends on the additional measurements $\{U_j\}_{j\in\nset{m}}$. Similarly as in \cref{sec: 2} we have now all ingredients to define an estimator of the unknown Mellin transformation $\MelX$. Indeed, having the convolution theorem \eqref{eq:multconvothm} in mind (stating $\MelY = \MelX\cdot\MelU$) we propose as an estimator of $\MelX$ 
\begin{equation*}
  \MelXhat := \MelYhat \MelUhatrez\1_\frakM. 
\end{equation*}
To simplify the presentation later, we further write
$\MelXhat^k:=\MelXhat\1_{[-k,k]}$ and $\MelX^k:=\MelX\1_{[-k,k]}$ for
any $k\in\pRz$.
\subsection{Nonparametric density estimation - unknown error density}
Motivated by the estimation strategy in \cref{subsec:2.3} we propose in this paragraph a thresholded spectral cut-off density
estimator for $f$ in the multiplicative deconvolution problem with
unknown error distribution.
\begin{definition}[Thresholded spectral cut-off estimator]\label{def: estimator}
  Assuming an unknown density $f\in\IL^2_+(\rmx^{2c-1})$ of
  $X$  for some $c\in\IR$ define the thresholded
  spectral density estimator $\hat{f}_k$  for each $k,x\in\pRz$ by
  \begin{equation*}
    \hat{f}_{k}(x):=\int_{[-k,k]}x^{-c-\iota2\pi t}\MelYhat(t)\MelUhatrez(t)\1_\mathfrak{M}(t)\mathrm{d}\lambda(t) = \M_c^{-1}\left[\MelXhat^k\right](x).
  \end{equation*}
\end{definition}
The next lemma provides a
 representation of the global $\IL^2_+(\rmx^{2c-1})$-risk of the
 estimator $\hat{f}_k$ similar to the decomposition \eqref{thm: RiskBound_Known}
 of the global $\IL^2_+(\rmx^{2c-1})$-risk of  $\tilde{f}_k$ in  \cref{subsec:2.3}.
 \pagebreak
\begin{lem}[Risk representation]\label{lem: 1}
  For $k\in\pRz$ consider the density estimator $\hat{f}_k$ given in
  \cref{def: estimator} and recall the definition of $\IV_Y^2$ (see
  \cref{not: MellinT}).  We then have
  \begin{align*}
    \E\left[\norm{\hat{f}_k-f}^2_{\IL^2_+(\rmx^{2c-1})}\right] &= \IE\left[\norm{\MelXhat^k-\MelX}_{\IL^2}^2\right]\\
    &=\norm{\MelX\mathds{1}_{[-k,k]^C}}^2_{\IL^2} 
    +\frac{1}{n}\E\left[\norm{\MelUhat^\dagger\1_\mathfrak{M}\mathbb{V}_Y\mathds{1}_{[-k,k]}}^2_{\IL^2}\right]\\
   &\phantom{=}+\E\left[\norm{\MelX^k\1_{\mathfrak{M}^C}}_{\IL^2}^2\right] + \IE\left[\norm{\MelUhat^\dagger\1_\mathfrak{M}(\MelU-\MelUhat)\MelX^k}^2_{\IL^2}\right].
  \end{align*}
\end{lem}
\begin{proof}[\cref{lem: 1}]			
  Recalling the definition of $\hat{f}_k$ as well as Plancherel's
  identity (see \cref{eq:Plancherel}), we have
  \begin{align*}
    \norm{\hat{f}_k-f}_{\IL^2_+(\rmx^{2c-1})}^2
    &= \norm{\MelXhat^k-\MelX}_{\IL^2}^2
      = \norm{\MelXhat^k-\MelX(\1_{[-k,k]}+\1_{[-k,k]^C})}_{\IL^2}^2\\
    &=\norm{\MelX\1_{[-k,k]}-\mathds{1}_{[-k,k]} \MelYhat\MelUhatrez\1_\mathfrak{M}}^2_{\IL^2}+\norm{\MelX\mathds{1}_{[-k,k]^C}}^2_{\IL^2}\\
    &=\norm{\MelUhat^\dagger\1_\frakM(\MelYhat-\MelUhat\MelX)\1_{[-k,k]}}^2_{\IL^2} 
   + \norm{\MelX\mathds{1}_{[-k,k]}\1_{\mathfrak{M}^C}}_{\IL^2}^2 +\norm{\MelX\mathds{1}_{[-k,k]^C}}^2_{\IL^2},
  \end{align*}
  where we used in the last step that 
  \begin{equation}
    \MelX\1_{[-k,k]}=	\MelX\1_{[-k,k]}(\1_\frakM+\1_{\frakM^C})= \MelUhat^\dagger\MelUhat\1_\mathfrak{M}\MelX\1_{[-k,k]} + \MelX\1_{[-k,k]}\1_{\frakM^C}.
  \end{equation}
  Studying only the first summand further studied we obtain for each $t\in\IR$
  \begin{equation*}
    n\E\left[|\MelYhat(t)-\MelY(t)|^2\right] = \IE\left[|Y_1^{c-1+\iota2\pi t}-\MelY(t)|^2\right] 
    =\mathbb{V}_Y^2(t).
  \end{equation*}
  By exploiting the independence of $\{Y_i\}_{i\in\nset{n}}$ and
  $\{U_j\}_{j\in\nset{m}}$, we finally have
  \begin{align*}
    \IE\left[\norm{\MelUhat^\dagger\1_\frakM(\MelYhat-\MelUhat\MelX)\1_{[-k,k]}}^2_{\IL^2} \right] &= \IE\left[\norm{\MelUhat^\dagger\1_\mathfrak{M}(\MelU-\MelUhat)\MelX^k}^2_{\IL^2}\right]
    +\frac{1}{n}\E\left[\norm{\MelUhat^\dagger\1_\mathfrak{M}\mathbb{V}_Y\mathds{1}_{[-k,k]}}^2_{\IL^2}\right],
  \end{align*}
  which shows the claim.
\end{proof}
At this point we want to stress out that the
$\IL^2_+(\rmx^{2c-1})$-risk representation of \cref{lem: 1} for
$\hat{f}_k$ has a very similar structure as the corresponding risk
representation of $\tilde{f}_k$ in \eqref{thm: RiskBound_Known} assuming the error
density $f^U$ to be known. Indeed, the first term remains the same -
it represents the bias term, which can be specified later, considering
different regularity assumptions on $f$.  Later, we see that the
second term actually represents the variance term. The two additional
summands, only depending on the additional measurements
$\{U_j\}_{j\in\nset{m}}$, occur in this particular situation, where
we estimate the unknown $\MelU$ as well.
\subsection{Nonparametric survival analysis - unknown error density}
Considering the survival analysis in \cref{subsec:2.4}, we propose now
an estimator for the survival function $\SX$ of $X$ under
multiplicative measurement errors with unknown error density $f^U$ and
additional measurements $\{U_j\}_{j\in\nset{m}}$. Indeed, we follow the definition of
$\tilde{S}_k$ and analogously as for the density estimator
$\tilde{f}_k$ we replace the unknown Mellin transformation $\MelU$ by
its (sufficiently truncated) counterpart, which leads to the
following definition.
\pagebreak
\begin{definition}[Thresholded spectral cut-off estimator of $\SX$]\label{def: estimatorSX}
  Assuming an unknown density $f\in\IL^2_+(\rmx^{2c-1})$ of
  $X$  for some $c\in\IR_{>1}$. The thresholded spectral cut-off estimator $\SXhat_k$ of the survival function $\SX$ of $X$ is defined for each $k,x\in\pRz$ by
  \begin{align*}
    \SXhat_k(x):&=\int_{[-k,k]}x^{-c+1-\iota2\pi t}\frac{\MelYhat(t)\MelUhatrez(t)}{c-1+\iota2\pi t}\1_\mathfrak{M}(t)\mathrm{d}\lambda(t)
     = \M_{c-1}^{-1}\left[(c-1+\iota2\pi\cdot)^{-1}\MelXhat^k\right](x).
  \end{align*}
\end{definition}
Similar to \eqref{eq:riskSX2} we are again interested in quantifying
the accuracy of $\SXhat_k$ in terms of its global
$\IL_+^2(\rmx^{2c-3})$-risk. The representation in the next corollary
follows line by line the proof of \cref{lem: 1} and is hence omitted.
\begin{co}[Risk Representation]\label{lem: 1SX}
  For $k\in\pRz$ consider the estimator $\hat{S}_k$ given in
  \cref{def: estimatorSX} and recall that $\rmt_c:\IR\to\pRz$ with
  $t\mapsto\rmt_c(t):=\left((c-1)^2+4\pi^2t^2\right)^{-1}$. We then
  have
  \begin{align*}
    \IE\left[\norm{\SXhat_k-\SX}^2_{\IL^2_+(\rmx^{2c-3})}\right] &= \IE\left[\norm{(c-1+\iota2\pi\cdot)^{-1}\left(\MelXhat^k-\MelX\right)}_{\IL^2}^2\right]= \IE\left[\norm{\MelXhat^k-\MelX}_{\IL^2(\rmt_c)}^2\right]\\
    &=\norm{\MelX\mathds{1}_{[-k,k]^C}}^2_{\IL^2(\rmt_c)} 
    +\frac{1}{n}\E\left[\norm{\MelUhat^\dagger\1_\mathfrak{M}\mathbb{V}_Y\mathds{1}_{[-k,k]}}^2_{\IL^2(\rmt_c)}\right]\\
    &\phantom{=}+\E\left[\norm{\MelX^k\1_{\mathfrak{M}^C}}_{\IL^2(\rmt_c)}^2\right] + \IE\left[\norm{\MelUhat^\dagger\1_\mathfrak{M}(\MelU-\MelUhat)\MelX^k}^2_{\IL^2(\rmt_c)}\right].
  \end{align*}
\end{co}
Analogously to estimating the density $f$, we obtain in \cref{lem:
  1SX} a risk representation with very similar structure as the
corresponding risk representation of $\tilde{S}_k$ in \eqref{eq:riskSX2} assuming the
error density $f^U$ to be known.  The first and second term represents
again the bias and variance as seen before. The last two summands,
depending only on the additional measurements
$\{U_j\}_{j\in\nset{m}}$, occur only due to the estimation of the
unknown Mellin transformation $\MelU$. Eventually, one observes
directly, that the risk decomposition of $\hat{f}_k$ and $\SXhat_k$
are identical up to the density function $\rmv = \1$ and
$\rmv = \rmt_c$, respectively. Hence, we subsequently study both cases
simultaneously by considering a general $\IL^2(\rmv)$-risk for an
arbitrary density function $\rmv$, namely
\begin{equation}
  \IE\left[\norm{\MelXhat^k-\MelX}_{\IL^2(\rmv)}^2\right].
\end{equation}
\subsection{Oracle type inequalities and minimax optimal rates}\label{subsec:3.3}
We start by formalising assumptions needed to be satisfied in the subsequent parts.
\renewcommand{\theass}{B.\Roman{ass}}
\begin{ass}\label{ass:B1}\leavevmode\\
  In addition to \cref{ass:A1} let $\{U_j\}_{j\in\nset{m}}$ be i.i.d. copies of $U$, independently drawn from the sample $\{Y_i\}_{i\in\nset{n}}$ and let $f^U\in\IL_+^1(\rmx^{4(c-1)})$. Moreover suppose that $\1_{[-k,k]}\in\IL^2(\rmv)$ for each $k\in\pRz$.\qed
\end{ass}
Under \cref{ass:B1} we start with the general $\IL^2(\rmv)$-risk representation for $k\in\pRz$, given by
\begin{multline*}
  \IE\left[\norm{\MelXhat^k-\MelX}_{\IL^2(\rmv)}^2\right]
  =\norm{\MelX\mathds{1}_{[-k,k]^C}}^2_{\IL^2(\rmv)} 
  +\frac{1}{n}\E\left[\norm{\MelUhat^\dagger\1_\mathfrak{M}\mathbb{V}_Y\mathds{1}_{[-k,k]}}^2_{\IL^2(\rmv)}\right]\\
  \phantom{=}+\E\left[\norm{\MelX^k\1_{\mathfrak{M}^C}}_{\IL^2(\rmv)}^2\right] + \IE\left[\norm{\MelUhat^\dagger\1_\mathfrak{M}(\MelU-\MelUhat)\MelX^k}^2_{\IL^2(\rmv)}\right],
\end{multline*}
where the proof follows line by line of the proof of \cref{lem: 1}. We upper bound the last three summands of the last display with the help of \cref{lem:2} in \cref{app: sec3} and derive for each $k\in\pRz$
\begin{align*}
  \IE\left[\norm{\MelXhat^k-\MelX}_{\IL^2(\rmv)}^2\right]
  &\leq  \norm{\MelX\mathds{1}_{[-k,k]^C}}^2_{\IL^2(\rmv)}
 + 4\left(1\lor\IE[U_1^{2(c-1)}]\right)^2\cdot\IE[X_1^{2(c-1)}]\cdot\norm{\MelUrez\mathds{1}_{[-k,k]}}^2_{\IL^2(\rmv)}\cdot n^{-1}\\
  &\phantom{=}+4\big(1\lor\IE[U_1^{2(c-1)}]\big)\cdot\norm{\MelX^k\big(1\vee|\MelU|^2(m\land n)\big)^{-1/2}}_{\IL^2(\rmv)}^2\\
  &\phantom{=}+4\big(1\lor\Const\,\IE[U_1^{4(c-1)}]\big) \cdot\norm{\MelX^k\big(1\vee|\MelU|^2m\big)^{-1/2}}_{\IL^2(\rmv)}^2.
\end{align*}
Further, elementary computations show that
\begin{align*}
  \norm{\MelX^k\big(1\vee|\MelU|^2(m\land
    n)\big)^{-1/2}}_{\IL^2(\rmv)}^2
  \leq \norm{\MelX^k\big(1\vee|\MelU|^2m\big)^{-1/2}}_{\IL^2(\rmv)}^2 + \IE[X_1^{2(c-1)}]\norm{\MelUrez\mathds{1}_{[-k,k]}}^2_{\IL^2(\rmv)}\cdot n^{-1},
\end{align*}
such that we finally obtain
\begin{align*}
  \IE\left[\norm{\MelXhat-\MelX}_{\IL^2(\rmv)}^2\right]&\leq
  \norm{\MelX\mathds{1}_{[-k,k]^C}}^2_{\IL^2(\rmv)}+ 8\left(1\lor\IE[U_1^{2(c-1)}]\right)^2\cdot\IE[X_1^{2(c-1)}]\cdot\norm{\MelUrez\mathds{1}_{[-k,k]}}^2_{\IL^2(\rmv)}\cdot n^{-1}\\
  &\phantom{=}+8\big(1\lor\Const\,\IE[U_1^{4(c-1)}]\big) \norm{\MelX^k\big(1\vee|\MelU|^2m\big)^{-1/2}}_{\IL^2(\rmv)}^2.
\end{align*}
Having the last upper bound of the risk at hand we select the tuning
parameter $k_\frako\in\pRz$ minimising the risk as defined in
\eqref{eq:oracle2}, such that
\begin{align}\nonumber
  \IE\left[\norm{\MelXhat^{k_\frako}-\MelX}_{\IL^2(\mathrm{v})}^2\right]
  &\leq
    8\left(1\lor\IE[U_1^{2(c-1)}]\right)^2\cdot\left(1\lor\IE[X_1^{2(c-1)}]\right)\\\nonumber
  &\hspace*{-10ex}\times\inf_{k\in\pRz}\left\{\norm{\MelX\mathds{1}_{[-k,k]^C}}^2_{\IL^2(\rmv)} 
    +\frac{1}{n}\left[\norm{\mathds{1}_{[-k,k]}\MelUrez}^2_{\IL^2(\rmv)}\right]\right\}\\\label{eq:oracle3}
  &\hspace*{-10ex}\phantom{=}+8\big(1\lor\Const\,\IE[U_1^{4(c-1)}]\big)\cdot \norm{\MelX\big(1\vee|\MelU|^2m\big)^{-1/2}}_{\IL^2(\rmv)}^2.
\end{align}
Observe that in the oracle-type inequality \eqref{eq:oracle2} and \eqref{eq:oracle3}
the first summand is  equal up  to the constant and we
refer to its discussion in \cref{subsec:2.5}.  Again, $k_\frako$
represents an oracle choice, as it depends on the unknown Mellin
transformation $\MelX$. In case of additive convolution on the circle
and the real line \cite{Johannes2013} and \cite{Neumann1997} derive
respectively a oracle-type inequality with similar structure as in
\eqref{eq:oracle3}. Moreover they show that the last summand is
unavoidable in a minimax-sense, indicating that this might be also
true for multiplicative deconvolution. Similar to  \cref{subsec:2.5}
we assume in the following brief discussion
 an unknown density $f\in\IW^s(L)$,
where the regularity $s$ is specified below. Regarding the Mellin
transformation $\MelU$, we subsequently assume again its
\textit{ordinary smoothness} \eqref{eq:os} or \textit{super
  smoothness} \eqref{eq:es}. In order to discuss the convergences
rates for the $\IL^2(\rmv)$-risk under these regularity assumptions,
we restrict ourselves again to the choice $\rmv:=\rmt_c^a$ for
$a\in\IR$, observing that $a = 0$ corresponds to the global risk for
estimating the density $f$ and $a=1$ corresponds to estimating the
survival function $\SX$ of $X$ as discussed before. In
\cref{tab:table3} below we state for the specifications considered in
\cref{tab:table1} the order of the optimal choice $k_\frako$ again and
the upper risk bound \eqref{eq:oracle3} as $n,m\to\infty$, which
follow immediately from elementary computations. 
  \begin{table}[H]
  	\centering
    \begin{tabular}{||@{\hspace*{3pt}}c@{\hspace*{3pt}}c@{\hspace*{3pt}}c@{\hspace*{3pt}}c@{\hspace*{3pt}}||}
      \toprule
      &$\MelU$ & $k_\frako$ & Maximal oracle
                              risk \\ [0.5ex] 
      \midrule\midrule
      $a\in(-1/2-\gamma,s)$& $\mathfrak{o.s.}$ & $n^{\frac{1}{2\gamma+2s+1}}$ &$n^{-\frac{2(s-a)}{2\gamma+2s+1}}+m^{-(\frac{s-a}{\gamma}\land 1)}$\\ 
      \midrule
      $s>a$& $\mathfrak{s.s.}$  & $(\log n)^{\frac{1}{2\gamma}}$ & $(\log n)^{-\frac{s-a}{\gamma}}+(\log m)^{-\frac{s-a}{\gamma}}$\\ [1ex] 
      \bottomrule
    \end{tabular}
    \caption{}\label{tab:table3}
  \end{table}
 Considering additive deconvolution  in
\cite{Johannes2013} and \cite{Neumann1997} it has been shown that
the rates in \cref{tab:table3}  are actually minimax-optimal.
	\section{Data driven estimation}\label{sec: 4}
In this section we will provide a fully data-driven selection method
for $k$ based on the construction given in \cref{subsec:2.6} but
dismissing the knowledge of the error density $f^U$. A similar
approach has been considered for additive deconvolution problems for
instance in \cite{ComteLacour2011} and \cite{Johannes2013}. More precisely, we select  
\begin{equation}
	\khat:\in\argmin_{k\in\nset{k_n}}\left\{-\norm{\mathds{1}_{[-k,k]} \MelYhat\MelUhatrez\1_\mathfrak{M}}^2_{\IL^2(\rmv)}+2\sigmaYhat\pen_k^{\rmvhat}\right\}\\=\argmin_{k\in\nset{k_n}}\left\{-\norm{\MelXhat^k}^2_{\IL^2(\rmv)}+2\sigmaYhat\pen_k^{\rmvhat}\right\},
\end{equation}
where the penalisation term $\pen_k^{\rmvhat}$ depends on a \textit{random} density function $\rmvhat$, which is defined and specified later as well as the choice of $k_n\in\IN$.
In contrast to the selection \eqref{eq:khat} in \cref{subsec:2.6}, here we have replaced $\MelUrez$ by its empirical counterpart $\MelUhatrez$. Our aim is to analyse the global $\IL^2(\rmv)$-risk again, namely
\begin{equation}\label{eq:4_reminder_risk}
	\IE\left[\norm{\MelXhat^{\khat}-\MelX}_{\IL^2(\mathrm{v})}^2\right].
\end{equation}
Our upper bounds necessities also slightly stronger assumptions than \cref{ass:B1}, which we formulate next.
\addtocounter{ass}{1}
\begin{ass}\label{ass:B2}\leavevmode\\ In addition to
  \cref{ass:B1} and \cref{ass:A2}  let $\inf_{t\in[-1,1]}\{|\MelU(t)|\}\in\pRz$,\linebreak
   $f^U\in\IL_+^1(\rmx^{7(c-1)})\cap\IL_+^1(\rmx^{2(c-1)}|\log\rmx|^\gamma)$
  	for some $\gamma\in\pRz$, $f^Y\in\IL_+^1(\rmx^{16(c-1)})$, and
        such that there
        are $\eta_Y,\eta_U,\eta_X\in\Rz_{\geq1}$ satisfying
  \begin{enumerate}
  	\item[i)] $\eta_Y\geq\max\left\{\norm{f^Y}_{\IL_+^\infty(\rmx^{2c-1})},\norm{f^Y}_{\IL_+^1(\rmx^{16(c-1)})}\right\}$,
  	\item[ii)]
          $\eta_U\geq\max\left\{\norm{f^U}_{\IL_+^1(\rmx^{7c-1})}^{1/7},\norm{f^U}_{\IL_+^1(\rmx^{2(c-1)}|\log\rmx|^\gamma)}\right\}$,
        \item[iii)] $\eta_X\geq \max\left\{\norm{f^X}_{\IL_+^1(\rmx^{2(c-1)})},\norm{\MelX}_{\IL^2(\rmv)}\right\}$.
        \end{enumerate}
      We set $a_Y\in\pRz$ and $k_Y\in\IR_{\geq1}$ as in \cref{ass:A2}.\qed
\end{ass}
Motivated by the key argument in \eqref{eq:keyargument1} in case of a
known error density $f^U$, the next lemma provides an error bound when $f^U$ is unknown. Its proof can be found in \cref{app: 4}.
\begin{lem}[Error bound]\label{re:kg}Consider an arbitrary event  $\evO$ with complement $\evO^C$, and denote
	by  $\evA^C$ the complement of the event
	$\evA:=\{\sigma_Y^2\leq2\sigmaYhat\}$. Given $k_n\in\IR_{\geq 1}$ and 
	\begin{equation}\label{re:kg:khat}
		\hat{k}:\in\argmin_{k\in\nset{k_n}}\left\{-\norm{\mathds{1}_{[-k,k]} \MelYhat\MelUhatrez\mathds{1}_{\mathfrak{M}}}^2_{\IL^2(\rmv)}+2\sigmaYhat\pen_k^{\rmvhat}\right\}
	\end{equation}
	for any $\ko\in\nset{k_n}$ we have 
	\begin{align}\nonumber
		\norm{\MelXhat^{\hat{k}}-\MelX}^2_{\IL^2(\rmv)}   &\leq
		15\norm{\mathds{1}_{[-\ko,\ko]}
                                                                    (\MelYhat-\MelY)\MelUhatrez\mathds{1}_{\mathfrak{M}}}^2_{\IL^2(\rmv)}\\\nonumber
          &\phantom{=}+15\norm{\mathds{1}_{[-\ko,\ko]^C}\MelX}_{\IL^2(\rmv)}^2+           24\sigmaYhat\pen_{\ko}^{\rmvhat}\mathds{1}_{\evO}+6\norm{\MelX}_{\IL^2(\rmv)}^2\mathds{1}_{\evO^c}\\\nonumber
          &\phantom{=}+15\norm{(\MelU\MelUhatrez\mathds{1}_{\mathfrak{M}}-\mathds{1})\MelX^{k_n}}_{\IL^2(\rmv)}^2\\\nonumber
            &\phantom{=}+12\max_{k\in\nset{k_n}}\left\{\left(\norm{\mathds{1}_{[-k,k]} (\MelYhat-\MelY)\MelUhatrez\mathds{1}_{\mathfrak{M}}}_{\IL^2(\rmv)}^2-
          \tfrac{\sigma_{Y}^2}{4}\pen_{k}^{\rmvhat}\right)_+\right\}\\\label{re:kg:e}
          &\phantom{=}+3\norm{\mathds{1}_{[-\kn,\kn]}
			(\MelYhat-\MelY)\MelUhatrez\mathds{1}_{\mathfrak{M}}}_{\IL^2(\rmv)}^2
		(\mathds{1}_{\evO^c}+\mathds{1}_{\evA^c}).
	\end{align}
\end{lem}
In the sequel, we aim to apply the expectation on both sides of \eqref{re:kg:e} in order to derive an upper bound for the risk. Therefore, we need to control the expectation of
\begin{equation}\label{eq:maxrefo}
\max_{k\in\nset{k_n}}\left\{\left(\norm{\mathds{1}_{[-k,k]} (\MelYhat-\MelY)\MelUhatrez\mathds{1}_{\mathfrak{M}}}_{\IL^2(\rmv)}^2-
\tfrac{\sigma_{Y}^2}{4}\pen_{k}^{\rmvhat}\right)_+\right\}.
\end{equation}
In \cref{subsec:2.6} a similar term was controlled by introducing the density function $\rmvU:=|\MelUrez|^2\rmv$ and providing a concentration inequality in \cref{re:conc}. In contrast, we use in the sequel its empirical counterpart, the random density function $\rmvhat:=|\MelUhatrez\1_\frakM|^2\rmv$, which depends on the sample $\{U_j\}_{j\in\nset{m}}$ only. Then \eqref{eq:maxrefo} reads as
\begin{multline*}
	\max_{k\in\nset{k_n}}\left\{\left(\norm{\mathds{1}_{[-k,k]} (\MelYhat-\MelY)\MelUhatrez\mathds{1}_{\mathfrak{M}}}_{\IL^2(\rmv)}^2-
	\tfrac{\sigma_{Y}^2}{4}\pen_{k}^{\rmvhat}\right)_+\right\} \\= \max_{k\in\nset{k_n}}\left\{\left(\norm{\mathds{1}_{[-k,k]} (\MelYhat-\MelY)}_{\IL^2(\rmvhat)}^2-
	\tfrac{\sigma_{Y}^2}{4}\pen_{k}^{\rmvhat}\right)_+\right\}.
\end{multline*}
The next proposition provides a concentration inequality for the expectation of the quantity in the last display.
\begin{pr}[Concentration inequality]\label{lem:conc2} 
		Under \cref{ass:B2} for $\rmv$ and $\rmvhat:=|\MelUhatrez|^2\rmv$   define $\Delta_k^{\rmv}$ and $\delta_k^{\rmv}$ as well as  $\Delta_k^{\rmvhat}$ and $\delta_k^{\rmvhat}$ as in \eqref{re:conc:def}. We consider $k_n$, $n\in\IN$, defined by 
	\begin{equation}\label{def:kn}
		k_n:=\max\{k\in\nset{n}:k\Delta_{k}^{\rmv} \leq n\Delta_1^\rmv\}.
	\end{equation}
		We then have
	\begin{equation*}
		\IE\left[\max_{k\in\nset{k_n}}\left\{\left(\norm{\mathds{1}_{[-k,k]}
			(\MelYhat-\MelY)}^2_{\IL^2(\rmvhat)}-12\sigma_{Y}^2\Delta_k^{\rmvhat}\delta_k^{\rmvhat} k
		n^{-1}\right)_+\right\}\right]\leq
          \Const\cdot\eta_Y(1\vee\eta_U^2k_Y\Delta_{k_Y}^{\rmv})\cdot
          n^{-1}.
	\end{equation*}
\end{pr}
\begin{proof}[\cref{lem:conc2}]
  Since $\rmvhat:=|\MelUhatrez\1_\frakM|^2\rmv$ depends on the sample
  $\{U_j\}_{j\in\nset{m}}$ only,  we apply the law of total
  expectation leading to
  \begin{multline*}
    \IE\left[ \max_{k\in\nset{k_n}}\left\{\left(\norm{\mathds{1}_{[-k,k]} (\MelYhat-\MelY)}^2_{\IL^2(\rmvhat)}-\sigma_Y^2\frac{\pen^{\rmvhat}_k}{4}\right)_+\right\} \right] \\
    = \IE\left[\IE\left[
        \max_{k\in\nset{k_n}}\left\{\left(\norm{\mathds{1}_{[-k,k]}
              (\MelYhat-\MelY)}^2_{\IL^2(\rmvhat)}-\sigma_Y^2\frac{\pen^{\rmvhat}_k}{4}\right)_+\right\}
        \bigg| \{U_j\}_{j\in\nset{m}}\right] \right].
  \end{multline*}
  Evidently, conditioning on $\{U_j\}_{j\in\nset{m}}$ the density
  function $\rmvhat$ is deterministic, thus similar to the proof of
  \eqref{eq:proof:CI1} making again use of the decomposition
  \eqref{proof:conc:decomp} in \cref{rem:conc} (with $\rmw=\rmvhat$)
  and applying \cref{re:conc:1} we obtain
  \begin{align*}
    &\IE\left[\max_{k\in\nset{k_n}}\left\{\left(\norm{\mathds{1}_{[-k,k]}
            (\MelYhat-\MelY)}^2_{\IL^2(\rmvhat)}-12\sigma_{Y}^2\Delta_k^{\rmvhat}\delta_k^\rmh k
      n^{-1}\right)_+\right\}\bigg|\{U_j\}_{j\in\nset{m}}\right]\\
    &\hspace*{5ex}\leq n^{-1} \Const \bigg[\E\big[Y_1^{8(c-1)}\big]
      n^{-2}k_n\Delta_{k_n}^{\rmvhat
      }+n^{-4}k_n^2\Delta_{k_n}^{\rmvhat } \\
    &\hspace*{5ex}\phantom{=}+(1\vee\norm{f^Y}_{\IL^\infty(\rmx^{2c-1})}^2)(1+\norm{\mathds{1}_{[-k_Y,k_Y]}}_{\IL^2(\rmvhat)}^2)\sum_{k\in\nset{k_n}}\frac{1}{a_Y}\exp\big(\frac{-k}{a_Y}\big)\bigg].
  \end{align*}
  Observing that $\Delta_k^{\rmvhat}\leq(m\land n)\Delta_k^\rmv$ and
  exploiting $\sum_{k\in\Nz}\exp(-k/a_Y)\leq a_Y$ as well as the
  definition \eqref{def:kn} of $k_n\in\nset{n}$, we have
  \begin{multline*}
    \IE\left[\max_{k\in\nset{k_n}}\left\{\left(\norm{\mathds{1}_{[-k,k]}
            (\MelYhat-\MelY)}^2_{\IL^2(\rmvhat)}-12\sigma_{Y}^2\Delta_k^{\rmvhat}\delta_k^{\rmvhat} k
          n^{-1}\right)_+\right\}\bigg|\{U_j\}_{j\in\nset{m}}\right] \\ \leq n^{-1} \Const
    \bigg[(1\vee\E\big[Y_1^{8(c-1)}\big])\Delta_1^{\rmv} +(1\vee\norm{f^Y}_{\IL^\infty(\rmx^{2c-1})}^2)(1+\norm{\mathds{1}_{[-k_Y,k_Y]}}_{\IL^2(\rmvhat)}^2)\big)\bigg].
  \end{multline*}
  The last upper bound does not depend on the additional measurements
  $\{U_j\}_{j\in\nset{m}}$ up to the last norm, namely
  $\norm{\mathds{1}_{[-k_Y,k_Y]}}_{\IL^2(\rmvhat)}^2=\norm{\mathds{1}_{[-k_Y,k_Y]}\MelUhatrez\1_\frakM}_{\IL^2(\rmv)}^2$. Its
  expectation under the distribution of $\{U_j\}_{j\in\nset{m}}$ is
  bounded in \cref{lem:2} as follows
		\begin{equation*}
			\IE\left[\norm{\mathds{1}_{[-k_Y,k_Y]}}_{\IL^2(\rmvhat)}^2\right]\leq 4\left(1\lor\IE[U_1^{2(c-1)}]\right) \norm{\mathds{1}_{[-k_Y,k_Y]}}_{\IL^2(\rmv)}^2.
		\end{equation*}
                Computing the total expectation together with $\norm{\mathds{1}_{[-k_Y,k_Y]}}_{\IL^2(\rmv)}^2\leq
                2k_Y\Delta_{k_Y}^{\rmv}$, $\Delta_1^{\rmv}\leq
                k_Y\Delta_{k_Y}^{\rmv}$ and  the definition of $\eta_Y$ and $\eta_U$ leads to the claim.
\end{proof}
Following the argumentation in \cref{subsec:2.6}, we choose $k_n$ as in \eqref{def:kn} and $\pen_k^{\rmvhat}$ appropriately, namely for each $k\in\nset{k_n}$ by
\begin{equation}\label{eq:penvhat}
	\pen_k^{\rmvhat} :=24\Delta_k^{\rmvhat}\delta_k^{\rmvhat} k
	n^{-1},
\end{equation}
observing that $k_n\in\nset{n}$ by definition. Thus, the data-driven
dimension parameter is specified as follows,
\begin{equation*}
	\hat{k}:\in\argmin_{k\in\nset{k_n}}\left\{-\norm{\MelXhat^k}^2_{\IL^2(\rmvhat)}+48 \sigmaYhat  \Delta_k^{\rmvhat}\delta_k^{\rmvhat} k\right\}.
\end{equation*}	
For any $\ko\in\nset{k_n}$ we intend to apply \cref{re:kg} with the random set 
\begin{equation}\label{eq:mho}
	\mho:=\mho_{\ko}:=\left\{\sup_{t\in[-\ko,\ko]}\left|\MelUhat(t)\MelU(t)-1\right|\leq\frac{1}{3}\right\}\subseteq\left\{\sup_{t\in[-\ko,\ko]}\left|\MelUhat^\dagger(t)\MelU(t)\right|^2\leq\frac{9}{4}\right\}
\end{equation}
where its complement evidently satisfies
\begin{equation*}
	\mho_{\ko}^C=\left\{\exists t\in[-\ko,\ko]:\,\left|\MelUhat(t)-\MelU(t)\right|>\frac{1}{3}|\MelU(t)|\right\}.
\end{equation*}
The following lemma provides a first upper bound of the risk, which
follows directly by applying the expectation on both sides of
\cref{re:kg} as well as the concentration inequality in
\cref{lem:conc2}. The proof with all details can be found in
\cref{app: 4}.
\begin{lem}[Risk bound]\label{co:riskboundadaptive}
  Under \cref{ass:B2} for any $\ko\in\nset{k_n}$ we have
  \begin{align}\nonumber
    \IE\left[\norm{\MelXhat^{\hat{k}}-\MelX}^2_{\IL^2(\rmv)}\right]
    &\leq \Const\cdot(\eta_U^4\eta_X+\sigma_Y^2)
    \min_{k\in\nset{k_n}}\left\{
      \norm{\mathds{1}_{[-k,k]^C}\MelX}_{\IL^2(\mathrm{v})}^2 +
      \Delta_k^{\rmvU}\delta_k^{\rmvU} k n^{-1}\right\} 
    \\\nonumber
  &\phantom{=}+\Const\cdot\eta_U^4\norm{\MelX\big(1\vee|\MelU|^2m\big)^{-1/2}}_{\IL^2(\rmv)}^2+
  \Const\cdot\eta_Y(1\vee\eta_U^2k_Y\Delta_{k_Y}^{\rmv})\cdot
  n^{-1} \\\label{re:kg:e:co}&\phantom{=}+15\cdot\norm{\mathds{1}_{[-\ko,\ko]^C}\MelX}_{\IL^2(\mathrm{v})}^2+6\cdot(\sigma_Y^2\Delta_{1}^{\rmv}+\eta_X)\Pz(\mho_{\ko}^C).
  \end{align}
\end{lem}
It remains to bound the  probability of the event $\mho_{\ko}^C$,
which turns out to be rather involved. The proof of the upper bound is
again based on an inequality due to \cite{Talagrand1996} which in  the
form of \cref{lem:Talagrand:2} in the \cref{app: 4} for example is stated by
\cite{BirgeMassart1998} in equation (5.13) in Corollary 2. However,
its application necessitates to bound the expectation of the
supremum of a normalised Mellin function process which we establish
in the next lemma. Its proof follows along the lines of the proof of Theorem 4.1 in
  \cite{NeumannReiss2009} where a similar result for a  normalised
  characteristic function process is shown. The proof of
  \cref{lem:8maxi} is also postponed to the \cref{app: 4}.
\begin{pr}[Normalised Mellin function process]\label{lem:8maxi}
	Let $\{Z_j\}_{j\in\IN}$ be a family of i.i.d $\pRz$-valued random variables
	  and assume there exists a constant $\eta\in\Rz_{\geq1}$, such that $\eta\geq(\E[Z_1^{2\beta}])^{1/2}$ and
	$\eta\geq\E[Z_1^{2\beta}|\log(Z_1)|^\gamma]$ for some
	$\beta\in\Rz$ and $\gamma\in\pRz$. Define the normalised
	Mellin function process by
	\begin{equation}\label{lem:8maxi:e1}
		c_m(t):=\frac{1}{\sqrt{m}}\sum_{j\in\nset{m}}\left\{Z_{j}^{\beta+\iota2\pi
			t}-\IE\left[Z_j^{\beta+\iota2\pi
			t}\right]\right\},\quad\forall t\in\Rz,
	\end{equation}
	and for $\rho\in\pRz$ the density function  $\rmwbar:\IR\to(0,1]$ by 
	\begin{equation}\label{lem:8maxi:e2}
		\rmwbar(t):=\left(\log(e+|t|)\right)^{-\frac{1}{2}-\rho},\quad\forall t\in\Rz.
	\end{equation}
	Then there exists a constant $\Const(\eta,\rho)\in\Rz_{\geq1}$ only
	depending on $\eta$ and $\rho$, such that
	\begin{equation*}
		\sup_{m\in\IN}\left\{\IE\left[\norm{c_m}_{\IL^\infty(\rmwbar)}\right]\right\}\leq \Const(\eta,\rho).
	\end{equation*}
\end{pr}
Now, we are in the position  to state an upper bound  of the probability of
the event $\mho_{\ko}^C$, which is proven in the \cref{app: 4}.  
\begin{pr}\label{bound:event}
        Let
        \cref{ass:B2} be satisfied and let   $\rho\in\pRz$ be
	arbitrary but fixed.  Consider  the density function
	$\rmwbar:\IR\to(0,1]$ given in \eqref{lem:8maxi:e2} and let
	$\Const(\eta_U,\rho)\in\Rz_{\geq1}$ be given by \cref{lem:8maxi}
	(with $\beta=c-1$). Given the universal numerical constant
	$\Const_{\mathrm{tal}}\in\IR_{>0}$ determined by Talagrand's
	inequality in \cref{lem:Talagrand:2} we set 
	\begin{align}\nonumber
		&\tau_m:=\tau_m(\eta_U,\rho):= 2\eta_U\Const_{\mathrm{tal}}^{-1/2}(\log
		m)^{1/2}+2\Const(\eta_U,\rho)\quad\forall m\in\Nz\quad{and}\\\nonumber
		&m_{\circ}:=m_{\circ}(\MelU,\eta_U,\rho)\\\label{bound:event:mo}&:=\min\left\{m\in\Nz_{\geq3}:\,\inf_{t\in[-1,1]}\rmwbar(t)|\MelU(t)|\geq
		6\tau_m m^{-1/2} \wedge \Const_{\mathrm{tal}} \eta_U^2m\geq\log
		m\right\}.
	\end{align}
	For $m\in\nset{m_{\circ}}$ let  $k_m\in\pRz$ be arbitrary
        while   for $m\in\Nz_{\geq m_{\circ}}$ we set
	\begin{equation}\label{bound:event:km}
		k_m:=\sup\left\{k\in\Nz:\,\inf_{t\in[-k,k]}\rmwbar(t)|\MelU(t)|\geq
		6\tau_m m^{-1/2}\right\}, 
	\end{equation}
	where the defining set is not empty for all $m\in\Nz_{\geq
		m_{\circ}}$. For any $m\in\Nz$ consider the event
	\begin{equation*}
		\mho_{k_m}^C:=\left\{\exists t\in[-k_m,k_m]:\,\left|\MelUhat(t)-\MelU(t)\right|>\frac{1}{3}|\MelU(t)|\right\}.
	\end{equation*}
	then there is an universal numerical constant
	$\Const:=3+11\Const_{\mathrm{tal}}^{-3}\in\IR_{\geq 1}$ such that we have    
	\begin{equation*}
		\Pz(\mho_{k_m}^C)\leq (m_{\circ}^2\vee
                \Const\eta_U)\,m^{-2},\quad\forall m\in\Nz.
	\end{equation*}
\end{pr}
We can now formulate our main result.
\pagebreak
\begin{thm}[Main result]\label{dd:ub} Let
        \cref{ass:B2} be satisfied and let   $\rho\in\pRz$ be
	arbitrary but fixed. Consider $m_{\circ}$ as in
        \eqref{bound:event:mo} and for each $m,n\in\Nz$,  $k_n$ as in
        \eqref{def:kn} and $k_m$ as in
        \eqref{bound:event:km}. Then for all  $m,n\in\Nz$ we have
  \begin{align*}
    \IE\left[\norm{\MelXhat^{\hat{k}}-\MelX}^2_{\IL^2(\rmv)}\right]&
    \leq \Const\cdot(\eta_U^4\eta_X+\sigma_Y^2)
    \min_{k\in\nset{k_n}}\left\{
      \norm{\mathds{1}_{[-k,k]^C}\MelX}_{\IL^2(\mathrm{v})}^2 +
      \Delta_k^{\rmvU}\delta_k^{\rmvU} k n^{-1}\right\} 
    \\
  &\phantom{=}+\Const\cdot\eta_U^4\norm{\MelX\big(1\vee|\MelU|^2m\big)^{-1/2}}_{\IL^2(\rmv)}^2+
  \Const\cdot\eta_Y(1\vee\eta_U^2k_Y\Delta_{k_Y}^{\rmv})\cdot
  n^{-1} \\
  &\phantom{=}+15\cdot\norm{\mathds{1}_{[-k_m,k_m]^C}\MelX}_{\IL^2(\mathrm{v})}^2+\Const\cdot(\sigma_Y^2\Delta_{1}^{\rmv}+\eta_X)(m_{\circ}^2\vee \eta_U)\,m^{-1}.
\end{align*}
\end{thm}

\begin{proof}[\cref{dd:ub}]
  By combining \cref{co:riskboundadaptive} with $k_{\circ}:=k_n\wedge
  k_m$, the elementary bound
  \begin{equation}
    \norm{\mathds{1}_{[-\ko,\ko]^C}\MelX}_{\IL^2(\mathrm{v})}^2\leq
    \min_{k\in\nset{k_n}}\left\{
      \norm{\mathds{1}_{[-k,k]^C}\MelX}_{\IL^2(\mathrm{v})}^2 +
      \Delta_k^{\rmvU}\delta_k^{\rmvU} k n^{-1}\right\} +\norm{\mathds{1}_{[-k_m,k_m]^C}\MelX}_{\IL^2(\mathrm{v})}^2
  \end{equation}
  and $\Pz(\mho_{\ko}^C)\leq \Pz(\mho_{k_m}^C)\leq (m_{\circ}^2\vee \Const\eta_U)\,m^{-2}$
  we immediately obtain the claim, which completes the proof.
\end{proof}

\begin{rem}Let us briefly compare the upper bound for the risk of the 
  data-driven estimator with known and unknown error distribution
  given in \eqref{eq:oracleadaptive2} and in \cref{dd:ub},
  respectively. Up to the constants estimating the error distribution
  leads to the three additional terms
  \begin{equation}\label{rem:ub:addterms}
    \norm{\MelX\big(1\vee|\MelU|^2m\big)^{-1/2}}_{\IL^2(\rmv)}^2,\,    m^{-1},\text{ and }\norm{\mathds{1}_{[-k_m,k_m]^C}\MelX}_{\IL^2(\mathrm{v})}^2.
  \end{equation}
  Evidently,
  they depend all on the sample size $m$ only, and the second is
  negligible with respect to the first term. Moreover, the
  first term in \eqref{rem:ub:addterms} is already present in the upper risk bound
  \eqref{eq:oracle3} of the oracle estimator when estimating the error
  density too. Thus the third term in \eqref{rem:ub:addterms}
  characterises the prize we pay for selecting the dimension parameter
  fully data-driven while estimating the error density. Comparing the first
  and third term  in \eqref{rem:ub:addterms} it is not obvious to
  us which one is the leading term.  If there exists a constant
  $\Const\in\Rz_{>1}$ such that
  \begin{equation}\label{cond:MelU}
    [-k_m,k_m]^C\subset\{\rmwbar^2|\MelU|^2m(\log
    m)^{-1}<\Const\},\quad\forall m\in\Nz,
  \end{equation}
  then both, the first and third term in \eqref{rem:ub:addterms}, are bounded up to the constant $\Const$ by
  \begin{equation}\label{rem:ub:dd}\norm{\MelX\big(1\vee\rmwbar^2|\MelU|^2m/\log
      m\big)^{-1/2}}_{\IL^2(\rmv)}^2.
  \end{equation}
  Note that the additional
  condition \eqref{cond:MelU} is satisfied whenever $|\MelU|^2$ is monotonically
  decreasing. However, the term in \eqref{rem:ub:dd} might over
  estimated both, the first and third term in \eqref{rem:ub:addterms}.
  Take for example the situation $\mathfrak{o.s.}$ in
  \cref{tab:table4} below. In case $(s-a)>\gamma$ all three terms in \eqref{rem:ub:addterms} are
  of order $m^{-1}$ while the term in \eqref{rem:ub:dd} is of order
  $m^{-1}(\log m)^{2(1+\rho)}$. In contrast in the situation
  $\mathfrak{s.s.}$ the first and third term in
  \eqref{rem:ub:addterms} as well as the term in \eqref{rem:ub:dd} are of order $(\log m)^{-\frac{s-a}{\gamma}}$.
\end{rem}
Continuing the brief discussion in \cref{subsec:3.3} we assume an
unknown density $f\in\IW^s(L)$, where the regularity $s$ is specified
below. Regarding the Mellin transformation $\MelU$, we subsequently
assume again its \textit{ordinary smoothness} \eqref{eq:os} or
\textit{super smoothness} \eqref{eq:es}. In order to discuss the
convergences rates for the $\IL^2(\rmv)$-risk under these regularity
assumptions, we restrict ourselves again to the choice
$\rmv:=\rmt_c^a$ for $a\in\IR$, observing that $a = 0$ corresponds to
the global risk for estimating the density $f$ and $a=1$ corresponds
to estimating the survival function $\SX$ of $X$ as discussed
before. In \cref{tab:table4} below we state for the specifications
considered in \cref{tab:table1} the order of the maximal oracle risk again as well as  the upper
risk bound in \cref{dd:ub} as
$n,m\to\infty$, which follow immediately from elementary
computations. 
  \begin{table}[H]
  	\centering
    \begin{tabular}{||@{\hspace*{2pt}}c@{\hspace*{2pt}}c@{\hspace*{2pt}}c@{\hspace*{2pt}}c@{\hspace*{2pt}}||}
      \toprule
      &$\MelU$ & Maximal oracle
                              risk & Data-driven risk \\ [0.5ex] 
      \midrule
      \midrule
      \scriptsize$a\in(-\tfrac{1}{2}-\gamma,s)$&
                             \scriptsize$\mathfrak{o.s.}$
                &$n^{-\frac{2(s-a)}{2\gamma+2s+1}}+m^{-(\frac{s-a}{\gamma}\land 1)}$&$n^{-\frac{2(s-a)}{2\gamma+2s+1}}+\big(\tfrac{(\log m)^{2(1+\rho)}}{m}\big)^{(\frac{s-a}{\gamma})}\vee m^{-1}$\\ 
      \midrule
      \scriptsize$s>a$& \scriptsize$\mathfrak{s.s.}$   & $(\log n)^{-\frac{s-a}{\gamma}}+(\log m)^{-\frac{s-a}{\gamma}}$& $(\log n)^{-\frac{s-a}{\gamma}}+(\log m)^{-\frac{s-a}{\gamma}}$\\ [1ex] 
      \bottomrule
    \end{tabular}
    \caption{}\label{tab:table4}
  \end{table}
We note that in case $\mathfrak{s.s.}$ for all $s>a$ and in case
$\mathfrak{o.s.}$ for $s-a>\gamma$ the maximal oracle risk and the
data-driven risk coincide. In other words we do not pay an additional prize for
the data-driven selection of the dimension parameter. In case
$\mathfrak{s.s.}$ for $s-a\leq\gamma$ the rates differ, but the data
driven rate features only a deterioration by  a factor $(\log m)^{2(1+\rho)(a-s)/\gamma}$.

	\section{Numerical study}\label{sec: 5}
In this section we are going to illustrate the performance of the data-driven estimation procedure as presented in section \ref{sec: 4}. Eventually, we want to highlight the behaviour of $\hat{f}_{\khat}$ under the influence of an increase of additional measurements for different types of unknown probability functions $f$ of $X$. Particularly, we are interested in four different densities $f$ of $X$, we aim to estimate, namely
\begin{enumerate}
	\item[i)] Gamma - distribution, $\Gamma(q,p)$:
		\begin{equation*}
			f_1(x)=
				\frac{\displaystyle q^p}{\displaystyle\Gamma(p)}x^{p-1}\exp(-qx)\1_{\pRz}(x)
		\end{equation*}
		with $q=1$ and $p=3$.
	\item[ii)] Weibull - distribution, $\operatorname{Weib}(l,k)$:
		\begin{equation*}
			f_2(x) =  s\cdot k \cdot (s \cdot x)^{k-1} \exp(-(s \cdot x)^k) \1_{\pRz}(x)
		\end{equation*}
	with $s= 1 $ and $k = 3 $.
	\item[iii)] Beta - distribution, $\operatorname{Beta}(a,b)$:
		\begin{equation*}
			f_3(x) = \frac{1}{\operatorname{B}(a,b)} x^{a-1}(1-x)^{b-1}\1_{(0,1)}(x)
	\end{equation*}
	with $a =10$ and $b = 5$.
	\item[iv)] Log-Normal - distribution, $\operatorname{LogN}(\mu,\sigma^2)$:
		\begin{equation*}
		f_4(x)= \frac{1}{\sqrt{2\pi}\sigma x }\,\exp\Big( -\frac{(\ln(x)-\mu)^2}{2\sigma^2}\Big)\1_{\pRz}(x)
	\end{equation*}
	with $ \mu = 0$ and $\sigma^2 = 1$.
\end{enumerate}
Moreover, as the unknown error distribution of $U$, we consider a Pareto distribution, $\operatorname{Pareto}(l,x_{\min})$,
\begin{equation*}
	f^U(x)=	\frac{l x_{\min}^l}{x^{l+1}}\1_{[x_{\min},\infty)}(x)
\end{equation*}
with $l=1$ and $x_{\min} = 1$. Elementary computations show that this
choice of error distribution actually satisfies the ordinary
smoothness conditions \eqref{eq:os} with $\gamma = 1$. In a first
place we document how an increasing number of additional measurements
$m$ has an impact on the statistical behaviour of
$\hat{f}_{\khat}$. To do so, we consider first $f_1$ as target density
and generate a sample $\{Y_i\}_{i\in\nset{1000}}$ following the law of
$Y = X\cdot U$ with independent $X\sim\Gamma(1,3)$ and
$U\sim\operatorname{Pareto}(1,1)$.  Moreover, we have sampled
$m \in\{100,1000,4000\}$ additional observations of
$U\sim\operatorname{Pareto}(1,1)$. With those samples we have computed
the data driven choice $\khat$ and afterwards $\hat{f}_{\khat}$
according to the estimation strategies presented in section \ref{sec:
  3} and section \ref{sec: 4}, where we have chosen $c=\frac{1}{2}$
and $0.3$ as constant in the penalty. As note by several authors (see
for instance \cite{ComteRozenholcTaupin2006}) the constant $48$ in
\eqref{eq:penvhat}, though convenient for deriving the theory, is far
too large in practise. In order to capture the randomness, we have
repeated this procedure for $N=500$ Monte-Carlo iterations, meaning we
have computed a family of data driven spectral cut-off density
estimators $\{\hat{f}^j_{\khat_j}\}_{j\in\nset{N}}$. For a direct
comparison to the situation of knowing the error density $f^U$, we
have also computed a family $\{\tilde{f}^j_{\khat_j}\}_{j\in\nset{N}}$
of spectral cut-off density estimators of $f_1$ as presented in
section \ref{sec: 2} with $c=\frac{1}{2}$ and $0.6$ as constant in the
penalty. The true density $f_1$, the family of estimators, as well as
a point-wise computed median are depict in \cref{fig:plot1} -
\cref{fig:plot4}, where also the corresponding empirical mean
integrated squared errors ($\operatorname{eMISE}$) are stated. As the
theory indicates, the estimation becomes more accurate for an
increasing sample size $m=100$ to $1000$, while for $m = 1000$ to
$m = 4000$ there is no significant improvement. And the accuracy
corresponds nearly to the case of a known error density.  Secondly, we
illustrate the behaviour of $\hat{f}_{\khat}$ for a fixed number of
samples $n$ and $m$, but for different target densities of $X$, namely
for $f_1,\dots,f_4$. We fix $n=m=2000$ and sample $n$ observations of
$Y^{i}$, $i\in\nset{4}$ according to the relation $Y ^i = X^i\cdot U$,
where $X^{i}$ follows the law of $f_i$, $i\in\nset{4}$ and $U$ is
still $\operatorname{Pareto}(1,1)$ distributed. Again, given $m$
additional measurements as i.i.d copies of $U$ we compute afterwards
$\hat{k}$ as well as $\hat{f}_{\khat}$ as before, following the
definitions in section \ref{sec: 3} and section \ref{sec: 4} with
choosing again $c = \frac{1}{2}$ and constant $0.3$ for the
penalty. Repeating this procedure for $N:=500$ Monte-Carlo iterations
we obtain four families of estimators
$\{\hat{f}^{i,j}_{\khat_{i,j}}\}_{i\in\nset{4},j\in\nset{N}}$. The
results as well as the empirical mean integrated squared error can be
found in \cref{fig:plot5} - \cref{fig:plot8}.
\begin{figure}[H]
	\centering
	\begin{minipage}{.4\textwidth}
		\centering
		\includegraphics[width=0.8\linewidth]{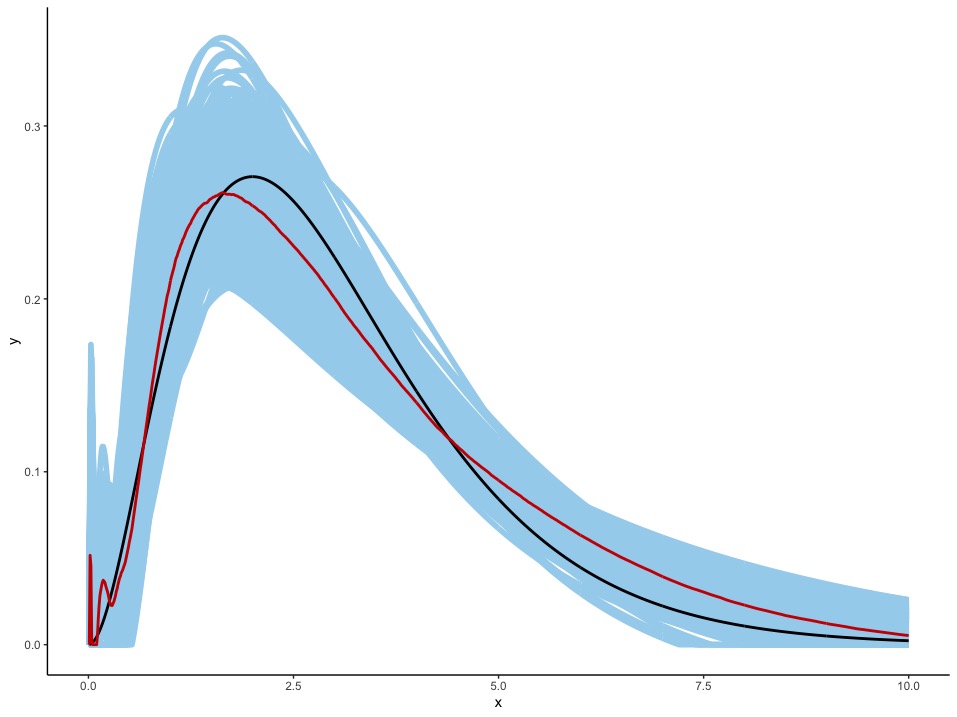}
		\caption{\phantom{.}Unknown error distribution, $n = 1000$, $m = 100$. Black line: true density $f_1$, Blue lines: $N$ Monte-Carlo estimations of $f_1$, Red Line: point-wise median. $\operatorname{eMISE} = 0.00575$.}
		\label{fig:plot1}
	\end{minipage}
	\hspace{0.5cm}
	\begin{minipage}{.4\textwidth}
		\centering
		\includegraphics[width=0.8\linewidth]{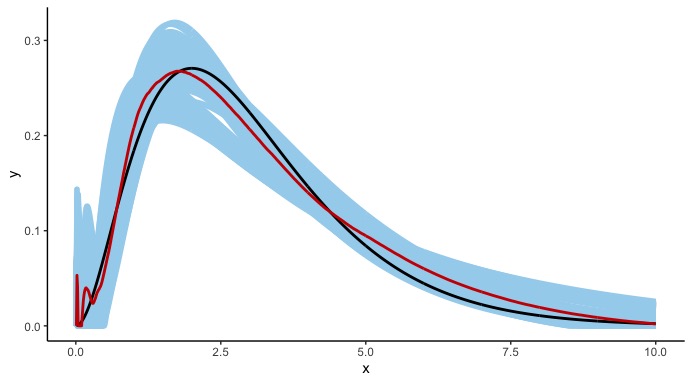}
		\caption{\phantom{.}Unknown error distribution, $n = 1000$, $m = 1000$. Black line: true density $f_1$, Blue lines: $N$ Monte-Carlo estimations of $f_1$, Red Line: point-wise median. $\operatorname{eMISE} = 0.00458$.}
		\label{fig:plot2}
	\end{minipage}
	\centering
	\begin{minipage}{.4\textwidth}
		\centering
		\includegraphics[width=0.8\linewidth]{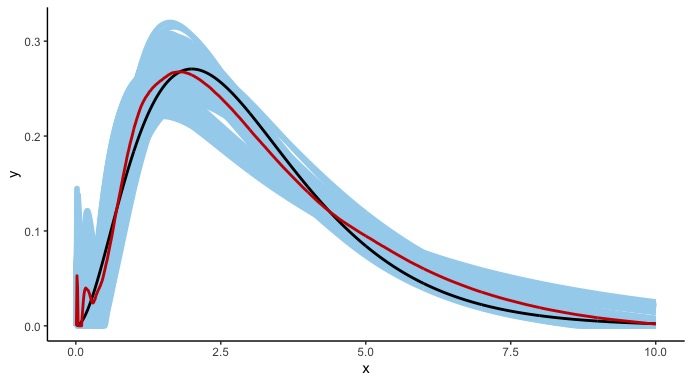}
		\caption{\phantom{.}Unknown error distribution, $n = 1000$, $m = 4000$. Black line: true density $f_1$, Blue lines: $N$ Monte-Carlo estimations of $f_1$, Red Line: Point-wise median. $\operatorname{eMISE} = 0.00449 $.}
		\label{fig:plot3}
	\end{minipage}
	\hspace{0.5cm}
	\begin{minipage}{.4\textwidth}
		\centering
		\includegraphics[width=0.8\linewidth]{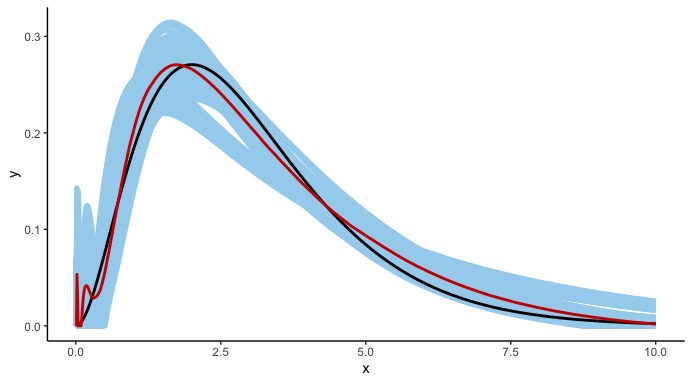}
		\caption{\phantom{.}Known error distribution, $n = 1000$. Black line: true  density $f_1$, Blue lines: $N$ Monte-Carlo estimations of $f_1$, Red Line: point-wise median. $\operatorname{eMISE} = 0.00299$.}
		\label{fig:plot4}
	\end{minipage}
	\centering
	\begin{minipage}{0.4\textwidth}
		\centering
		\includegraphics[width=0.8\linewidth]{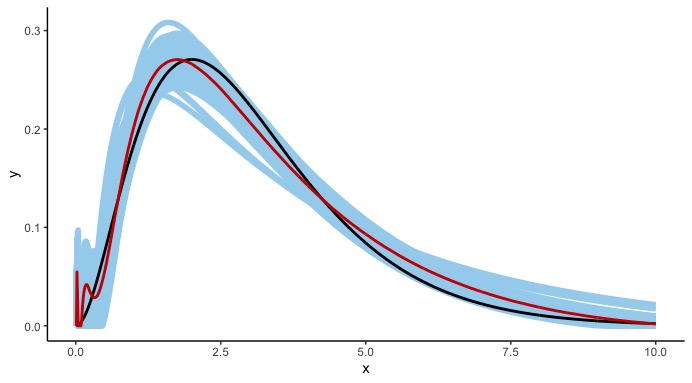}
		\caption{\phantom{.}Unknown error distribution, $n = m = 2000$. Black line: true density $f_1$, Blue lines: $N$ Monte-Carlo estimations of $f_1$, Red Line: point-wise median. $\operatorname{eMISE}= 0.00184 $.}
		\label{fig:plot5}
	\end{minipage}
	\hspace{0.5cm}
	\begin{minipage}{.4\textwidth}
		\centering
		\includegraphics[width=0.8\linewidth]{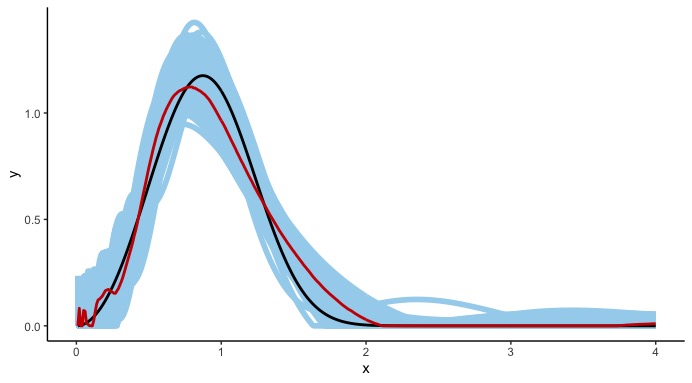}
		\caption{\phantom{.}Unknown error distribution, $n = m
                  = 2000$. Black line: true density $f_2$, Blue lines: $N$ Monte-Carlo estimations of $f_2$, Red Line: point-wise median. $\operatorname{eMISE}=0.0241  $.}
		\label{fig:plot6}
	\end{minipage}

	\centering
	\begin{minipage}{.4\textwidth}
		\centering
		\includegraphics[width=0.8\linewidth]{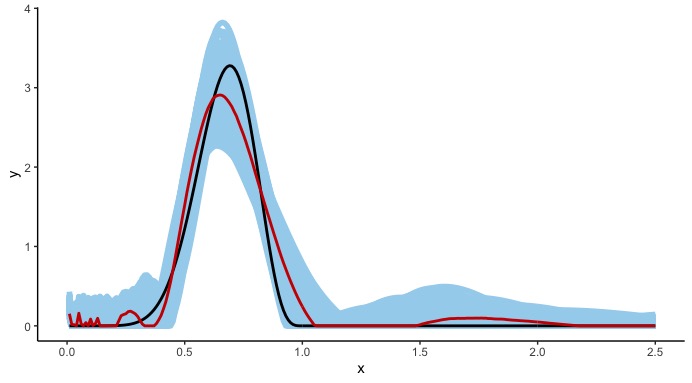}
		\caption{\phantom{.}Unknown error distribution, $n = m = 2000$. Black line: true density $f_3$, Blue lines: $N$ Monte-Carlo estimations of $f_3$, Red Line: point-wise median. $\operatorname{eMISE}=0.129$.}
		\label{fig:plot7}
	\end{minipage}
	\hspace{0.5cm}
	\begin{minipage}{.4\textwidth}
		\centering
		\includegraphics[width=0.8\linewidth]{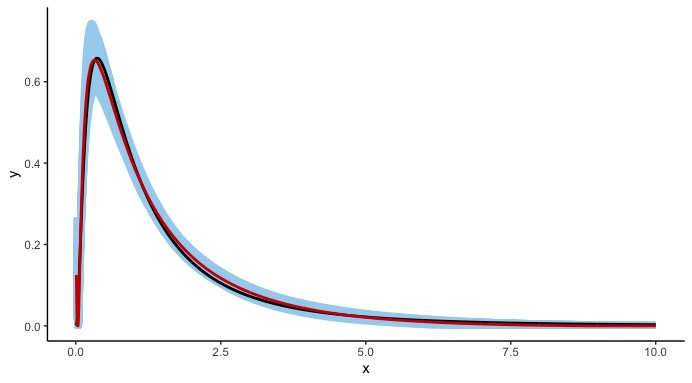}
		\caption{\phantom{.}Unknown error distribution, $n = m = 2000$. Black line: true density $f_4$, Blue lines: $N$ Monte-Carlo estimations of $f_4$, Red Line: point-wise median. $\operatorname{eMISE}= 0.00179$.}
		\label{fig:plot8}
	\end{minipage}
\end{figure}

		\pagebreak
		\Large{\scshape\centering{{Appendix}}}\\
		\normalsize{}
		
\section{Proofs of Section \ref{sec: 2}}\label{app: sec2}
The next assertion
provides our key argument in order to prove the concentration
inequality in \cref{re:conc}.  The  inequality is due to
\cite{Talagrand1996} and in this form for example given in \cite{KeinRio2005}.
\begin{lem}[Talagrand's inequality]\label{lem: Talagrand}  Let
	$(Z_i)_{i\in\nset{n}}$ be independent $\mathcal Z$-valued random
	variables and let $\{\nu_{h}:h\in\mathcal H\}$ be countable class of
	Borel-measurable functions. For $h\in\mathcal H$ setting
	$\overline{\nu}_{h}=n^{-1}\sum_{i\in\nset{n}}\left\{\nu_{h}(Z_i)-\E\left(\nu_{h}(Z_i)\right)
	\right\}$ we have
	\begin{align}
		\IE\left[\left(\sup_{h\in\mathcal H}\{|\overline{\nu}_h|^2\}-6\Psi^2\right)_+\right]\leq \Const_{\mathrm{tal}} \left[\frac{\tau}{n}\exp\left(\frac{-n\Psi^2}{6\tau}\right)+\frac{\psi^2}{n^2}\exp\left(\frac{-n \Psi}{100\psi}\right) \right]\label{tal:re1} 
	\end{align}
	for some universal numerical constants $\Const_{\mathrm{tal}}\in\IR_{>0}$  and where
	\begin{equation*}
		\sup_{h\in\mathcal H}\left\{\sup_{z\in\mathcal Z}\left\{|\nu_{h}(z)|\right\}\right\}\leq \psi,\qquad \E\left[\sup_{h\in \mathcal H}\{|\overline{\nu}_{h}|\}\right]\leq \Psi,\qquad \sup_{h\in \mathcal H}\left\{n\E\big[|\overline{\nu}_{h}|^2\big]\right\}\leq \tau.
	\end{equation*}
\end{lem}

\begin{rem}\label{rem:conc} Consider an arbitrary  density function $\rmw:\IR\to\IR_{\geq0}$  such that
	$\1_{[-k,k]}\in\IL^\infty(\rmw)$ for each $k\in\IR_{>0}$. For $k\in \IR_{>0}$ let us briefly reconsider
	the orthogonal projection
	$(\MelYhat-\MelY)\mathds{1}_{[-k,k]}$. Introduce the subset
	$\IS_k:=\{h\mathds{1}_{[-k,k]}:h\in \IL^2(\rmw)\}$ of $\IL^2(\rmw)$
	and 
	the  unit
	ball $\IB_{k}:=\{h\in\IS_k: \norm{h}_{\IL^2(\rmw)}\leq1\}$ in
	$\IS_k$. Clearly, setting  $\kappa:\IR_{>0}\times\IR\to\IC$ with $(y,t)\mapsto
	\kappa(y,t):=y^{c-1}y^{\iota2\pi t}$ we have
	$\MelYhat(t)=\tfrac{1}{n}\sum_{i\in\nset{n}}\kappa(Y_i,t)$ and
	$\MelY(t)=\E\big[\kappa(Y_i,t)\big]$ for each $t\in\IR$. Introducing $\kappa_y:\IR\to\IC$ with
	$t\mapsto\kappa_y(t):=\kappa(y,t)$, $y\in\IR_{>0}$,  for each $h\in\IS_k$ the function
	$\nu_h:\IR_{>0}\to\IC$ with  $y\mapsto\nu_h(y):=\skalar{\kappa_y,h}_{\IL^2(\rmw)}$ is
	Borel-measurable, where evidently
	$\tfrac{1}{n}\sum_{i\in\nset{n}}\nu_h(Y_i)=\skalar{\MelYhat,h}_{\IL^2(\rmw)}$ and
	hence
	$\overline{\nu}_{h}=\skalar{\MelYhat-\MelY,h}_{\IL^2(\rmw)}$. Consequently,
	we obtain
	\begin{equation}
		\norm{(\MelYhat-\MelY)\mathds{1}_{[-k,k]}}^2_{\IL^2(\rmw)}=\sup\{|\skalar{\MelYhat-\MelY,h}_{\IL^2(\rmw)}|^2:h\in\IB_{k}\}=
		\sup\{|\overline{\nu}_{h}|^2:h\in\IB_{k}\}.
	\end{equation}
	Note that, the unit ball
	$\IB_{k}$ is not a countable set, however, it contains a
	countable dense subset, say $\mathcal B_{k}$, since $\IL^2(\rmw)$ is
	separable. Exploiting the continuity of the inner product it is straightforward to see that
	\begin{equation*}
		\norm{(\MelYhat-\MelY)\mathds{1}_{[-k,k]}}^2_{\IL^2(\rmw)}=\sup\{|\overline{\nu}_{h}|^2:h\in\IB_{k}\}=
		\sup\{|\overline{\nu}_{h}|^2:h\in\mathcal B_{k}\}.
	\end{equation*}
	The last identity provides the necessary argument to apply below
	Talagrand's inequality (\ref{lem: Talagrand}) where we need to
	calculate  the three constants $\psi$, $\Psi$ and $\tau$. We note
	that the function $\kappa:\IR_{>0}\times\IR\to\IC$  is not bounded. Therefore we
	decompose $\kappa=\kappa^{\mathrm{b}}+\kappa^{\mathrm{u}}$ into a
	bounded function
	$\kappa^{\mathrm{b}}: \IR_{>0}\times\IR\to\IC$ with $(y,t)\mapsto
	\kappa^{\mathrm{b}}(y,t):=y^{c-1}\mathds{1}_{(0,d)}(y^{c-1})y^{\iota2\pi
		t}$ and an unbounded function $\kappa^{\mathrm{u}}: \IR_{>0}\times\IR\to\IC$ with $(y,t)\mapsto
	\kappa^{\mathrm{u}}(y,t):=y^{c-1}\mathds{1}_{[d,\infty)}(y^{c-1})y^{\iota2\pi
		t}$  where $d:=n^{1/3}$. For
	$\mathrm{j}\in\{\mathrm{b},\mathrm{u}\}$ setting
	$\MelYhat^{\mathrm{j}}(t)=\tfrac{1}{n}\sum_{i\in\nset{n}}\kappa^{\mathrm{j}}(Y_i,t)$ and
	$\MelY^{\mathrm{j}}(t)=\E\big[\kappa^{\mathrm{j}}(Y_i,t)\big]$ for
	each $t\in\IR$ it follows $\MelYhat-\MelY=\MelYhat^{\mathrm{b}}-\MelY^{\mathrm{b}}+\MelYhat^{\mathrm{u}}-\MelY^{\mathrm{u}}$.
	Introducing further $\nu_h^{\mathrm{b}}:\IR_{>0}\to\IC$ with
	$y\mapsto\nu_h^{\mathrm{b}}(y):=\skalar{\kappa_y^{\mathrm{b}},h}_{\IL^2(\rmw)} $  and
	$\nu_h^{\mathrm{u}}:\IR_{>0}\to\IC$ with  $y\mapsto\nu_h^{\mathrm{u}}(y):=\skalar{\kappa_y^{\mathrm{u}},h}_{\IL^2(\rmw)} $ we
	evidently have
	$\overline{\nu}_h=\overline{\nu}_h^{\mathrm{b}}+\overline{\nu}_h^{\mathrm{u}}$
	and thus
	\begin{align}\nonumber
		\norm{(\MelYhat-\MelY)\mathds{1}_{[-k,k]}}^2_{\IL^2(\rmw)}&=
		\sup\{|\overline{\nu}_{h}|^2:h\in\mathcal B_{k}\}=\sup\{|\overline{\nu}_h^{\mathrm{b}}+\overline{\nu}_h^{\mathrm{u}}|^2:h\in\mathcal B_{k}\}\\\nonumber
		&\hspace*{5ex}\leq
		2\sup\{|\overline{\nu}_{h}^{\mathrm{b}}|^2:h\in\mathcal B_{k}\}
		+
          2\sup\{|\overline{\nu}_{h}^{\mathrm{u}}|^2:h\in\mathcal
          B_{k}\}\hfill\\\label{proof:conc:decomp}
          &\hspace*{5ex}=2\norm{(\MelYhat^{\mathrm{b}}-\MelY^{\mathrm{b}})\mathds{1}_{[-k,k]}}^2_{\IL^2(\rmw)}+2\norm{(\MelYhat^{\mathrm{u}}-\MelY^{\mathrm{u}})\mathds{1}_{[-k,k]}}^2_{\IL^2(\rmw)}.
	\end{align}
	In \cref{re:conc:1} below we bound the
	expectation of the first
	term on the right hand side with the help of Talagrand's inequality (\cref{lem: Talagrand})
	and the  expectation of the second term.\qed
\end{rem}

\begin{lem}\label{re:conc:1} Let $\rmw:\IR\to\IR_{\geq0}$ be a density function, such that
  $\1_{[-k,k]}\in\IL^\infty(\rmw)$ for each $k\in\IR_{>0}$. Setting
	$\sigma_{Y}^2:=1+\E[Y_1^{2(c-1)}]$, 
	$a_Y:=6\norm{f^Y}_{\IL^\infty(\rmx^{2c-1})}/\sigma_Y^2$,
        $k_Y:=1\vee3a_Y^2$, 
	\begin{equation*}
		\Delta_k^{\rmw} := \norm{\1_{[-k,k]}}_{\IL^\infty(\rmw)}\text{
			and }\delta_k^\rmw
		:=
		\frac{\log(\Delta_k^\rmw\lor(k+2))}{\log(k+2)}\in\IR_{\geq1},\quad\forall
		k\in\IR_{\geq1},
	\end{equation*}
	there exists an universal numerical constant $\Const\in\IR_{>0}$  such that for all $n\in\IN$ and $k\in\IR_{\geq1}$  we have
	\begin{multline}\label{re:conc:1:e}
		\IE\left[\left(\norm{\mathds{1}_{[-k,k]}
			(\MelYhat^{\mathrm{b}}-\MelY^{\mathrm{b}})}^2_{\IL^2(\rmw)}-
		6\sigma_{Y}^2\Delta_k^{\rmw}\delta_k^\rmw k
		n^{-1}\right)_+\right] \\\leq  n^{-1}\times \Const\times
		\bigg[
		(1\vee\norm{f^Y}_{\IL^\infty(\rmx^{2c-1})}^2)(1+\norm{\mathds{1}_{[-k_Y,k_Y]}}_{\IL^2(\rmw)}^2)\frac{1}{a_Y}\exp\big(\frac{-k}{a_Y}\big)+n^{-4}k\Delta_k^\rmw\bigg]
	\end{multline}
	and $  \IE\left[\norm{\mathds{1}_{[-k,k]}
		(\MelYhat^{\mathrm{u}}-\MelY^{\mathrm{u}})}^2_{\IL^2(\rmw)}\right]\leq2\E\big[Y_1^{8(c-1)}\big]
	n^{-3}k\Delta_{k}^\rmw$.
\end{lem}
\vspace*{2ex}
\begin{proof}[\cref{re:conc:1}]
	Consider first the second claim. For each $t\in\IR$ we have $\E\big[\MelYhat^{\mathrm{u}}(t)\big]=\MelY^{\mathrm{u}}(t)=\E\big[\kappa^{\mathrm{u}}(Y_1,t)\big]$  which in
	turn implies
	\begin{align*}
		n\IE\big[
		|\MelYhat^{\mathrm{u}}(t)-\MelY^{\mathrm{u}}(t)|^2\big]&=\Var\big(\kappa^{\mathrm{u}}(Y_1,t)\big)\leq
		\E\big[|\kappa^{\mathrm{u}}(Y_1,t)|^2\big]= \E\big[Y_1^{2(c-1)}
		\mathds{1}_{[d,\infty)}(Y_1^{c-1})]\leq d^{-2l}\E\big[Y_1^{2(c-1)(1+l)}\big]
	\end{align*}
	and thus making use of  $d=n^{1/3}$
	we obtain the claim, that is
	\begin{align*}
		\IE\left[\norm{\mathds{1}_{[-k,k]}
			(\MelYhat^{\mathrm{u}}-\MelY^{\mathrm{u}})}^2_{\IL^2(\rmw)}\right]&\leq
		\tfrac{1}{nd^{6}}\E\big[Y_1^{2(c-1)(1+3)}\big]\norm{\mathds{1}_{[-k,k]}}^2_{\IL^2(\rmw)}\leq
		n^{-3}\E\big[Y_1^{8(c-1)}\big]2k\Delta_{k}^\rmw.
	\end{align*}
	Secondly consider \eqref{re:conc:1:e}. Given the identity (see Remark
	\ref{rem:conc})
	\begin{equation*}
		\norm{(\MelYhat^{\mathrm{b}}-\MelY^{\mathrm{b}})\mathds{1}_{[-k,k]}}^2_{\IL^2(\rmw)}=\sup\{|\overline{\nu}_{h}^{\mathrm{b}}|^2:h\in\mathcal B_{k}\}
	\end{equation*}
	we intent to apply Talagrand's inequality (\cref{lem: Talagrand}) where we need to
	calculate  the three quantities $\psi$, $\Psi$ and $\tau$. Consider
	$\psi$ first. Evidently, for each $k\in\IR_{\geq1}$ we have
	\begin{multline*}
		\sup_{h\in \mathcal  B_k}\left\{\sup_{y\in\IR_{>0}}\left\{|\nu_{h}^{\mathrm{b}}(y)|^2\right\}\right\}=
		\sup_{y\in\IR_{>0}}\left\{\sup_{h\in\mathcal B_k}\left\{|\skalar{\kappa_y^{\mathrm{b}},h}_{\IL^2(\rmw)}|^2\right\}\right\}\\=
		\sup_{y\in\IR_{>0}}\left\{\norm{\kappa_y^{\mathrm{b}}\mathds{1}_{[-k,k]}}_{\IL^2(\rmw)}^2\right\}\leq d^2\norm{\mathds{1}_{[-k,k]}}_{\IL^2(\rmw)}^2\leq2d^2k\Delta_k^\rmw=: \psi^2.
	\end{multline*}
	Consider next
	$\Psi$. Evidently, for each $t\in\IR$ we have
	\begin{equation*}
		n\IE\big[
		|\MelYhat^{\mathrm{b}}(t)-\MelY^{\mathrm{b}}(t)|^2\big]=\Var\big(\kappa^{\mathrm{b}}(Y_1,t)\big)\leq
		\E\big[|\kappa^{\mathrm{b}}(Y_1,t)|^2\big]\leq \E\big[Y_1^{2(c-1)}\big]\leq\sigma_Y^2
	\end{equation*}
	which for each
	$k\in\IR_{\geq1}$ implies
	\begin{align*}
		\E\left[\sup_{h\in \mathcal H}|\overline{\nu}_{h}^{\mathrm{b}}|^2\right]
		&= \E\left[\norm{(\MelYhat^{\mathrm{b}}-\MelY^{\mathrm{b}})\mathds{1}_{[-k,k]}}^2_{\IL^2(\rmw)}\right]
		\leq
		\tfrac{1}{n}\sigma_Y^2\norm{\mathds{1}_{[-k,k]}}_{\IL^2(\rmw)}^2\leq
		\tfrac{1}{n}\sigma_Y^22k\Delta_k^{\rmw}\delta_k^\rmw=:  \Psi^2.
	\end{align*}
	Finally, consider $\tau$. For each $h\in\mathcal B_k$ we have
	\begin{align*}
		n\E\big[|\overline{\nu}_{h}^{\mathrm{b}}|^2\big]=\Var\big[\skalar{\kappa^{\mathrm{b}}(Y_1,\cdot),h}_{\IL^2(\rmw)}\big]&\leq
		\E\big[|\skalar{\kappa^{\mathrm{b}}(Y_1,\cdot),h}_{\IL^2(\rmw)}|^2\big]\leq
		\E\big[|\skalar{\kappa(Y_1,\cdot),h}_{\IL^2(\rmw)}|^2\big]=\E\big[|\nu_h|^2\big].
	\end{align*}
	Since  $h\rmw\in\IL^2\cap\IL^1$ and
	$|\nu_h(y)|^2=|\skalar{h,\kappa_y}_{\IL^2(\rmw)}|^2=|\M_{1-c}^{-1}[h\rmw](y)|^2$
	for all $y\in\IR_{>0}$ we have $\nu_h\in\IL^2(\rmx^{1-2c})$. Consequently, using Parseval's identity
	we obtain
	\begin{equation*}
		\norm{\nu_h}_{\IL^2(\rmx^{1-2c})}^2=\norm{\M_{1-c}[\nu_h]}_{\IL^2}^2=\norm{\rmw
			h}_{\IL^2}^2\leq \norm{\rmw\mathds{1}_{[-k,k]}}_{\IL^\infty}\norm{h}_{\IL^2(\rmw)}^2\leq\Delta_k^\rmw
	\end{equation*}
	which in turn for each $h\in\mathcal B_k$ implies
	\begin{equation*}
		\E\big[|\nu_h|^2\big] \leq \norm{f^Y}_{\IL^\infty(\rmx^{2c-1})}\norm{\nu_h}_{\IL^2(\rmx^{1-2c})}^2\leq \norm{f^Y}_{\IL^\infty(\rmx^{2c-1})}\Delta_k^\rmw.
	\end{equation*}
	On the other hand side for each $h\in\mathcal B_k$ we have  also
        \begin{equation*}
	\E\big[|\nu_h|^2\big]\leq
	\norm{\mathds{1}_{[-k,k]}}_{\IL^2(\rmw)}^2\E[Y_1^{2(c-1)}]\leq
	\norm{\mathds{1}_{[-k,k]}}_{\IL^2(\rmw)}^2\sigma_Y^2,
      \end{equation*}
      and hence 
	\begin{equation*}
		n\E\big[|\overline{\nu}_{h}^{\mathrm{b}}|^2\big]\leq
		\norm{f^Y}_{\IL^\infty(\rmx^{2c-1})}\Delta_k^\rmw \wedge \sigma_Y^2 \norm{\mathds{1}_{[-k,k]}}_{\IL^2(\rmw)}^2=:\tau.
	\end{equation*}
	Given $a_Y=6\norm{f^Y}_{\IL^\infty(\rmx^{2c-1})}/\sigma_Y^2$ and
	$k_Y=1\vee3a_Y^2$ for any $k\geq k_Y$ we
	have
	\begin{equation*}
		\Delta_k^\rmw\exp\left(\frac{-\sigma_Y^2\delta_k^{\rmw}k}{6\norm{f^Y}_{\IL^\infty(\rmx^{2c-1})}}\right)=\Delta_k^\rmw\exp\big(\tfrac{-\delta_k^{\rmw}}{a_Y}k\big)\leq \exp(-\tfrac{\delta_k^{\rmw}}{a_Y}(k-a_Y\log(k+2))\leq1
	\end{equation*}
	by exploiting that $x\geq a\log(x+2)$ for all $a>0$ and
	$x\geq1\vee3a^2$. Consequently, for any $k\geq k_Y$ we obtain (keep
	$\delta_k^{\rmw}\geq1$ in mind)
	\begin{align*}
		\frac{\tau}{n}\exp\left(\frac{-n\Psi^2}{6\tau}\right)&\leq
		\tfrac{\norm{f^Y}_{\IL^\infty(\rmx^{2c-1})}\Delta_k^\rmw}{n}\exp\left(\frac{-\sigma_Y^2\delta_k^{\rmw}k}{3\norm{f^Y}_{\IL^\infty(\rmx^{2c-1})}}\right)\leq \tfrac{\norm{f^Y}_{\IL^\infty(\rmx^{2c-1})}}{n}\exp\big(\frac{-k}{a_Y}\big)
	\end{align*}
	while for any $k< k_Y$ we conclude (using again $\delta_k^{\rmw}\geq1$)
	\begin{align*}
		\frac{\tau}{n}\exp\left(\frac{-n\Psi^2}{6\tau}\right)&\leq
		\tfrac{\sigma_Y^2 \norm{\mathds{1}_{[-k,k]}}_{\IL^2(\rmw)}^2}{n}\exp\left(\frac{-\sigma_Y^2\delta_k^{\rmw}k}{3\norm{f^Y}_{\IL^\infty(\rmx^{2c-1})}}\right)\leq \tfrac{\sigma_Y^2 \norm{\mathds{1}_{[-k_Y,k_Y]}}_{\IL^2(\rmw)}^2}{n}\exp\big(\frac{-k}{a_Y}\big).
	\end{align*}
	Combining both cases $k\geq k_Y$ and $k< k_Y$ we obtain for any $k\in\IR_{\geq1}$
	\begin{equation*}
		\frac{\tau}{n}\exp\left(\frac{-n\Psi^2}{6\tau}\right)\leq\tfrac{\norm{f^Y}_{\IL^\infty(\rmx^{2c-1})}+\sigma_Y^2 \norm{\mathds{1}_{[-k_Y,k_Y]}}_{\IL^2(\rmw)}^2}{n}\exp\big(\frac{-k}{a_Y}\big).
	\end{equation*} 
	Evaluating the bound given by Talagrand's inequality (\ref{lem:
		Talagrand}) there exists an universal numerical constant
	$\Const_{\mathrm{tal}}\in\IR_{>0}$  such that  for each $k\in\IR_{\geq1}$ we have
	(keep $d=n^{1/3}$, $\sigma_Y^2\geq1$ and $\delta_k^\rmw\geq1$ in mind)
	\begin{align*}
		\IE\left[\left(\norm{\mathds{1}_{[-k,k]}
			(\MelYhat^{\mathrm{b}}-\MelY^{\mathrm{b}})}^2_{\IL^2(\rmw)}-
		6\sigma_{Y}^2\Delta_k^{\rmw}\delta_k^\rmw k
          n^{-1}\right)_+\right]
          \hspace*{-40ex}&\\
                         &\leq  n^{-1} \Const_{\mathrm{tal}}
		\bigg[n^{-1/3}k\Delta_k^\rmw\exp\left(\frac{-n^{1/6}}{100}\right) 
		\\&\phantom{=}+(\norm{f^Y}_{\IL^\infty(\rmx^{2c-1})}+\sigma_Y^2
		\norm{\mathds{1}_{[-k_Y,k_Y]}}_{\IL^2(\rmw)}^2)\exp\big(\frac{-k}{a_Y}\big)\bigg]\\
		&\leq  n^{-1} \Const
		\bigg[n^{-4}k\Delta_k^\rmw  n^{11/3}\exp\left(\frac{-n^{1/6}}{100}\right) 
		\\&\phantom{=}+(1\vee\norm{f^Y}_{\IL^\infty(\rmx^{2c-1})}^2)(1+\norm{\mathds{1}_{[-k_Y,k_Y]}}_{\IL^2(\rmw)}^2)\frac{1}{a_Y}\exp\big(\frac{-k}{a_Y}\big)\bigg],
	\end{align*}
	which together with   $n^b\exp(-an^{1/c})\leq (\tfrac{cb}{ae})^{cb}$ for all
	$a,b,c\in\IR_{>0}$, and hence $n^{11/3}\exp(\frac{-n^{1/6}}{100})\leq (1100)^{22}$  shows \eqref{re:conc:1:e} and completes the proof.
\end{proof}

\section{Proofs of Section \cref{sec: 3}}\label{app: sec3}
\begin{lem}\label{lem:2}
  There exists an universal numerical constant $\Const\in\IR_{\geq1}$
  such that for any $k\in\IR_{>0}$ we have
     \begin{align*}
     &\IE\left[\norm{\mathds{1}_{[-k,k]}\MelUhat^\dagger\mathds{1}_{\evM}}_{\IL^2(\rmv)}^2\right]
    \leq 4\big(1\lor\IE[U_1^{2(c-1)}]\big)\norm{\mathds{1}_{[-k,k]}\MelU^\dagger}_{\IL^2(\rmv)}^2,\\
      &\IE\left[\norm{\mathds{1}_{[-k,k]}\MelX\mathds{1}_{\evM^C}}_{\IL^2(\rmv)}^2\right]
      \leq 4\big(1\lor\IE[U_1^{2(c-1)}]\big)\norm{\mathds{1}_{[-k,k]}\MelX\big(1\vee|\MelU|^2(m\land n)\big)^{-1/2}}_{\IL^2(\rmv)}^2, \\
      &\IE\left[\norm{\mathds{1}_{[-k,k]}\MelX(\MelU-\MelUhat)\MelUhat^\dagger\mathds{1}_{\evM}}_{\IL^2(\rmv)}^2\right]
      \leq 4\big(1\lor\Const\,\IE[U_1^{4(c-1)}]\big) \norm{\mathds{1}_{[-k,k]}\MelX\big(1\vee|\MelU|^2m\big)^{-1/2}}_{\IL^2(\rmv)}^2.
   \end{align*}
\end{lem}
\begin{proof}[\cref{lem:2}]
  We start our proof with the observation that for each $t\in\IR$ we
  have $\IE[\MelUhat(t)]=\MelU(t)$,
  $m\IE\big[|\MelUhat(t)-\MelU(t)|^2\big]\leq \IE[U_1^{2(c-1)}]$, and
  $m^2\IE\big[|\MelUhat(t)-\MelU(t)|^4\big]\leq
  \Const\,\IE[U_1^{4(c-1)}]$ for some universal numerical constant
  $\Const\in\IR_{\geq1}$ by applying Theorem 2.10 in \cite{Petrov1995}. We
  use those bounds without further reference. Below we show that for
  each $t\in\IR$ we have
  \begin{align}\label{lem:2:p1}
    &\IE\left[|\MelU(t)\MelUhatrez(t)|^2\mathds{1}_{\evM}(t)\right]
      \leq 4\big(1\lor\IE[U_1^{2(c-1)}]\big)\\\nonumber
    &\IE\left[\mathds{1}_{\evM^C}(t)\right]=\IP\left((n\land m)|\MelUhat(t)|^2<1\right) \leq 4\big(1\lor\IE[U_1^{2(c-1)}]\big)\big(1\vee|\MelU(t)|^2(m\land
      n)\big)^{-1}\\
    &\IE\left[|\MelU(t)-\MelUhat(t)|^2|\MelUhat^\dagger(t)|^2\mathds{1}_{\evM}(t)\right]
      \leq 4\big(1\lor\Const\,\IE[U_1^{4(c-1)}])\big) (1\vee|\MelU(t)|^2m)^{-1}\label{lem:2:p3}
  \end{align}
  Evidently, by applying Fubini’s theorem from the last bounds follow
  immediatly \cref{lem:2}. It remains
  to show (\eqref{lem:2:p1}-\eqref{lem:2:p3}). Let $t\in\IR$ be fixed.
  Consider \eqref{lem:2:p1} first.  Exploiting
  $|\MelUhat(t)\MelUhatrez(t)|^2\mathds{1}_{\evM}(t)=\mathds{1}_{\evM}(t)$
  and $2|\MelUhat(t)|^2+2|\MelU(t)-\MelUhat(t)|^2\geq |\MelU(t)|^2$ we
  have
  \begin{align*}
    \IE\left[|\MelU(t)\MelUhatrez(t)|^2\mathds{1}_{\evM}(t)\right]&\leq 2
    \IE\left[|\MelU(t)-\MelUhat(t)|^2|\MelUhatrez(t)|^2\mathds{1}_{\evM}(t)
      + \mathds{1}_{\evM}(t)\right] \\&\leq 2\left((m\land
      n)\IE[|\MelU(t)-\MelUhat(t)|^2]+1\right)\\&\leq
    2\left(\IE[U_1^{2(c-1)}]+1\right)
    \leq 4\left(1\lor\IE[U_1^{2(c-1)}]\right)
  \end{align*}
  which shows \eqref{lem:2:p1}. Secondly, \eqref{lem:2:p3} is
  trivially satisfied if
  $1\leq 4(1\lor\IE[U_1^{2(c-1)}])(1\lor |\MelU(t)|^2(n\land
  m))^{-1}$. Otherwise, from
  $(1\lor |\MelU(t)|^2(n\land m))> 4(1\lor\IE[U_1^{2(c-1)}])\geq 4$
  follows $(n\land m)^{-1}<|\MelU(t)|^2/4$ which using Markov's
  inequality implies \eqref{lem:2:p2}, that is
  \begin{align*}
    \IP\left(|\MelUhat(t)|^2<(n\land m)^{-1}\right)&\leq
    \IP\left(|\MelUhat(t)-\MelU(t)|>|\MelU(t)|/2\right)\\&\leq
    4\IE\big[|\MelUhat(t)-\MelU(t)|^2\big]|\MelU(t)|^{-2}\\
    &\leq 4 \IE[U_1^{2(c-1)}] (|\MelU(t)|^2m)^{-1}\\
    &\leq   4(1\lor\IE[U_1^{2(c-1)}])(1\lor |\MelU(t)|^2(n\land m))^{-1}.
  \end{align*}
  Finally, consider \eqref{lem:2:p3}. Evidently, we have on the
  one hand 
  \begin{equation*}
    \IE\left[|\MelU(t)-\MelUhat(t)|^2|\MelUhat^\dagger(t)|^2\mathds{1}_{\evM}(t)\right]\leq
    (n\land m)m^{-1}  \IE[U_1^{2(c-1)}]\leq \IE[U_1^{2(c-1)}]
  \end{equation*}
  while on the other hand  
  \begin{multline*}
    \IE\left[|\MelU(t)-\MelUhat(t)|^2|\MelUhat^\dagger(t)|^2\mathds{1}_{\evM}(t)\right]\\\leq
    2\IE\left[|\MelU(t)-\MelUhat(t)|^2|\MelUhat^\dagger(t)|^2\mathds{1}_{\evM}(t)\big(\frac{|\MelU(t)-\MelUhat(t)|^2}{|\MelU(t)|^2}+\frac{|\MelUhat(t)|^2}{|\MelU(t)|^2}\big)\right]\\
    \leq 2 \frac{(m\land n)\IE\big[|\MelU(t)-\MelUhat(t)|^4\big]}{|\MelU(t)|^2}+2 \frac{\IE\big[|\MelU(t)-\MelUhat(t)|^2\big]}{|\MelU(t)|^2}\\
    \leq 2(\Const \IE[U_1^{4(c-1)}]+ \IE[U_1^{2(c-1)}]) (|\MelU(t)|^2m)^{-1}.
  \end{multline*}
  Combining both bounds we obtain \eqref{lem:2:p3}, which completes the
  proof.
\end{proof}

		\section{Proofs of Section \cref{sec: 4}}\label{app: 4}
We first recall an inequality due to \cite{Talagrand1996} which in  this
form  for example is stated by
\cite{BirgeMassart1998} in equation (5.13) in Corollary 2. We make use
of it in the proof of \cref{bound:event} below.
\begin{lem}[Talagrand's inequality]\label{lem:Talagrand:2}  Let
	$(Z_i)_{i\in\nset{m}}$ be independent and identically distributed $\mathcal Z$-valued random
	variables and let $\{\nu_{t}:t\in\mathcal T\}$ be countable class of
	Borel-measurable functions. For $t\in\mathcal T$ setting
        \begin{equation*}
	\overline{\nu}_{t}=m^{-1}\sum_{i\in\nset{m}}\left\{\nu_{t}(Z_i)-\E\left(\nu_{t}(Z_i)\right) \right\}
      \end{equation*}
 we have
	\begin{align}
		\Pz\left[\sup_{t\in\mathcal T}|\overline{\nu}_t|\geq
		2\Psi+\kappa\right]\leq 3
		\exp\left(-\Const_{\mathrm{tal}} m\big(\frac{\kappa^2}{\tau}\wedge\frac{\kappa}{\psi}\big)\right) 
	\end{align}
	for some universal numerical constant $\Const_{\mathrm{tal}}\in\IR_{>0}$  and where
	\begin{equation*}
		\sup_{t\in\mathcal T}\left\{\sup_{z\in\mathcal Z}\left\{|\nu_{t}(z)|\right\}\right\}\leq \psi,\qquad \E\left[\sup_{t\in \mathcal T}|\overline{\nu}_{t}|\right]\leq \Psi,\qquad \sup_{t\in \mathcal T}\left\{\E\big[|\nu_{t}(Z_1)|^2\big]\right\}\leq\tau.
	\end{equation*}
\end{lem}
\vspace*{2ex}
\begin{proof}[\cref{re:kg}]
	Let $\ko\in\mathcal{K}:=\nset{k_n}$ be arbitrary but fixed. Introduce
	$\MelXhat:=\MelYhat\MelUhatrez\mathds{1}_{\mathfrak{M}}$ and
	$\MelXcheck:=\MelY\MelUhatrez\mathds{1}_{\mathfrak{M}}=\MelX\MelU\MelUhatrez\mathds{1}_{\mathfrak{M}}$, for each
	$k\in\IR_{>0}$ we write shortly   $\MelXhat^k:=\MelXhat
	\mathds{1}_{[-k,k]}$,   $\MelXcheck^k:=\MelXcheck
	\mathds{1}_{[-k,k]}$ and $\MelX^k:=\MelX \mathds{1}_{[-k,k]}$. Consider the disjoint decomposition
	$\cK=\cK[<\ko]\cup\cK[\geq\ko]$ where
	$\cK[<\ko]:=\{k\in\cK:k<\ko\}$ and
	$\cK[\geq\ko]:=\{k\in\cK:k\geq\ko\}$, and similarly $\cK=\cK[\leq\ko]\cup\cK[>\ko]$. Evidently, we have
	\begin{equation*}
		\mathds{1}_{[-\khat,\khat]} \MelYhat\MelUhatrez\mathds{1}_{\mathfrak{M}}-\MelX=\MelXhat^{\khat}-\MelXcheck^{\khat}+(\MelXcheck^{\khat}-\MelXcheck^{\ko})\mathds{1}_{\khat\in\cK[<\ko]}+(\MelXcheck^{\ko}-\MelX)\mathds{1}_{\khat\in\cK[<\ko]}+(\MelXcheck^{\khat}-\MelX)\mathds{1}_{\khat\in\cK[\geq\ko]}
	\end{equation*}
	which in turn implies
	\begin{multline*}   
		\norm{\MelXhat^{\khat}- \MelX}_{\IL^2(\rmv)}^2\leq3\norm{\MelXhat^{\khat}-\MelXcheck^{\khat}}_{\IL^2(\rmv)}^2+3\norm{\MelXcheck^{\khat}-\MelXcheck^{\ko}}_{\IL^2(\rmv)}^2\mathds{1}_{\khat\in\cK[<\ko]}\\+3\big(\norm{\MelXcheck^{\ko}-\MelX}_{\IL^2(\rmv)}^2\mathds{1}_{\khat\in\cK[<\ko]}+\norm{\MelXcheck^{\khat}-\MelX}_{\IL^2(\rmv)}^2\mathds{1}_{\khat\in\cK[\geq\ko]}\big).
	\end{multline*}
	Combining the last bound and the elementary estimate (keep in mind
	that $\kn:=\max\cK$)
	\begin{equation*}
		\norm{\MelXcheck^{\ko}-\MelX}_{\IL^2(\rmv)}^2\mathds{1}_{\khat\in\cK[<\ko]}+\norm{\MelXcheck^{\khat}-\MelX}_{\IL^2(\rmv)}^2\mathds{1}_{\khat\in\cK[\geq\ko]}\leq
		\norm{\MelXcheck^{k_n}-\MelX^{k_n}}_{\IL^2(\rmv)}^2+\norm{\MelX^{\ko}-\MelX}_{\IL^2(\rmv)}^2
	\end{equation*}
	we have
	\begin{multline*}   
		\norm{\MelXhat^{\khat}-
			\MelX}_{\IL^2(\rmv)}^2\leq3\norm{\MelXhat^{\khat}-\MelXcheck^{\khat}}_{\IL^2(\rmv)}^2+3\norm{\MelXcheck^{\khat}-\MelXcheck^{\ko}}_{\IL^2(\rmv)}^2\mathds{1}_{\khat\in\cK[<\ko]}\\+3\norm{\MelXcheck^{k_n}-\MelX^{k_n}}_{\IL^2(\rmv)}^2+3\norm{\MelX^{\ko}-\MelX}_{\IL^2(\rmv)}^2
	\end{multline*}
	which together with the estimate
	\begin{align*}
		\norm{\MelXhat^{\khat}-\MelXcheck^{\khat}}_{\IL^2(\rmv)}^2&=\norm{\MelXhat^{\khat}-\MelXcheck^{\khat}}_{\IL^2(\rmv)}^2(\mathds{1}_{\khat\in\cK[\leq\ko]}+\mathds{1}_{\khat\in\cK[>\ko]})\leq \norm{\MelXhat^{\ko}-\MelXcheck^{\ko}}_{\IL^2(\rmv)}^2+\norm{\MelXhat^{\khat}-\MelXcheck^{\khat}}_{\IL^2(\rmv)}^2\mathds{1}_{\khat\in\cK[>\ko]}
	\end{align*}
	implies the upper bound
	\begin{align}\nonumber
		\norm{\MelXhat^{\khat}-	\MelX}^2&\leq3\norm{\MelXhat^{\ko}-\MelXcheck^{\ko}}_{\IL^2(\rmv)}^2+3\norm{\MelX-\MelX^{\ko}}_{\IL^2(\rmv)}^2+3\norm{\MelXcheck^{\kn}-\MelX^{k_n}}_{\IL^2(\rmv)}^2\\\label{re:kg:p1}&\phantom{=}+
		3
		\norm{\MelXhat^{\khat}-\MelXcheck^{\khat}}_{\IL^2(\rmv)}^2\mathds{1}_{\khat\in\cK[>\ko]}
		+3\norm{\MelXcheck^{\khat}-\MelXcheck^{\ko}}_{\IL^2(\rmv)}^2\mathds{1}_{\khat\in\cK[<\ko]}.
	\end{align}
	We bound separately the last two terms on the right hand side in \eqref{re:kg:p1}
	next. Consider
	$\norm{\MelXcheck^{\khat}-\MelXcheck^{\ko}}_{\IL^2(\rmv)}^2\mathds{1}_{\khat\in\cK[<\ko]}$
	first. Introduce the random decomposition $\cK[<\ko]=\cK[-]\cup\cK[-]^c$
	with index set
	\begin{equation}\label{re:kg:p-}
		\cK[-]:=\{k\in\cK[<\ko]:
		\norm{\MelXcheck^{\ko}-\MelXcheck^{k}}_{\IL^2(\rmv)}^2>8\sigmaYhat\pen_{\ko}^{\rmvhat}\}
	\end{equation}
	and its complement $\cK[-]^c:=\cK[<\ko]\setminus\cK[-]$. If
	$\cK[-]^c\ne\emptyset$ then for each $k\in\cK[-]^c$ we have
	\begin{equation*}
		\norm{\MelXcheck^{\ko}-\MelXcheck^{k}}_{\IL^2(\rmv)}^2\leq
		\norm{\MelXcheck^{\kn}}_{\IL^2(\rmv)}^2\leq 2\norm{\MelXcheck^{\kn}-\MelX^{\kn}}_{\IL^2(\rmv)}^2+2\norm{\MelX}_{\IL^2(\rmv)}^2
	\end{equation*}
	and also
	$\norm{\MelXcheck^{\ko}-\MelXcheck^{k}}_{\IL^2(\rmv)}^2\leq8\sigmaYhat\pen_{\ko}^{\rmvhat}$,
	which together imply
	\begin{equation*}
		\norm{\MelXcheck^{\khat}-\MelXcheck^{\ko}}_{\IL^2(\rmv)}^2\mathds{1}_{\khat\in\cK[-]^c}\leq \big(2\norm{\MelXcheck^{\kn}-\MelX^{\kn}}_{\IL^2(\rmv)}^2+2\norm{\MelX}_{\IL^2(\rmv)}^2\mathds{1}_{\evO^c}+8\sigmaYhat\pen_{\ko}^{\rmvhat}\mathds{1}_{\evO}\big)\mathds{1}_{\khat\in\cK[-]^c}.
	\end{equation*}
	If $\cK[-]\ne\emptyset$ then for each $k\in\cK[-]$ we have
	\begin{align*}
		\tfrac{1}{2}\norm{\MelXcheck^{\ko}-\MelXcheck^{k}}_{\IL^2(\rmv)}^2&\leq
		\norm{(\MelXcheck-\MelXhat)(\mathds{1}_{[-\ko,\ko]}-\mathds{1}_{[-k,k]})}_{\IL^2(\rmv)}^2+\norm{\MelXhat^{\ko}-\MelXhat^{k}}_{\IL^2(\rmv)}^2\\&\leq
		\norm{\MelXcheck^{\ko}-\MelXhat^{\ko}}_{\IL^2(\rmv)}^2+\norm{\MelXhat^{\ko}}_{\IL^2(\rmv)}^2-\norm{\MelXhat^{k}}_{\IL^2(\rmv)}^2
	\end{align*}
	which  using for the last estimate the definition
	\eqref{re:kg:khat} and \eqref{re:kg:p-} of $\khat$ and
	$\cK[-]$, respectively, implies
	\begin{align*}
		\norm{\MelXcheck^{\khat}-\MelXcheck^{\ko}}_{\IL^2(\rmv)}^2\mathds{1}_{\khat\in\cK[-]}
                                                                                                  		&\leq 4
		\norm{\MelXcheck^{\ko}-\MelXhat^{\ko}}_{\IL^2(\rmv)}^2\mathds{1}_{\khat\in\cK[-]}+4\big(\norm{\MelXhat^{\ko}}_{\IL^2(\rmv)}^2-\norm{\MelXhat^{\khat}}_{\IL^2(\rmv)}^2-2\sigmaYhat\pen_{\khat}^{\rmvhat}\big)\mathds{1}_{\khat\in\cK[-]}\\&\phantom{=}+4\big(2\sigmaYhat\pen_{\khat}^{\rmvhat}-\tfrac{1}{4}\norm{\MelXcheck^{\ko}-\MelXcheck^{\khat}}_{\IL^2(\rmv)}^2)\mathds{1}_{\khat\in\cK[-]}\\
		&\leq 4
		\norm{\MelXcheck^{\ko}-\MelXhat^{\ko}}_{\IL^2(\rmv)}^2\mathds{1}_{\khat\in\cK[-]}.
	\end{align*}
	Combining both cases $\khat\in\cK[-]$ and $\khat\in\cK[-]^c$  we obtain the bound 
	\begin{align}\nonumber
		\norm{\MelXcheck^{\khat}-\MelXcheck^{\ko}}_{\IL^2(\rmv)}^2\mathds{1}_{\khat\in\cK[<\ko]}&\leq \mathds{1}_{\khat\in\cK[<\ko]}\bigg(4
		\norm{\MelXcheck^{\ko}-\MelXhat^{\ko}}_{\IL^2(\rmv)}^2+2\norm{\MelXcheck^{\kn}-\MelX^{\kn}}_{\IL^2(\rmv)}^2\\\label{re:kg:p2}&\phantom{=}+2\norm{\MelX}_{\IL^2(\rmv)}^2\mathds{1}_{\evO^c}+8\sigmaYhat\pen_{\ko}^{\rmvhat}\mathds{1}_{\evO}\bigg).
	\end{align}
	Secondly, consider the term
	$\norm{\MelXhat^{\khat}-\MelXcheck^{\khat}}_{\IL^2(\rmv)}^2\mathds{1}_{\khat\in\cK[>\ko]}$
	in \eqref{re:kg:p1}. Introduce the random decomposition $\cK[>\ko]=\cK[+]\cup\cK[+]^c$
	with index set 
	\begin{equation}\label{re:kg:p+}
		\cK[+]:=\{k\in\cK[>\ko]:
		\norm{\MelXhat^{k}-\MelXcheck^{k}}_{\IL^2(\rmv)}^2>2\sigmaYhat\pen_{\ko}^{\rmvhat}\}
	\end{equation}
	and its complement $\cK[+]^c:=\cK[>\ko]\setminus\cK[+]$. If
	$\cK[+]^c\ne\emptyset$ then for each $k\in\cK[+]^c$ we have
	$\norm{\MelXhat^{k}-\MelXcheck^{k}}_{\IL^2(\rmv)}^2\leq
	\norm{\MelXhat^{\kn}-\MelXcheck^{\kn}}_{\IL^2(\rmv)}^2$ and
	$\norm{\MelXhat^{k}-\MelXcheck^{k}}_{\IL^2(\rmv)}^2\leq2\sigmaYhat\pen_{\ko}^{\rmvhat}$,
	which together imply
	\begin{equation*}
		\norm{\MelXhat^{\khat}-\MelXcheck^{\khat}}_{\IL^2(\rmv)}^2\mathds{1}_{\khat\in\cK[+]^c}\leq \big(\norm{\MelXhat^{\kn}-\MelXcheck^{\kn}}_{\IL^2(\rmv)}^2\mathds{1}_{\evO^c}+2\sigmaYhat\pen_{\ko}^{\rmvhat}\mathds{1}_{\evO}\big)\mathds{1}_{\khat\in\cK[+]^c}.
	\end{equation*}
	If $\cK[+]\ne\emptyset$ then for each $k\in\cK[+]$ we have 
	\begin{multline*}
		-\norm{\MelXhat^{\ko}}_{\IL^2(\rmv)}^2+\norm{\MelXhat^{k}}_{\IL^2(\rmv)}^2 - 2 \norm{\MelXhat^{k}-\MelXcheck^{k}}_{\IL^2(\rmv)}^2-2\norm{\MelXcheck^{k}-\MelXcheck^{\ko}}_{\IL^2(\rmv)}^2
		\\\hfill= \norm{\MelXhat^{k}-\MelXhat^{\ko}}_{\IL^2(\rmv)}^2  - 2
		\norm{\MelXhat^{k}-\MelXcheck^{k}}_{\IL^2(\rmv)}^2-2\norm{\MelXcheck^{k}-\MelXcheck^{\ko}}_{\IL^2(\rmv)}^2\\\hfill
		\leq \norm{\MelXhat^{k}-\MelXhat^{\ko}}_{\IL^2(\rmv)}^2-\norm{\MelXhat^{k}-\MelXhat^{\ko}}_{\IL^2(\rmv)}^2=0
	\end{multline*}
	which together with the definition \eqref{re:kg:khat} of $\khat$  implies
	\begin{multline*}
		\mathds{1}_{\khat\in\cK[+]}\big(- 2 \norm{\MelXhat^{\khat}-\MelXcheck^{\khat}}_{\IL^2(\rmv)}^2
		-2\sigmaYhat\pen_{\ko}^{\rmvhat}  +2\sigmaYhat\pen_{\khat}^{\rmvhat}
		-2\norm{\MelXcheck^{\khat}-\MelXcheck^{\ko}}_{\IL^2(\rmv)}^2\big)
		\\\leq \mathds{1}_{\khat\in\cK[+]}\big(\norm{\MelXhat^{\ko}}_{\IL^2(\rmv)}^2-
		2\sigmaYhat\pen_{\ko}^{\rmvhat} +2\sigmaYhat\pen_{\khat}^{\rmvhat}  - \norm{\MelXhat^{\khat}}_{\IL^2(\rmv)}^2\big)\leq0.
	\end{multline*}
	From  the last elementary bound we obtain (keep
	the definition \eqref{re:kg:p+} of $\cK[+]$ and $\evA:=\{\sigma_Y^2\leq2\sigmaYhat\}$ in mind)
	\begin{align*}
		\norm{\MelXhat^{\khat}-\MelXcheck^{\khat}}_{\IL^2(\rmv)}^2
		\mathds{1}_{\khat\in\cK[+]}\mathds{1}_{\evA}
	&=\mathds{1}_{\khat\in\cK[+]}\mathds{1}_{\evA}
		\big(4\norm{\MelXhat^{\khat}-\MelXcheck^{\khat}}_{\IL^2(\rmv)}^2- 2\sigmaYhat\pen_{\khat}^{\rmvhat}+2\norm{\MelXcheck^{\khat}-\MelXcheck^{\ko}}_{\IL^2(\rmv)}^2\big)\\
		&+\mathds{1}_{\khat\in\cK[+]}\mathds{1}_{\evA}
		\big(-\norm{\MelXhat^{\khat}-\MelXcheck^{\khat}}_{\IL^2(\rmv)}^2+2\sigmaYhat\pen_{\ko}^{\rmvhat}\big)\\
		&+\mathds{1}_{\khat\in\cK[+]}\mathds{1}_{\evA}
		\big(-2\norm{\MelXhat^{\khat}-\MelXcheck^{\khat}}_{\IL^2(\rmv)}^2+2\sigmaYhat\pen_{\khat}^{\rmvhat}\\&\hspace*{40ex}-2\sigmaYhat\pen_{\ko}^{\rmvhat}-2\norm{\MelXcheck^{\khat}-\MelXcheck^{\ko}}_{\IL^2(\rmv)}^2\big)\\
		&\leq \mathds{1}_{\khat\in\cK[+]}\mathds{1}_{\evA}
		\big(4\norm{\MelXhat^{\khat}-\MelXcheck^{\khat}}_{\IL^2(\rmv)}^2- 2\sigmaYhat\pen_{\khat}^{\rmvhat}                                                                                    +2\norm{\MelXcheck^{\khat}-\MelXcheck^{\ko}}_{\IL^2(\rmv)}^2\big)\\
		&\leq \mathds{1}_{\khat\in\cK[+]}\mathds{1}_{\evA}
		\big(4\norm{\MelXhat^{\khat}-\MelXcheck^{\khat}}_{\IL^2(\rmv)}^2- \sigma_{Y}^2\pen_{\khat}^{\rmvhat}+2\norm{\MelXcheck^{\khat}-\MelXcheck^{\ko}}_{\IL^2(\rmv)}^2\big)\\
		&\leq \mathds{1}_{\khat\in\cK[+]}
		\bigg(4\max_{k\in\cK[+]}\big\{\big(\norm{\MelXhat^{k}-\MelXcheck^{k}}_{\IL^2(\rmv)}^2- \tfrac{\sigma_{Y}^2}{4}\pen_{k}^{\rmvhat}\big)_+\big\}+2\norm{\MelXcheck^{\kn}-\MelXcheck^{\ko}}_{\IL^2(\rmv)}^2\bigg).
	\end{align*}
	which together with the elementary bounds 
        \begin{equation*}
	\norm{\MelXcheck^{\kn}-\MelXcheck^{\ko}}_{\IL^2(\rmv)}^2\leq 2 \norm{\MelXcheck^{\kn}-\MelX^{\kn}}_{\IL^2(\rmv)}^2 +  2\norm{\MelX^{\ko}-\MelX}_{\IL^2(\rmv)}^2
      \end{equation*}
 and  $\norm{\MelXhat^{\khat}-\MelXcheck^{\khat}}_{\IL^2(\rmv)}^2
	\mathds{1}_{\khat\in\cK[+]}\mathds{1}_{\evA^c}\leq\norm{\MelXhat^{\kn}-\MelXcheck^{\kn}}_{\IL^2(\rmv)}^2\mathds{1}_{\khat\in\cK[+]}\mathds{1}_{\evA^c}$
	implies
	\begin{align*}
		\norm{\MelXhat^{\khat}-\MelXcheck^{\khat}}_{\IL^2(\rmv)}^2 \mathds{1}_{\khat\in\cK[+]}&\leq\mathds{1}_{\khat\in\cK[+]}\bigg(\norm{\MelXhat^{\kn}-\MelXcheck^{\kn}}_{\IL^2(\rmv)}^2\mathds{1}_{\evA^c}
		\\&\hspace*{5ex}+ 4\max_{k\in\cK[+]}\left\{\big(\norm{\MelXhat^{k}-\MelXcheck^{k}}_{\IL^2(\rmv)}^2-
		\tfrac{\sigma_{Y}^2}{4}\pen_{k}^{\rmvhat}\big)_+\right\}\\&\hspace*{5ex}+4 \norm{\MelXcheck^{\kn}-\MelX^{\kn}}_{\IL^2(\rmv)}^2 +  4\norm{\MelX^{\ko}-\MelX}_{\IL^2(\rmv)}^2\bigg).
	\end{align*}
	Combining both cases $\khat\in\cK[+]$ and $\khat\in\cK[+]^c$
        we obtain the bound 
	\begin{align}\nonumber
		\norm{\MelXhat^{\khat}-\MelXcheck^{\khat}}_{\IL^2(\rmv)}^2\mathds{1}_{\khat\in\cK[>\ko]}\hspace*{-27ex}&\\\nonumber&
		\leq
		\mathds{1}_{\khat\in\cK[+]^c}\big(\norm{\MelXhat^{\kn}-\MelXcheck^{\kn}}_{\IL^2(\rmv)}^2\mathds{1}_{\evO^c}+2\sigmaYhat\pen_{\ko}^{\rmvhat}\mathds{1}_{\evO}\big) \\\nonumber
		&+  \mathds{1}_{\khat\in\cK[+]}\bigg(\norm{\MelXhat^{\kn}-\MelXcheck^{\kn}}_{\IL^2(\rmv)}^2\mathds{1}_{\evA^c}
		+ 4\max_{k\in\cK[+]}\big\{\big(\norm{\MelXhat^{k}-\MelXcheck^{k}}_{\IL^2(\rmv)}^2-
		\tfrac{\sigma_{Y}^2}{4}\pen_{k}^{\rmvhat}\big)_+\big\}\\\nonumber&\hspace*{10ex}+4
		\norm{\MelXcheck^{\kn}-\MelX^{\kn}}_{\IL^2(\rmv)}^2 +
		4\norm{\MelX^{\ko}-\MelX}_{\IL^2(\rmv)}^2\bigg)\\\nonumber
		&\leq\mathds{1}_{\khat\in\cK[>\ko]}
		\bigg(\norm{\MelXhat^{\kn}-\MelXcheck^{\kn}}_{\IL^2(\rmv)}^2
		(\mathds{1}_{\evO^c}+\mathds{1}_{\evA^c}) +
		2\sigmaYhat\pen_{\ko}^{\rmvhat}\mathds{1}_{\evO} \\\nonumber
		&\hspace*{10ex}+4\max_{k\in\cK[+]}\big\{\big(\norm{\MelXhat^{k}-\MelXcheck^{k}}_{\IL^2(\rmv)}^2-
		\tfrac{\sigma_{Y}^2}{4}\pen_{k}^{\rmvhat}\big)_+\big\}\\\label{re:kg:p3}&\hspace*{10ex}+4
		\norm{\MelXcheck^{\kn}-\MelX^{\kn}}_{\IL^2(\rmv)}^2 +
		4\norm{\MelX^{\ko}-\MelX}_{\IL^2(\rmv)}^2\bigg).
	\end{align}
	Making use of \eqref{re:kg:p2} and \eqref{re:kg:p3} we obtain
	\begin{align*}
		\norm{\MelXcheck^{\khat}-\MelXcheck^{\ko}}_{\IL^2(\rmv)}^2\mathds{1}_{\khat\in\cK[-]}
		+
          \norm{\MelXhat^{\khat}-\MelXcheck^{\khat}}_{\IL^2(\rmv)}^2\mathds{1}_{\khat\in\cK[>\ko]}
          \hspace*{-50ex}&\\
          &\leq
		\mathds{1}_{\khat\in\cK[<\ko]}\bigg(4
            \norm{\MelXcheck^{\ko}-\MelXhat^{\ko}}_{\IL^2(\rmv)}^2\\
                         &\hspace*{10ex}+2\norm{\MelXcheck^{\kn}-\MelX^{\kn}}_{\IL^2(\rmv)}^2+2\norm{\MelX}_{\IL^2(\rmv)}^2\mathds{1}_{\evO^c}+8\sigmaYhat\pen_{\ko}^{\rmvhat}\mathds{1}_{\evO}\bigg)\\
		&\phantom{=}+ \mathds{1}_{\khat\in\cK[>\ko]}
		\bigg(\norm{\MelXhat^{\kn}-\MelXcheck^{\kn}}_{\IL^2(\rmv)}^2
		(\mathds{1}_{\evO^c}+\mathds{1}_{\evA^c}) +
          2\sigmaYhat\pen_{\ko}^{\rmvhat}\mathds{1}_{\evO} \\
                         &
		\hspace*{10ex}+4\max_{k\in\cK[+]}\big\{\big(\norm{\MelXhat^{k}-\MelXcheck^{k}}_{\IL^2(\rmv)}^2-
		\tfrac{\sigma_{Y}^2}{4}\pen_{k}^{\rmvhat}\big)_+\big\}\\&\hspace*{10ex}+4
		\norm{\MelXcheck^{\kn}-\MelX^{\kn}}_{\IL^2(\rmv)}^2 +
		4\norm{\MelX^{\ko}-\MelX}_{\IL^2(\rmv)}^2\bigg)\\
		&\leq 4\norm{\MelX^{\ko}-\MelX}_{\IL^2(\rmv)}^2 + 4\norm{\MelXcheck^{\ko}-\MelXhat^{\ko}}_{\IL^2(\rmv)}^2+ 8\sigmaYhat\pen_{\ko}^{\rmvhat}\mathds{1}_{\evO}
		\\&\phantom{=}+4\norm{\MelXcheck^{\kn}-\MelX^{\kn}}_{\IL^2(\rmv)}^2+4\max_{k\in\cK[+]}\big\{\big(\norm{\MelXhat^{k}-\MelXcheck^{k}}_{\IL^2(\rmv)}^2-
		\tfrac{\sigma_{Y}^2}{4}\pen_{k}^{\rmvhat}\big)_+\big\}\\ &\phantom{=}+\norm{\MelXhat^{\kn}-\MelXcheck^{\kn}}_{\IL^2(\rmv)}^2
		(\mathds{1}_{\evO^c}+\mathds{1}_{\evA^c}) +2\norm{\MelX}_{\IL^2(\rmv)}^2\mathds{1}_{\evO^c}
	\end{align*}
	which together with \eqref{re:kg:p1} implies the claim
	\begin{align*}
		\norm{\MelXhat^{\khat}-
          \MelX}_{\IL^2(\rmv)}^2&\leq3\norm{\MelXhat^{\ko}-\MelXcheck^{\ko}}_{\IL^2(\rmv)}^2+3\norm{\MelX-\MelX^{\ko}}_{\IL^2(\rmv)}^2+3\norm{\MelXcheck^{\kn}-\MelX^{k_n}}_{\IL^2(\rmv)}^2\\&\phantom{=}+
		3
		\norm{\MelXhat^{\khat}-\MelXcheck^{\khat}}_{\IL^2(\rmv)}^2\mathds{1}_{\khat\in\cK[>\ko]}
		+3\norm{\MelXcheck^{\khat}-\MelXcheck^{\ko}}_{\IL^2(\rmv)}^2\mathds{1}_{\khat\in\cK[<\ko]}\\
		&\leq
		15\norm{\MelXhat^{\ko}-\MelXcheck^{\ko}}_{\IL^2(\rmv)}^2+15\norm{\MelX-\MelX^{\ko}}_{\IL^2(\rmv)}^2+ 24\sigmaYhat\pen_{\ko}^{\rmvhat}\mathds{1}_{\evO}\\&\phantom{=}+15\norm{\MelXcheck^{\kn}-\MelX^{k_n}}_{\IL^2(\rmv)}^2
		+12\max_{k\in\cK[+]}\big\{\big(\norm{\MelXhat^{k}-\MelXcheck^{k}}_{\IL^2(\rmv)}^2-
		\tfrac{\sigma_{Y}^2}{4}\pen_{k}^{\rmvhat}\big)_+\big\}\\&+3\norm{\MelXhat^{\kn}-\MelXcheck^{\kn}}_{\IL^2(\rmv)}^2
		(\mathds{1}_{\evO^c}+\mathds{1}_{\evA^c}) +6\norm{\MelX}_{\IL^2(\rmv)}^2\mathds{1}_{\evO^c}
	\end{align*}
	and completes the proof.
\end{proof}
\begin{proof}[\cref{co:riskboundadaptive}]
We start the proof with taking the expectation on both sides of the
upper bound given in \cref{re:kg} and making use of the concentration inequality in
\cref{lem:conc2}, which  for each $\ko\in\nset{k_n}$
leads to (similar to the proof of \cref{lem: 1})
\begin{align*}
  \IE\left[\norm{\MelXhat^{\hat{k}}-\MelX}^2_{\IL^2(\rmv)}\right] 
          & \leq
15\frac{1}{n}\IE\left[\norm{\mathds{1}_{[-\ko,\ko]}\mathbb{V}_Y\MelUhatrez\mathds{1}_{\mathfrak{M}}}^2_{\IL^2(\rmv)}\right]+15\norm{\mathds{1}_{[-\ko,\ko]^C}\MelX}_{\IL^2(\rmv)}^2\\&\phantom{=}+
24\IE\left[\sigmaYhat\pen_{\ko}^{\rmvhat}\mathds{1}_{\mho_{\ko}}\right]+15\IE\left[\norm{\mathds{1}_{\mathfrak{M}^C}\MelX^{k_n}}_{\IL^2(\rmv)}^2\right]
\\&\phantom{=}+15\IE\left[\norm{\MelUhatrez\mathds{1}_{\mathfrak{M}}(\MelU-\MelUhat)\MelX^{k_n}}_{\IL^2(\rmv)}^2\right]
+ \Const\cdot\eta_Y(1\vee\eta_U^2k_Y\Delta_{k_Y}^{\rmv})\cdot
          n^{-1}\\&\phantom{=}+3\IE\left[\norm{\mathds{1}_{[-\kn,\kn]} (\MelYhat-\MelY)\MelUhatrez\mathds{1}_{\mathfrak{M}}}_{\IL^2(\rmv)}^2
(\mathds{1}_{\mho_{\ko}^C}+\mathds{1}_{\evA^C})\right] \\&\phantom{=}+6\norm{\MelX}_{\IL^2(\rmv)}^2\Pz(\mho_{\ko}^C).
\end{align*}
The first, fourth and fith term in the last upper bound we
estimate with the help of \cref{lem:2} (line by line as in
\cref{subsec:3.3} and using
$\norm{\MelUrez\mathds{1}_{[-\ko,\ko]}}^2_{\IL^2(\rmv)}\leq \Delta_{\ko}^{\rmv_U}\delta_{\ko}^{\rmv_U}\ko$ as well as
the definition of $\eta_U$), which implies
\begin{align*}
  \IE\left[\norm{\MelXhat^{\hat{k}}-\MelX}^2_{\IL^2(\rmv)}\right] 
          &  \leq
\Const\eta_U^4\cdot\IE[X_1^{2(c-1)}]\cdot\Delta_{\ko}^{\rmv_U}\delta_{\ko}^{\rmv_U}\ko n^{-1}+15\norm{\mathds{1}_{[-\ko,\ko]^C}\MelX}_{\IL^2(\rmv)}^2\\&\phantom{=}+
24\IE\left[\sigmaYhat\pen_{\ko}^{\rmvhat}\mathds{1}_{\mho_{\ko}}\right]
+\Const\eta_U^4\norm{\MelX\big(1\vee|\MelU|^2m\big)^{-1/2}}_{\IL^2(\rmv)}^2\\&\phantom{=}+ \Const\cdot\eta_Y(1\vee\eta_U^2k_Y\Delta_{k_Y}^{\rmv})\cdot
          n^{-1} \\&\phantom{=}+3\IE\left[\norm{\mathds{1}_{[-\kn,\kn]} (\MelYhat-\MelY)\MelUhatrez\mathds{1}_{\mathfrak{M}}}_{\IL^2(\rmv)}^2
(\mathds{1}_{\mho_{\ko}^C}+\mathds{1}_{\evA^C})\right] \\&\phantom{=}+6\norm{\MelX}_{\IL^2(\rmv)}^2\Pz(\mho_{\ko}^C).
\end{align*}
Since $\Delta_{\ko}^{\rmvhat}\mathds{1}_{\mho_{\ko}}\leq
(9/4)\Delta_{\ko}^{\rmv_U}$ it follows
$\delta_{\ko}^{\rmvhat}\mathds{1}_{\mho_{\ko}}\leq\delta_{\ko}^{\rmv_U}(\log(9/4)/\log(3)+1)$
and hence $\pen_{\ko}^{\rmvhat}\mathds{1}_{\mho_{\ko}}\leq
\Const\Delta_{\ko}^{\rmv_U}\delta_{\ko}^{\rmv_U}\ko n^{-1}$, which
together with $\IE[\sigmaYhat]=\sigma_Y^2$ implies
  \begin{align}\nonumber
  \IE\left[\norm{\MelXhat^{\hat{k}}-\MelX}^2_{\IL^2(\rmv)}\right] 
          &   \leq15\norm{\mathds{1}_{[-\ko,\ko]^C}\MelX}_{\IL^2(\rmv)}^2+
\Const(\eta_U^4\cdot\IE[X_1^{2(c-1)}]+\sigma_Y^2)\cdot\Delta_{\ko}^{\rmv_U}\delta_{\ko}^{\rmv_U}\ko n^{-1}\\\nonumber
&\phantom{=}+\Const\eta_U^4\norm{\MelX\big(1\vee|\MelU|^2m\big)^{-1/2}}_{\IL^2(\rmv)}^2+ \Const\cdot\eta_Y(1\vee\eta_U^2k_Y\Delta_{k_Y}^{\rmv})\cdot
          n^{-1} \\\nonumber&\phantom{=}+3\IE\left[\norm{\mathds{1}_{[-\kn,\kn]} (\MelYhat-\MelY)\MelUhatrez\mathds{1}_{\mathfrak{M}}}_{\IL^2(\rmv)}^2
(\mathds{1}_{\mho_{\ko}^C}+\mathds{1}_{\evA^C})\right]\\\label{re:kg:e:co:p1}&\phantom{=} +6\norm{\MelX}_{\IL^2(\rmv)}^2\Pz(\mho_{\ko}^C).
\end{align}
Since $\norm{\MelUhatrez\mathds{1}_{\mathfrak{M}}}_{\IL^\infty}^2\leq m$ and
$\norm{\mathds{1}_{[-\kn,\kn]}}_{\IL^2(\rmv)}^2\leq 2
\kn\Delta_{\kn}^{\rmv}\leq 2 n\Delta_{1}^{\rmv}$ due to the definition \eqref{def:kn} of $\kn$ we obtain
\begin{align*}
\IE\left[\norm{\mathds{1}_{[-\kn,\kn]}
    (\MelYhat-\MelY)\MelUhatrez\mathds{1}_{\mathfrak{M}}}_{\IL^2(\rmv)}^2\mathds{1}_{\mho_{\ko}^C}\right]
  &\leq\IE\left[\norm{\mathds{1}_{[-\kn,\kn]}
    (\MelYhat-\MelY)}_{\IL^2(\rmv)}^2\right] m\Pz(\mho_{\ko}^C)\\&\leq
\sigma_Y^2 n^{-1}\norm{\mathds{1}_{[-\kn,\kn]}}_{\IL^2(\rmv)}^2m\Pz(\mho_{\ko}^C)\leq2\sigma_Y^2\Delta_{1}^{\rmv}m\Pz(\mho_{\ko}^C).
\end{align*}
Moreover, for each $t\in\Rz$ we have
\begin{equation*}
\IE\left[|\MelYhat(t)-\MelY(t)|^2\mathds{1}_{\evA^C}\right]\leq
\Const n^{-1} (\IE(Y_1^{4(c-1)})^{1/2}(\Pz(\evA^C)^{1/2}\leq \Const n^{-1} \eta_Y^{1/2}(\Pz(\evA^C)^{1/2}
\end{equation*}
 and hence by
exploiting  $\norm{\MelUhatrez\mathds{1}_{\mathfrak{M}}}_{\IL^\infty}^2\leq n$ and
$\norm{\mathds{1}_{[-\kn,\kn]}}_{\IL^2(\rmv)}^2\leq 2
n\Delta_{1}^{\rmv}$ we obtain
\begin{equation*}
\IE\left[\norm{\mathds{1}_{[-\kn,\kn]}
    (\MelYhat-\MelY)\MelUhatrez\mathds{1}_{\mathfrak{M}}}_{\IL^2(\rmv)}^2\mathds{1}_{\evA^C}\right]
\leq 
\Const\eta_Y^{1/2}\Delta_{1}^{\rmv}n(\Pz(\evA^C)^{1/2}.
\end{equation*}
Since $\mathfrak{A}^C \subseteq	\left\{	|\hat{\sigma}_Y-\sigma_{Y}| >
  {\sigma_{Y}}/{2}\right\}$ and $\sigma_{Y}\geq1$ by Markov's
inequality we have 
$\Pz(\evA^C)\leq \Const \IE[Y_1^{16(c-1)}]n^{-4}\leq \Const \eta_Y n^{-4}$. Combining the last
bounds we conclude 
\begin{equation*}
\IE\left[\norm{\mathds{1}_{[-\kn,\kn]}
    (\MelYhat-\MelY)\MelUhatrez\mathds{1}_{\mathfrak{M}}}_{\IL^2(\rmv)}^2(\mathds{1}_{\mho_{\ko}^C}+\mathds{1}_{\evA^C})\right]
\leq 2\sigma_Y^2\Delta_{1}^{\rmv}m\Pz(\mho_{\ko}^C)+
\Const \eta_Y\Delta_{1}^{\rmv}n^{-1}
\end{equation*}
which together with \eqref{re:kg:e:co:p1} and the definition of $\eta_X$ implies
\begin{align*}
  \IE\left[\norm{\MelXhat^{\hat{k}}-\MelX}^2_{\IL^2(\rmv)}\right]
                 &\leq15\cdot\norm{\mathds{1}_{[-\ko,\ko]^C}\MelX}_{\IL^2(\rmv)}^2+
  \Const\cdot(\eta_U^4\eta_X+\sigma_Y^2)\cdot\Delta_{\ko}^{\rmv_U}\delta_{\ko}^{\rmv_U}\ko n^{-1}\\\hfill
  &\phantom{=}+\Const\cdot\eta_U^4\norm{\MelX\big(1\vee|\MelU|^2m\big)^{-1/2}}_{\IL^2(\rmv)}^2+
  \Const\cdot\eta_Y(1\vee\eta_U^2k_Y\Delta_{k_Y}^{\rmv})\cdot
  n^{-1} \\&\phantom{=}+6(\sigma_Y^2\Delta_{1}^{\rmv}+\eta_X)m\Pz(\mho_{\ko}^C).
\end{align*}
Since the last bound is valid for all $\ko\in\nset{\kn}$ we
immediately obtain the claim \eqref{re:kg:e:co}, which completes the
proof.
\end{proof}
\begin{proof}[\cref{lem:8maxi}]
  The proof follows along the lines of the proof of Theorem 4.1 in
  \cite{NeumannReiss2009}. In order to show the claim, we need some
  definitions and notations. For two functions $l,u:\IR\to\IR$ with
  $l\leq u$
  introduce the bracket
  \begin{equation*}
    [l,u]:=\{f:\IR\to\IR:\, l\leq f \leq u\}.
  \end{equation*}
  For a set of functions $\mathscr{G}$ and $\varepsilon\in\pRz$ we denote by
  $\operatorname{N}_{[\cdot]}(\varepsilon,\mathscr{G})$ the minimum number of
  brackets $[l_i,u_i]$, satisfying
  $\IE[(u_i(Z_1)-l_i(Z_1))^2]\leq\varepsilon^2$, that are needed to cover
  $\mathscr{G}$. The associated bracking entropy integral is defined
  as
  \begin{equation*}
    \operatorname{J}_{[\cdot]}(\delta,\mathscr{G}):=  \int_{(0,\delta)}\sqrt{\log\operatorname{N}_{[\cdot]}(\varepsilon,\mathscr{G})}\rmd\lambda(\varepsilon),\quad\forall\delta\in\pRz.
  \end{equation*}
  Further, a function $\bar{f}$ is called an envelope of
  $\mathscr{G}$, if $|f|\leq \bar{f}$ for all
  $f\in\mathscr{G}$. Analogously as presented in
  \cite{NeumannReiss2009}, we aim to apply Lemma 19.34 and Corollary
  19.35 from \cite{Vaart1998}. Hence, we start by decomposing $c_m$
  into its real and imaginary part, namely
  \begin{equation*}
    \mathfrak{Re}(c_m(t)) =
    \frac{1}{\sqrt{m}}\sum_{j\in\nset{m}}\left\{Z_{j}^\beta\cos(\log(2\pi
      tZ_j))-\IE\left[Z_j^\beta\cos(2\pi
        t\log(Z_j))\right]\right\}
  \end{equation*}
  and
  \begin{equation*}
    \mathfrak{Im}(c_m(t)) =
    \frac{1}{\sqrt{m}}\sum_{j\in\nset{m}}\left\{Z_{j}^\beta\sin(\log(2\pi
      tZ_j))-\IE\left[Z_j^\beta\sin(2\pi t\log(Z_j))\right]\right\},
  \end{equation*}
  such that $c_m(t) = \mathfrak{Re}(c_m(t)) +\iota\cdot \mathfrak{Im}(c_m(t))$ for all $t\in\IR$.
  Therefore, define the following class of functions 
  \begin{equation*}
    \mathscr{G}_\beta:= \left\{ \pRz\ni z\mapsto
      \rmwbar(t)z^\beta\cos(2\pi
      t\log(z))\in\IR:t\in\IR\right\} \cup\left\{\pRz\ni z\mapsto
      \rmwbar(t)z^\beta\sin(2\pi t\log(z))\in\IR:t\in\IR\right\},
  \end{equation*}
	whose envelope is given by $f(z):=z^\beta$. Now applying Lemma 19.34 of \cite{Vaart1998} and following the argumentation of the proof of Corollary 19.35 within, we conclude 
	\begin{equation*}
		\IE\left[\norm{c_m}_{\IL^\infty(\rmwbar)}\right] \leq \Const\left(\sqrt{\IE[Z_1^{2\beta}]}+\operatorname{J}_{[\cdot]}(\sqrt{\IE[Z_1^{2\beta}]},\mathscr{G}_\beta)\right)\leq\Const\left(\eta+\operatorname{J}_{[\cdot]}(\eta,\mathscr{G}_\beta)\right).
	\end{equation*}
	As $\eta\in\pRz$, it suffices to show that the entropy integral is finite. Hence, inspired by \cite{Yukich1985}, we set
	\begin{equation}\label{lem:8maxi:p1}
		B_\varepsilon:=\inf\left\{b\in\pRz:\E[Z_1^{2\beta}\mathds{1}_{\{|\log(Z_1)|>b\}}]\leq\varepsilon^2\right\}\leq \left(\frac{\E[Z_1^{2\beta}|\log(Z_1)|^\gamma]}{\varepsilon^2}\right)^{\frac{1}{\gamma}}\leq \left(\frac{\eta}{\varepsilon^2}\right)^{\frac{1}{\gamma}}
        \end{equation}
        due to  the generalised Markov inequality.
	Furthermore, for grid points $t_j\in\Rz$ specified below, we define 
	\begin{equation*}
		g^{\pm}_j(z):=\left(\rmwbar(t_j)z^\beta\cos(2\pi t_j\log(z))\pm \varepsilon z^\beta\right)\mathds{1}_{[0,B_\varepsilon]}(|\log(z)|)\pm\norm{\rmwbar}_\infty z^\beta\mathds{1}_{(B_\varepsilon,\infty)}(|\log(z)|),
	\end{equation*}
	as well as 
\begin{equation*}
	h^{\pm}_j(z):=\left(\rmwbar(t_j)z^\beta\sin(2\pi t_j\log(z))\pm \varepsilon z^\beta\right)\mathds{1}_{[0,B_\varepsilon]}(|\log(z)|)\pm\norm{\rmwbar}_\infty z^\beta\mathds{1}_{(B_\varepsilon,\infty)}(|\log(z)|).
\end{equation*}
	We obtain
	\begin{align*}
		\E\left[\left(g^{+}_j(Z_1)-g^{-}_j(Z_1)\right)^2\right]
                &\leq4\varepsilon^2\E\left[Z_1^{2\beta}\mathds{1}_{[0,B_\varepsilon]}(|\log(Z_1)|)\right]+4\norm{\rmwbar}_\infty^2\E\left[Z_1^{2\beta}\mathds{1}_{\{|\log(Z_1)|>B_\varepsilon\}}\right]\\
                &\leq4\varepsilon^2\left(\E[Z_1^{2\beta}]+\norm{\rmwbar}_\infty^2\right).
	\end{align*}
	and analogously
	\begin{equation*}
		\E\left[\left(h^{+}_j(Z_1)-h^{-}_j(Z_1)\right)^2\right] \leq 4\varepsilon^2\left(\E[Z_1^{2\beta}]+\norm{\rmwbar}_\infty^2\right)
	\end{equation*}
	It remains to choose the grid points $t_j$ in such a way that
        the brackets cover the set $\mathscr{G}_\beta$. Let $t\in\IR$ be
        arbitrarily chosen and take some arbitrary grid point
        $t_j$. Then with the Lipschitz constant $L_{\rmwbar}$ of the
        density function $\rmwbar$, we have for $z\in\pRz$ evidently
        $|\rmwbar(t)z^\beta\cos(2\pi
        t\log(z))-\rmwbar(t_j)z^\beta\cos(2\pi t_j\log(z))|\leq
        z^\beta(\rmwbar(t)+\rmwbar(t_j))$
        and 
	\begin{align*}
		|\rmwbar(t)z^\beta\cos(2\pi
          t\log(z))&-\rmwbar(t_j)z^\beta\cos(2\pi t_j\log(z))|\\
		 &\leq |\rmwbar(t)-\rmwbar(t_j)|\cdot
                   z^\beta\cdot|\cos(2\pi t\log(z))|
                   \\&\phantom{=}+ |\rmwbar(t_j)|\cdot z^\beta\cdot |\cos(2\pi
                   t_j\log(z))-\cos(2\pi t\log(z))|\\
		 &\leq L_{\rmwbar}|t-t_j|\cdot z^\beta + 2\pi\norm{\rmwbar}_\infty\cdot z^\beta \cdot |t-t_j|\cdot B_\varepsilon.
	\end{align*}
	Hence, the function $\pRz\ni z\mapsto
        \rmwbar(t)z^\beta\cos(2\pi t\log(z))$ is contained in the bracket $[g_j^-,g_j^+]$, if
	\begin{equation}
		\min\left\{|t_j-t|(L_{\rmwbar} +2\pi\norm{\rmwbar}_\infty B_\varepsilon), \rmwbar(t)+\rmwbar(t_j)\right\} \leq \varepsilon.
	\end{equation}
	Thus, for integer $j\in[-J_{\varepsilon},J_{\varepsilon}]$ we are choosing the grid points in the following way:
	\begin{equation*}
		t_j = \frac{j\varepsilon}{(L_{\rmwbar}+2\pi\norm{\rmwbar}_\infty B_\varepsilon)}
	\end{equation*}
 where $J_{\varepsilon}$ is the smallest integer such that $t_{J_{\varepsilon}}$ is greater than or equal to 
	\begin{equation*}
		T_{\varepsilon} := \inf\bigg\{t\in\pRz:\,\sup_{v:|v|\geq t} \rmwbar(v) \leq {\varepsilon}/{2}\bigg\}
              \end{equation*}
        satisfying $\log (T_{\varepsilon})=O(\varepsilon^{-\kappa})$ with $\kappa:=1/(\rho+1/2)$.      
        Evidently, there are at most $2J_{\varepsilon}+1$ of those grid points  and, hence  $\operatorname{N}_{[\cdot]}(\varepsilon,\mathscr{G}) \leq 2(2J(\varepsilon)+1)$. 
        Keeping the bound \eqref{lem:8maxi:p1} in mind we also have $\log
        (B_\varepsilon/\varepsilon)=O(\log(\varepsilon^{-1-2/\gamma}))$ and
        thus from the  inequality
        \begin{equation*}
        J_{\varepsilon} \leq  T_{\varepsilon} (L_{\rmwbar} +2\pi\norm{\rmwbar}_\infty B_\varepsilon) \varepsilon^{-1}+1
      \end{equation*}
      we obtain $\log\operatorname{N}_{[\cdot]}(\varepsilon,\mathscr{G}) = O(\log(J(\varepsilon))) = O(\varepsilon^{-\kappa}+\log(\varepsilon^{-1-2/\gamma}))=O(\varepsilon^{-\kappa})$.
      Since $\kappa<2$ we conclude
      \begin{equation*}
        \int_{(0,\delta)}\sqrt{\log\operatorname{N}_{[\cdot]}(\varepsilon,\mathscr{G}_\beta)}\rmd\lambda(\varepsilon)<\infty,\quad\forall\delta\in\pRz
      \end{equation*}
  which completes the proof.
\end{proof}

\begin{proof}[\cref{bound:event}]For $m\in\nset{m_{\circ}}$
	we trivially have $\Pz(\mho_{k_m}^C)\leq m_{\circ}^2
	m^{-2}$. Therefore, let $m\in\Nz_{\geq  m_{\circ}}$. 
	We decompose $\MelUhat=\MelZhat+\MelUhat^{\mathrm{ub}}$ into a bounded part $\MelZhat$
	and an unbounded part $\MelUhat^{\mathrm{ub}}$ (similiar to
	\cref{rem:conc}). Precisely, for $b\in\Rz_{\geq1}$ (to be specified below)
	introduce
	the  random variables
	\begin{equation*}
		Z_j:=U_j\mathds{1}_{(0,b]}(U_j^{c-1})+\mathds{1}_{(b,\infty)}(U_j^{c-1}),\quad\forall
		j\in\nset{m}
	\end{equation*}
	and form accordinglgy for each $t\in\Rz$ the empirical Mellin transform
	\begin{align*}
		\MelZhat(t) &:=
		\frac{1}{m}\sum_{j\in\nset{m}}Z_j^{c-1+\iota2\pi t}=
		\frac{1}{m}\sum_{j\in\nset{m}}\{U_j^{c-1}\exp(\iota2\pi
                t\log(U_j))\mathds{1}_{(0,b]}(U_j^{c-1})+\mathds{1}_{(b,\infty)}(U_j^{c-1})\}.
	\end{align*}
	Evidently, the unbounded part
	$\MelUhat^{\mathrm{ub}}:=\MelUhat-\MelZhat$ satisfies
	\begin{equation*}
		\MelUhat^{\mathrm{ub}}(t) =
		\frac{1}{m}\sum_{j\in\nset{m}}(U_j^{c-1+\iota2\pi t}-1)\mathds{1}_{(b,\infty)}(U_j^{c-1}),\quad\forall
		t\in\Rz.
	\end{equation*}
	Exploiting the decomposition we obtain the elementary bound
	\begin{multline}\label{bound:event:p1}
		\Pz(\mho_{k_m}^C)\leq \IP\left(\exists t\in[-k_m,k_m]:\,|\MelZhat(t)-\IE[\MelZhat(t)]|>\frac{1}{6}|\MelU(t)| \right) \\+ \IP\left(\exists t\in[-k_m,k_m]:\, |\MelUhat^{\mathrm{ub}}(t)-\IE[\MelUhat^{\mathrm{ub}}(t)]|>\frac{1}{6}|\MelU(t)| \right)
	\end{multline}
	where we estimate separately the two terms of the  bound starting with the first one. Evidently, multiplying with the
	density function $\rmwbar$ given in \eqref{lem:8maxi:e2} and
	making use of the definition of $k_m$ given in \eqref{bound:event:km} we have 
	\begin{multline*}
		\IP\left(\exists
		t\in[-k_m,k_m]:\,|\MelZhat(t)-\IE[\MelZhat(t)]|>\frac{1}{6}|\MelU(t)|
		\right) \\\leq  \IP\left(\exists
		t\in[-k_m,k_m]:\,\rmwbar(t)|\MelZhat(t)-\IE[\MelZhat(t)]|>\tfrac{1}{6}\inf_{s\in[-k_m,k_m]}\hspace*{-3ex}\rmwbar(s)|\MelU(s)|
		\right)\\\hfill\leq \IP\left(\exists
		t\in[-k_m,k_m]:\,\rmwbar(t)|\MelZhat(t)-\IE[\MelZhat(t)]|>\tau_mm^{-1/2}\right).
	\end{multline*}
	By continuity of $t\mapsto\rmwbar(t)(\MelZhat(t)-\IE[\MelZhat(t)])$ we obtain 
	\begin{equation*}
		\IP\left(\exists
		t\in[-k_m,k_m]:\,\rmwbar(t)|\MelZhat(t)-\IE[\MelZhat(t)]|>\tau_mm^{-1/2}\right)\leq 
		\Pz\left(\sup_{t\in\mathcal T_m}|\overline{\nu}_{t}|>\tau_mm^{-1/2}\right)
	\end{equation*}
	setting $\mathcal T_m:=[-k_m,k_m]\cap\IQ$ and
	$\overline{\nu}_{t}=m^{-1}\sum_{i\in\nset{m}}\left\{\nu_{t}(Z_i)-\E\left(\nu_{t}(Z_i)\right)\right\}$
	with $\nu_{t}(z):=z^{c-1+\iota2\pi t}\rmwbar(t)$ for $z\in\Rz$. Observe
	that each $Z_i$ takes values in $\mathcal Z:=\{z\in\pRz:z^{c-1}\in(0,b]\}$ only. 
	Since $\mathcal T_m\subset\Rz$ is countable we eventually apply Talagrand's inequality given in Lemma
	\cref{lem:Talagrand:2} where we need to determine the three
	quantities $\psi$, $\Psi$ and $\tau$. Consider
	$\psi$ first.   Evidently, since $\norm{\rmwbar}_{\infty}=1$ we have
	\begin{equation*} 
		\sup_{t\in\Rz}\left\{\sup_{z\in\mathcal Z}|\nu_{t}(z)|\right\}\leq  b\norm{\rmwbar}_{\infty}=b=:\psi.
	\end{equation*} 
	Consider next $\tau$. Making use of
        \begin{equation*}
          \E[Z_1^{2(c-1)}]=\E[U_1^{2(c-1)}\mathds{1}_{(0,b]}(U_1^{c-1})+\mathds{1}_{(b,\infty)}(U_1^{c-1})]\leq
          \E[U_1^{2(c-1)}]\leq\eta_U^2
      \end{equation*}
 (keep in mind that $b\in\Rz_{\geq1}$
	and $\norm{\rmwbar}_{\infty}^2=1$) for each $t\in\IR$ we have
	\begin{equation*}
		\IE\big[|\nu_{t}(Z_1)|^2]\leq
		\norm{\rmwbar}_{\infty}^2\IE\big[Z_1^{2(c-1)}\big]\leq \eta_U^2=:\tau.
	\end{equation*}
	Finally, consider $\Psi$. Recalling the normalised Mellin function
	process $c_m$ defined in \eqref{lem:8maxi:e1} (with $\beta=c-1$) for all
	$t\in\Rz$ we have
	$m^{-1/2}\rmwbar(t)c_m(t)=\overline{\nu}_{t}$. Since $(\E[Z_1^{2(c-1)}])^{1/2}\leq (\E[U_1^{2(c-1)}])^{1/2}\leq \eta_U$ and 
	\begin{equation*}
		\E[Z_1^{2(c-1)}|\log(Z_1)|^\gamma]=\E[U_1^{2(c-1)}|\log(U_1)|^\gamma]\mathds{1}_{(0,b]}(U_j^{c-1})]\leq
		\E[U_1^{2(c-1)}|\log(U_1)|^\gamma]\leq \eta_U
	\end{equation*}
	from \cref{lem:8maxi} it follows
	\begin{equation*}
		\E\left[\sup_{t\in \mathcal T_m}|\overline{\nu}_{t}|\right]=m^{-1/2}\E\left[\sup_{t\in \mathcal T_m}\rmwbar(t)|c_m(t)|\right]
		\leq m^{-1/2}\IE\left[\norm{c_m}_{\IL^\infty(\rmwbar)}\right]\leq m^{-1/2}\Const(\eta_U,\rho)=:\Psi.
	\end{equation*}
	Due to \cref{lem:Talagrand:2} with
	$\kappa_m:=(\tau_m-2\Const(\eta_U,\rho))m^{-1/2}\in\pRz$ we obtain
	\begin{align*}
		\IP\left(\exists
		t\in[-k_m,k_m]:\,|\MelZhat(t)-\IE[\MelZhat(t)]|>\frac{1}{6}|\MelU(t)|
          \right)
          \hspace*{-50ex}&\\&\leq
		\Pz\left(\sup_{t\in\mathcal T_m}|\overline{\nu}_{t}|>\tau_mm^{-1/2}\right)
		=  \Pz\left(\sup_{t\in\mathcal T_m}|\overline{\nu}_{t}|>m^{-1/2}2\Const(\eta_U,\rho) +\kappa_m\right)\\&\leq 3
		\exp\left(-\Const_{\mathrm{tal}}
		\big(\frac{(\tau_m-2\Const(\eta_U,\rho))^2}{\eta_U^2}\wedge\frac{m^{1/2}(\tau_m-2\Const(\eta_U,\rho))}{b}\big)\right).
	\end{align*}
	Setting
	$b=2m^{1/2}\eta_U^2/(\tau_m-2\Const(\eta_U,\rho))=m^{1/2}(\log
	m)^{-1/2}\eta_U\Const_{\mathrm{tal}}^{1/2}\in\Rz_{\geq1}$ (keep \eqref{bound:event:mo} in mind) it  follows
	\begin{equation}\label{bound:event:p2}
		\IP\left(\exists
		t\in[-k_m,k_m]:\,|\MelZhat(t)-\IE[\MelZhat(t)]|>\frac{1}{6}|\MelU(t)|
		\right)\leq 3\exp\left(-\Const_{\mathrm{tal}}
		\frac{(\tau_m-2\Const(\eta_U,\rho))^2}{2\eta_U^2}\right) = 3 m^{-2}.
	\end{equation}
	Consider the second term  on the right hand side of \eqref{bound:event:p1}. For
	each $t\in\Rz$ (keep $b\in\IR_{\geq1}$ in mind) we have
	\begin{align*}
		|\MelUhat^{\mathrm{ub}}(t)]|&\leq\frac{1}{m}\sum_{j\in\nset{m}} |U_j^{c-1+\iota2\pi t}-1|\mathds{1}_{(b,\infty)}(U_j^{c-1})
		\\&\leq\frac{1}{m}\sum_{j\in\nset{m}} (U_j^{c-1}+1)\mathds{1}_{(b,\infty)}(U_j^{c-1})\leq\frac{2}{m}\sum_{j\in\nset{m}} 
		U_j^{c-1}\mathds{1}_{(b,\infty)}(U_j^{c-1})
	\end{align*}
	and  hence   $|\IE[\MelUhat^{\mathrm{ub}}(t)]|\leq 2
	b^{-1}\IE[U_1^{2(c-1)}]\leq 2
	b^{-1}\eta_U^2=m^{-1/2}(\tau_m-2\Const(\eta_U,\rho))$.
	Multiplying the
	density function $\rmwbar:\Rz\to(0,1]$ given in \eqref{lem:8maxi:e2} and
	making use of the definition \eqref{bound:event:km} of $k_m$ 
	 it follows (with $\Const(\eta_U,\rho)\in\IR_{\geq1}$)
	\begin{align*}
		\IP\bigg(\exists & t\in[-k_m,k_m]:\,
		|\MelUhat^{\mathrm{ub}}(t)
		-\IE[\MelUhat^{\mathrm{ub}}(t)]|>\frac{1}{6}|\MelU(t)|\bigg)&\\
		&\leq \IP\Bigg(\exists t\in[-k_m,k_m]:\, |\MelUhat^{\mathrm{ub}}(t)| >\frac{1}{6}\rmwbar(t)|\MelU(t)|-|\IE[\MelUhat^{\mathrm{ub}}(t)]| \Bigg)\\ 
		&\leq
		\IP\Bigg(\frac{2}{m}\sum_{j\in\nset{m}}U_j^{c-1}\mathds{1}_{(b,\infty)}(U_j^{c-1})>\frac{1}{6}\inf_{v\in[-k_m,k_m]}\rmwbar(v)|\MelU(v)|-m^{-1/2}(\tau_m-2\Const(\eta_U,\rho))\Bigg)\\
		&\leq
		\IP\left(\sum_{j\in\nset{m}}U_j^{c-1}\mathds{1}_{(b,\infty)}(U_j^{c-1})>\Const(\eta_U,\rho)m^{1/2}\right)\\&\leq
		m^{1/2}\IE\left[U_1^{c-1}\mathds{1}_{(b,\infty)}(U_1^{c-1})\right]
		\leq m^{1/2}b^{-6}\IE[U_1^{7(c-1)}]
		\\&\leq\Const_{\mathrm{tal}}^{-3}\eta_U  m^{-1/2}(\log m)^{3}   m^{-2}
		\leq    (6/e)^3\Const_{\mathrm{tal}}^{-3}\eta_U
           m^{-2}\leq 11\Const_{\mathrm{tal}}^{-3}\eta_U m^{-2}.
	\end{align*}
	Combining the last upper bound, the upper bound \eqref{bound:event:p2} and
	the decomposition \eqref{bound:event:p1} we obtain the claim
	 which
	completes the proof. 
\end{proof}

	\vspace{2cm}
	\centering{\Large{\scshape\centering{{Acknowledgement}}}}\\
	\normalsize{This work is funded by Deutsche Forschungsgemeinschaft (DFG, German Research Foundation) under Germany’s Excellence Strategy EXC-2181/1-39090098 (the Heidelberg STRUCTURES Cluster of Excellence)}
	\bibliography{MUDUNER.bib}
\end{document}